\documentclass[a4paper, 12pt]{amsart}

\usepackage{amsmath,amssymb,enumitem,verbatim,stmaryrd,xcolor,microtype,graphicx,aliascnt,fancyvrb}

\usepackage[normalem]{ulem}
\usepackage[english]{babel} 
\usepackage[top=3.5cm,bottom=3.5cm,left=3.2cm,right=3.2cm]{geometry}
\usepackage[bookmarksdepth=2,linktoc=page,colorlinks,linkcolor={red!80!black},citecolor={red!80!black},urlcolor={blue!80!black},pdftitle={Quiver matroids},pdfauthor={Manoel Jarra, Oliver Lorscheid and Eduardo Vital}]{hyperref}
\usepackage{float}

\usepackage{tikz}\usetikzlibrary{matrix,arrows,decorations.markings}
\usepackage{tikz-cd}

\allowdisplaybreaks 

\usepackage{mathptmx}                                        
\usepackage{etoolbox}\makeatletter\patchcmd{\@startsection}{\@afterindenttrue}{\@afterindentfalse}{}{}\makeatother    
\patchcmd{\section}{\scshape}{\bfseries}{}{}\makeatletter\renewcommand{\@secnumfont}{\bfseries}\makeatother           
\usepackage[backgroundcolor=orange!30!white,linecolor=orange!80!white,textsize=footnotesize]{todonotes} \makeatletter \providecommand \@dotsep{5} \def\listtodoname{List of Todos} \def\listoftodos{\@starttoc{tdo}\listtodoname} \makeatother 

\addto\extrasenglish{}  
\addto\extrasenglish{}  
\theoremstyle{plain}
\newtheorem{thm}{Theorem}[section] 
\newaliascnt{lemma}{thm}\newtheorem{lemma}[lemma]{Lemma}\aliascntresetthe{lemma}
\newaliascnt{cor}{thm}\newtheorem{cor}[cor]{Corollary}\aliascntresetthe{cor}
\newaliascnt{prop}{thm}\newtheorem{prop}[prop]{Proposition}\aliascntresetthe{prop}
\newtheorem{thmA}{Theorem} 
\newaliascnt{propA}{thmA}\newtheorem{propA}[propA]{Proposition}\aliascntresetthe{propA}
\newaliascnt{conjA}{thmA}\newtheorem{conjA}[conjA]{Conjecture}\aliascntresetthe{conjA}

\newtheorem*{thm*}{Theorem}
\newtheorem*{lem*}{Lemma}
\newtheorem*{cor*}{Corollary}
\newtheorem*{problem*}{Problem}

\theoremstyle{definition}
\newaliascnt{df}{thm}\newtheorem{df}[df]{Definition}\aliascntresetthe{df}
\newaliascnt{rem}{thm}\newtheorem{rem}[rem]{Remark}\aliascntresetthe{rem}
\newaliascnt{ex}{thm}\newtheorem{ex}[ex]{Example}\aliascntresetthe{ex}
\newaliascnt{notation}{thm}\newtheorem{notation}[notation]{Notation}\aliascntresetthe{notation}

\newtheorem*{df*}{Definition}
\newtheorem*{ex*}{Example}
\newtheorem*{rem*}{Remark}
\newtheorem*{warning*}{Warning}
\newtheorem*{notation*}{Notation}

\usepackage{etoolbox}
\makeatletter
\patchcmd{\@startsection}{\@afterindenttrue}{\@afterindentfalse}{}{}             
\patchcmd{\part}{\bfseries}{\bfseries\LARGE}{}{}
\patchcmd{\section}{\scshape}{\bfseries}{}{}\renewcommand{\@secnumfont}{\bfseries} 
\patchcmd{\@settitle}{\uppercasenonmath\@title}{\large}{}{}
\patchcmd{\@setauthors}{\MakeUppercase}{}{}{}
\addto{\captionsenglish}{} 
\addto{\captionsenglish}{} 
\addto{\captionsenglish}{} 
\makeatother

\usepackage{fancyhdr}

\DeclareRobustCommand{\gobblefour}[5]{}    

\DeclareFontFamily{OT1}{pzc}{}                                
\DeclareFontShape{OT1}{pzc}{m}{it}{<-> s * [1.10] pzcmi7t}{}
\DeclareMathAlphabet{\mathpzc}{OT1}{pzc}{m}{it}
\DeclareSymbolFont{sfoperators}{OT1}{bch}{m}{n} \DeclareSymbolFontAlphabet{\mathsf}{sfoperators} \makeatletter\def\operator@font{\mathgroup\symsfoperators}\makeatother 
\DeclareSymbolFont{cmletters}{OML}{cmm}{m}{it}              
\DeclareSymbolFont{cmsymbols}{OMS}{cmsy}{m}{n}
\DeclareSymbolFont{cmlargesymbols}{OMX}{cmex}{m}{n}
\DeclareMathSymbol{\myjmath}{\mathord}{cmletters}{"7C}     \let\jmath\myjmath 
\DeclareMathSymbol{\myamalg}{\mathbin}{cmsymbols}{"71}     
\DeclareMathSymbol{\mycoprod}{\mathop}{cmlargesymbols}{"60}
\DeclareMathSymbol{\myalpha}{\mathord}{cmletters}{"0B}     \let\alpha\myalpha 
\DeclareMathSymbol{\mybeta}{\mathord}{cmletters}{"0C}      \let\beta\mybeta
\DeclareMathSymbol{\mygamma}{\mathord}{cmletters}{"0D}     \let\gamma\mygamma
\DeclareMathSymbol{\mydelta}{\mathord}{cmletters}{"0E}     \let\delta\mydelta
\DeclareMathSymbol{\myepsilon}{\mathord}{cmletters}{"0F}   \let\epsilon\myepsilon
\DeclareMathSymbol{\myzeta}{\mathord}{cmletters}{"10}      \let\zeta\myzeta
\DeclareMathSymbol{\myeta}{\mathord}{cmletters}{"11}       \let\eta\myeta
\DeclareMathSymbol{\mytheta}{\mathord}{cmletters}{"12}     \let\theta\mytheta
\DeclareMathSymbol{\myiota}{\mathord}{cmletters}{"13}      \let\iota\myiota
\DeclareMathSymbol{\mykappa}{\mathord}{cmletters}{"14}     \let\kappa\mykappa
\DeclareMathSymbol{\mylambda}{\mathord}{cmletters}{"15}    \let\lambda\mylambda
\DeclareMathSymbol{\mymu}{\mathord}{cmletters}{"16}        \let\mu\mymu
\DeclareMathSymbol{\mynu}{\mathord}{cmletters}{"17}        \let\nu\mynu
\DeclareMathSymbol{\myxi}{\mathord}{cmletters}{"18}        \let\xi\myxi
\DeclareMathSymbol{\mypi}{\mathord}{cmletters}{"19}        \let\pi\mypi
\DeclareMathSymbol{\myrho}{\mathord}{cmletters}{"1A}       \let\rho\myrho
\DeclareMathSymbol{\mysigma}{\mathord}{cmletters}{"1B}     \let\sigma\mysigma
\DeclareMathSymbol{\mytau}{\mathord}{cmletters}{"1C}       \let\tau\mytau
\DeclareMathSymbol{\myupsilon}{\mathord}{cmletters}{"1D}   \let\upsilon\myupsilon
\DeclareMathSymbol{\myphi}{\mathord}{cmletters}{"1E}       \let\phi\myphi
\DeclareMathSymbol{\mychi}{\mathord}{cmletters}{"1F}       \let\chi\mychi
\DeclareMathSymbol{\mypsi}{\mathord}{cmletters}{"20}       \let\psi\mypsi
\DeclareMathSymbol{\myomega}{\mathord}{cmletters}{"21}     \let\omega\myomega
\DeclareMathSymbol{\myvarepsilon}{\mathord}{cmletters}{"22}\let\varepsilon\myvarepsilon
\DeclareMathSymbol{\myvartheta}{\mathord}{cmletters}{"23}  \let\vartheta\myvartheta
\DeclareMathSymbol{\myvarpi}{\mathord}{cmletters}{"24}     \let\varpi\myvarpi
\DeclareMathSymbol{\myvarrho}{\mathord}{cmletters}{"25}    \let\varrho\myvarrho
\DeclareMathSymbol{\myvarsigma}{\mathord}{cmletters}{"26}  \let\varsigma\myvarsigma
\DeclareMathSymbol{\myvarphi}{\mathord}{cmletters}{"27}    \let\varphi\myvarphi

\DeclareMathOperator{\rep}{Rep}
\DeclareMathOperator{\res}{res}

\DeclareMathOperator{\Spec}{Spec}

\DeclareMathOperator{\Bands}{Bands}
\DeclareMathOperator{\BandSpaces}{BandSpaces}
\DeclareMathOperator{\Rings}{Rings}
\DeclareMathOperator{\BSch}{BSch}
\DeclareMathOperator{\BAff}{BAff}
\DeclareMathOperator{\Hom}{Hom}

\DeclareMathOperator{\Proj}{Proj}
\DeclareMathOperator{\Mat}{Mat}
\DeclareMathOperator{\Fl}{Fl}

\DeclareMathOperator{\Gr}{Gr}

\DeclareMathOperator{\colim}{colim}

\DeclareMathOperator{\im}{im}

\DeclareMathOperator{\Mod}{{Mod}}

\DeclareMathOperator{\Top}{{Top}}

\DeclareMathOperator{\OBlpr}{{OBlpr}}

\DeclareMathOperator{\Idylls}{{Idylls}}
\DeclareMathOperator{\Vect}{{Vect}}
\DeclareMathOperator{\Sets}{Sets}

\DeclareMathOperator*{\hypersum}{\,\raisebox{-2.2pt}{\larger[2]{$\boxplus$}}\,}
\DeclareMathOperator{\sign}{{sign}}

\DeclareMathOperator{\pl}{{pl}}
\DeclareMathOperator{\SL}{{SL}}
\DeclareMathOperator{\rk}{{rk}}

\newcommand{\cat}[1]{\mathsf{#1}}
\newcommand\A{{\mathbb A}}

\newcommand\C{{\mathbb C}}
\newcommand\F{{\mathbb F}}
\newcommand\G{{\mathbb G}}
\newcommand\K{{\mathbb K}}

\newcommand\N{{\mathbb N}}
\renewcommand\P{{\mathbb P}}

\newcommand\R{{\mathbb R}}
\newcommand\uA{{\underline A}}
\newcommand\uB{{\underline B}}
\newcommand\uE{{\underline E}}
\newcommand\uM{{\underline M}}
\newcommand\uN{{\underline N}}
\newcommand\uS{{\underline S}}
\newcommand\uT{{\underline T}}
\newcommand\uPhi{{\underline\Phi}}

\renewcommand\S{{\mathbb S}}
\newcommand\T{{\mathbb T}}

\newcommand\Z{{\mathbb Z}}

\newcommand\br{{\mathbf{r}}}

\newcommand\bfx{{\mathbf{x}}}

\newcommand\bfy{{\mathbf{y}}}
\newcommand\bfz{{\mathbf{z}}}

\newcommand\cB{{\mathcal B}}
\newcommand\cC{{\mathcal C}}

\newcommand\cL{{\mathcal L}}
\newcommand\cM{{\mathcal M}}
\newcommand\cN{{\mathcal N}}
\newcommand\cO{{\mathcal O}}

\newcommand\cS{{\mathcal S}}

\newcommand\cU{{\mathcal U}}
\newcommand\cV{{\mathcal V}}

\newcommand\cX{{\mathcal X}}
\newcommand\cY{{\mathcal Y}}

\newcommand\fb{{\mathfrak{b}}}

\newcommand\fm{{\mathfrak m}}
\newcommand\fp{{\mathfrak p}}
\newcommand\fq{{\mathfrak q}}

\newcommand\bx{{\mathbf x}}
\newcommand\by{{\mathbf y}}

\newcommand\Fun{{\F_1}}

\newcommand\Funpm{{\F_1^\pm}}

\renewcommand\int{\textup{int}}

\newcommand\id{\textup{id}}

\newcommand\Tits{\textup{Tits}}

\newcommand\oblpr{\textup{oblpr}}
\newcommand\band{\textup{band}}

\newcommand\univ{\textup{univ}}

\newcommand\qr{\mathpzc{qr}}
\renewcommand\geq{\geqslant}
\renewcommand\leq{\leqslant}
\newcommand{\gen}[1]{\langle #1 \rangle}

\newcommand{\norm}[1]{| #1 |}

\newcommand{\bandquot}[2]{#1\!\sslash\!#2}
\newcommand{\bandgenquot}[2]{#1\!\sslash\!\gen{#2}}
\newcommand{\smallmat}[4]{\big[\begin{smallmatrix} #1 & #2 \\ #3 & #4 \end{smallmatrix}\big]}
\newcommand{\smallvector}[2]{\big[\begin{smallmatrix} #1 \\ #2 \end{smallmatrix}\big]}
\newcommand{\smalltrivector}[3]{\Big[\begin{smallmatrix} #1 \\ #2 \\ #3 \end{smallmatrix}\Big]}

\VerbatimFootnotes   

\title{Quiver matroids \\[10pt] \normalsize Matroid morphisms, quiver Grassmannians,\\ their Euler characteristics and $\Fun$-points}

\author{Manoel Jarra}
\address{\rm Manoel Jarra, University of Groningen, the Netherlands, and IMPA, Rio de Janeiro, Brazil}
\email{{manoel.jarra@impa.br}}

\author{Oliver Lorscheid}
\address{\rm Oliver Lorscheid, University of Groningen, the Netherlands}
\email{{o.lorscheid@rug.nl}}

\author{Eduardo Vital}
\address{\rm Eduardo Vital, Bielefeld University, Germany}
\email{{evital@math.uni-bielefeld.de}}

\begin{document}

\begin{abstract}
 In this paper, we introduce morphisms for matroids with coefficients (in the sense of Baker and Bowler) and quiver matroids. We investigate their basic properties, such as functoriality, duality, minors and cryptomorphic characterizations in terms of vectors, circuits and bases (a.k.a.\ Grassmann-Pl\"ucker functions). We generalize quiver matroids to quiver matroid bundles and construct their moduli space, which is an $\Fun$-analogue of a complex quiver Grassmannian. Eventually we introduce a suitable interpretation of $\Fun$-points for these moduli spaces, so that in `nice' cases their number is equal to the Euler characteristic of the associated complex quiver Grassmannian.
\end{abstract}

\maketitle

\thispagestyle{empty} 

\begin{small} \tableofcontents \end{small}

\section*{Introduction}

This text provides a general framework for a m\'elange of ideas around matroid theory, quiver representations and $\Fun$-geometry. More precisely, this concerns the following.

\subsection*{Matroid quotients and flag matroids}
A matroid quotient is a strong map that is the identity on the respective underlying sets. Matroid quotients have been generalized to oriented matroids in \cite{LasVergnas75} (see also \cite[Def.~7.7.2]{Bjorner-LasVergnas-Sturmfels-White-Ziegler99}), to valuated matroids in and \cite{Brandt-Eur-Zhang21} and to matroids with coefficients in idylls in \cite{Jarra-Lorscheid24}.

Flag matroids have been introduced in \cite{Borovik-Gelfand-Vince-White97} as a sequence of matroid quotients, and they have been generalized to valuated flag matroids in \cite{Haque12} and \cite{Brandt-Eur-Zhang21} and to flag matroids with coefficients in idylls in \cite{Jarra-Lorscheid24}, including cryptomorphic descriptions in terms of circuits, vectors and Grassmann-Pl\"ucker functions, duality and minors, and corresponding moduli spaces, which are flag varieties over the regular partial field.

\subsection*{Quiver representations and quiver Grassmannians}
\label{subsection - Quiver repr. and quiver Grassmannians}
A quiver representation over a field is a collection of vector spaces, indexed by the vertices of the quiver, together with a collection of linear maps, indexed by the arrows of the quiver. Quiver representations have been widely studied because of their central relevance for the representation theory of finite dimensional algebras; for instance, see \cite{Assem-Simson-Skowronski06}.

The quiver Grassmannian of a quiver representation is the collection of all subrepresentations with a given dimension vector. The Caldero-Chapoton formula (\cite{Caldero-Chapoton06}) expresses the structure constants of cluster algebras in terms of the Euler characteristics of complex quiver Grassmannians, which turns an intricate combinatorial problem into a geometric problem. While the quiver Grassmannians of a wild quiver assume every isomorphism type of a projective scheme (\cite{Reineke2013}, \cite{Ringel18}), quiver Grassmannians of tame quivers are varieties that have a Schubert decomposition into affine spaces (\cite{Lorscheid-Weist19}). In particular, the Euler characteristic corresponds to the number of Schubert cells in the latter case.

Flag varieties resurface as a particular type of quiver Grassmannians: the underlying quiver is of (equi-oriented) type $A_n$ and the quiver representation consists of a sequence of identity self-maps of a linear $n$-space. A point of this flag variety corresponds to a flag of linear subspaces, which is the same thing as a flag matroid over the base field.

\subsection*{\texorpdfstring{$\Fun$}{F1}-models}
A general heuristic for $\Fun$-geometry is that the number of $\Fun$-points of an $\Fun$-variety $X$ should be equal to the Euler characteristic of the base extension of $X$ to $\C$ (this goes back to \cite[Section 6.2]{Soule04}; for details see~\cite[Section 4]{Lorscheid16}). This provides a new strategy to study the Euler characteristic of a quiver Grassmannian in terms of $\Fun$-points of a suitable $\Fun$-model.\footnote{For the purpose of this introduction, we use $\Fun$ as a place holder for different theories that are used in the cited papers.} 

Quiver representations over $\Fun$ have been studied in \cite{Szczesny11}, \cite{Jun-SistkoEuler}, \cite{Jun-Sistko23}, \cite{Jun-Sistko23b}, \cite{Kleinau24}. As a particular example of a quiver Grassmannian, flag varieties over $\Fun$ have been shown to be the moduli space of flag matroids (\cite{Baker-Lorscheid21}, \cite{Jarra-Lorscheid24}). Models of quiver Grassmannians over $\Fun$ have been considered in \cite{Lorscheid16} where also the connection to Euler characteristics of the complex quiver Grassmannian was studied in special cases. Tropical quiver Grassmannians have been considered in \cite{Iezzi-Schleis23}.

\subsection*{Outline of the contents of this text}
We provide a common footing for all of the aforementioned concepts. This includes the development of: 
\begin{itemize}
 \item \emph{morphisms of matroids with coefficients}, as a common generalization of strong maps of usual matroids (\autoref{propC}), quotients of oriented matroids and $k$-linear maps where $k$ is a field;
 \item \emph{quiver matroids}, which is a common generalization of quiver representations and flag matroids;
 \item \emph{quiver Grassmannians over $\Fun$}, which generalize Grassmannians and flag varieties over $\Fun$.
\end{itemize}
Our results include:
\begin{itemize}
 \item functoriality (\autoref{propA}), duality (\autoref{propB}), pre-images (\autoref{propD}) and minors (\autoref{propF}) for matroid morphisms;
 \item cryptomorphic characterizations of matroid morphisms and quiver matroids in terms of circuits, vectors and Grassmann-Pl\"ucker functions (\autoref{thmE});
 \item the quiver Grassmannian over $\Fun$ as the moduli space of quiver matroid bundles (\autoref{thmG});
 \item the Euler characteristic of complex quiver Grassmannians as $\Fun$-points of their $\Fun$-models for `nice' representations (\autoref{thmH}).
\end{itemize}
In the following, we explain these contents in more detail.

\subsection*{Morphisms of matroids}

For the purpose of this introduction, we assume that the reader is familiar with Baker-Bowler theory (as in \cite{Baker-Bowler19}); we refer to \autoref{Baker-Bowler theory} for a summary. An \emph{idyll} is a tract whose nullset is closed under addition. An idyll is \emph{perfect} if vectors are orthogonal to covectors for all matroids over the idyll, which is an important property for our purposes; in particular, the notions of weak and strong matroids agree for perfect idylls (\cite[Thm. 3.46]{Baker-Bowler19}). Examples of perfect idylls are fields and partial fields, the Krasner hyperfield $\K$, the tropical hyperfield $\T$, the sign hyperfield $\S$ and, more generally, stringent hyperfields (\cite{Bowler_Pendavingh_2019}). More generally, every idyll can be completed to a perfect idyll (\cite{Baker_Zhang_2023}).

Let $F$ be a perfect idyll. A matrix $\Phi$ with entries $\Phi_{i,j}\in F$ is \emph{submonomial} if every row and every column contains at most one nonzero entry. If $\Phi$ is indexed by $T\times S$, then it defines a map $\Phi:F^S\to F^T$ by the rule $(\Phi\cdot X)_i=\sum_{j\in S} \Phi_{i,j} X_j$, which defines an element of $F$ since the sum contains at most one nonzero term. Note that the composition of such maps is well-defined since the product $\Psi\cdot\Phi$ of submonomial matrices $\Phi$ and $\Psi$ is again submonomial. 

Consider $F$-matroids $N$ and $M$ with respective underlying sets $S$ and $T$ and respective vector sets $\cV_N\subseteq F^S$ and $\cV_M\subseteq F^T$. A \emph{morphism $\Phi:N\to M$ of $F$-matroids} (or, for short, an \emph{$F$-morphism}) is a submonomial matrix $\Phi$ indexed by $T\times S$ such that $\Phi\cdot \cV_N\subseteq\cV_M$. It is clear from the definition that morphisms of $F$-matroids are composable, which yields the category $\Mat_F$ of $F$-matroids.

\subsubsection*{Push-forward along an idyll morphism}

Let $f:F\to G$ be a morphism of perfect idylls and $\Phi$ a submonomial matrix with coefficients in $F$. The \emph{push-forward of $\Phi$ along $f$} is the matrix $f_\ast\Phi$ with coefficients $(f_\ast\Phi)_{i,j}=f(\Phi_{i,j})$, which are elements of $G$. The following is \autoref{prop-functoriality}.

\begin{propA}[functoriality]\label{propA}
 The push-forward of an $F$-morphism $\Phi:N\to M$ along a morphism $f:F\to G$ of perfect idylls is a $G$-morphism $f_\ast\Phi:f_\ast N\to f_\ast M$. This defines a functor $f_\ast:\Mat_F\to\Mat_{G}$.
\end{propA}

\subsubsection*{Duality}

We denote the transpose of a submonomial matrix $\Phi$ by $\Phi^t$, which is also a submonomial matrix.

\begin{propA}[duality]\label{propB}
 The transpose of an $F$-morphism $\Phi:N\to M$ is an $F$-morphism $\Phi^t:M^\ast\to N^\ast$. This defines a contravariant endofunctor on $\Mat_F$ whose square is the identity. 
\end{propA}

\subsubsection*{Relation to strong maps}
The \emph{augmentation} of a set $S$ is the pointed set $\underline{S}=S\sqcup\{0\}$ with base point $0$. Let $\Phi$ be a submonomial matrix with coefficients in $F$ and indexed by $T\times S$. The \emph{underlying map} of $\Phi$ is the map $\uPhi:\uS\to\uT$ with $\underline\Phi(j)=i$ if $\Phi_{i,j}\neq0$ and $\underline\Phi(j)=0$ if $j=0$ or if $\Phi_{i,j}=0$ for all $i$. Note that $\uPhi:\uS\to\uT$ is \emph{$\Fun$-linear}, in the sense that it preserves base points and that for every $i\in T$, there is at most one $j\in S$ with $\uPhi(j)=i$.

We denote the underlying matroid of an $F$-matroid $M$ by $\underline{M}$. Note that $\uM$ has the same ground set as $M$, and that strong maps of matroids are, by definition, maps between augmented ground sets.

\begin{propA}[relation to strong maps]\label{propC}
 Let $N$ and $M$ be $\K$-matroids with respective underlying sets $S$ and $T$. A submonomial matrix $\Phi$ with entries in $\K$ is a morphism $\Phi:N\to M$ of $\K$-matroids if and only if $\uPhi:\uN\to\uM$ is a strong map. In particular, the underlying map $\underline\Phi:\underline S\to\underline T$ of an $F$-morphism $\Phi:N\to M$ (where $F$ is any perfect idyll) is a strong map $\uPhi:\uN\to\uM$ between the underlying matroids. 
\end{propA}

\begin{rem*}
 Since $1$ is the unique nonzero element of $\K$, a submonomial matrix $\Phi$ with coefficients in $\K$ is uniquely determined by its underlying map $\uPhi$. For the same reason, a $\K$-matroid $M$ is uniquely determined by its underlying matroid $\uM$. This means that \autoref{propC} provides a faithful embedding of $\Mat_\K$ into the category of usual matroids together with strong maps.
 
 While every matroid is the underlying matroid $\uM$ of a (uniquely determined) $\K$-matroid $M$, not every strong map comes from a $\K$-morphism. More precisely, a strong map $\sigma:\uN\to\uM$ comes from a $\K$-morphism if and only if it is $\Fun$-linear. 
 
 This means that isomorphisms, restrictions, contractions, quotients, and compositions thereof, come from $\K$-morphism. Strong maps that identify distinct parallel elements do not come from $\K$-morphisms.
\end{rem*}

\subsubsection*{Further examples}

If $k$ is a field, then $\cV_N$ is a linear subspace of $k^S$. A $k$-morphism is a $k$-linear map $\Phi:k^S\to k^T$ that maps the subspace $\cV_N$ to $\cV_M$, that maps coordinate vectors of $k^S$ to coordinate vectors of $k^T$ and whose kernel is spanned by coordinate vectors of $k^S$.

A $\T$-matroid is a tropical linear space $\cV_N$ in $\T^S$. A $\T$-morphism is a $\T$-linear map $\Phi:\T^S\to\T^T$ that preserves the corresponding tropical linear spaces $\cV_N$ and $\cV_M$, that maps coordinate vectors to coordinate vectors and whose non-empty fibres all have the same dimension; see also \cite[Section 2]{Iezzi-Schleis23}.

Quotients of oriented matroids are $\S$-morphisms; see \cite[Def. 7.7.2]{Bjorner-LasVergnas-Sturmfels-White-Ziegler99}.

\begin{warning*}
 There is a tension between linear algebra and matroid theory, which stems from the fact for a field $k$, the vector set of a $k$-matroid represents the dual of its underlying matroid. This culminates in the most confusing fact that a (submonomial) linear map represents a strong map in the opposite direction. 

 To explain, the rank of a $k$-matroid $M$ is equal to the \emph{codimension} of its vector set $\cV_M\subseteq k^T$. More to the point, $\cV_M$ represents the dual $\underline M^\ast$ of the underlying matroid $\underline M$ of $M$. The \emph{dimension} of $\cV_M$ equals the rank of $\underline M^\ast$.
 
 A $k$-morphism $\Phi:N\to M$ between $k$-matroids $N$ and $M$ defines a strong map $\underline \Phi:\underline N\to\underline M$ between the respective underlying matroids \emph{in the same direction}. The matroids \emph{represented by $\cV_N$ and $\cV_M$} are, however, the duals of $\underline N$ and $\underline M$, respectively. Therefore $\Phi$ induces the strong map $\underline\Phi^t:\underline M^\ast\to \underline N^\ast$ between the matroids represented by $\cV_N$ and $\cV_M$ \emph{in the opposite direction}.
\end{warning*}

\subsubsection*{Pre-image along a matroid morphism}
Let $M$ be an $F$-matroid on $T$ and $\Phi$ a submonomial matrix with coefficients in $F$ indexed by $T\times S$.

\begin{propA}[pre-image]\label{propD}
 There is a unique $F$-matroid $N$ on $S$ with vector set
 \[
  \cV_N \ = \ \Phi^{-1}(\cV_M) \ = \ \big\{ V\in F^S \, \big| \, \Phi\cdot V\in\cV_M \big\}.
 \]
\end{propA}

We write $\Phi^{-1}(M)$ for $N$ and call it the \emph{pre-image of $M$ along $\Phi$}. See \autoref{subsection: pre-image} for a characterization of the pre-image in terms of Grassmann-Pl\"ucker functions.

\subsubsection*{Cryptomorphisms}

Let $N$ and $M$ be $F$-matroids with respective underlying sets $S$ and $T$, respective circuit sets $\cC_N$ and $\cC_M$, and respective Grassmann-Pl\"ucker functions $\nu$ and $\mu$. We denote the cocircuits of $M$ by $\cC_M^\ast$. Let $w$ and $r$ be the respective ranks of $N$ and $M$. The following summarizes \autoref{thm: cryptomorphisms} and \autoref{factorization of surjective strong maps through quotients}.

\begin{thmA}[cryptomorphisms]\label{thmE}
 Let $\Phi$ be a submonomial $F$-matrix, indexed by $T\times S$. Then the following are equivalent:
 \begin{enumerate}
  \item $\Phi:N\to M$ is a morphism, i.e.\ $\Phi\cdot\cV_N\subseteq\cV_M$;
  \item $\Phi^{-1}(M)$ is a quotient of $N$, i.e.\ $\cV_N\subseteq\cV_{\Phi^{-1}(M)}$;
  \item $\Phi\cdot\cC_N\perp\cC_M^\ast$;
  \item for all $y_0,\dotsc,y_{w}\in S$ and $x_2,\dotsc,x_{r}\in T$,
  \[
   \sum_{k=0}^{w} \ (-1)^k \cdot \nu(y_0,\dotsc,\widehat{y_k},\dotsc,y_{w}) \cdot \Phi_{\uPhi(y_k),y_k} \cdot \mu\big(\underline{\Phi}(y_k),x_2,\dotsc,x_{r}\big) \quad \in \quad N_F,
  \]
  with the convention that $\Phi_{0,y_k}=\mu(0,x_2,\dotsc,x_r\big)=0$.
 \end{enumerate}
\end{thmA}

\subsubsection*{Minors}
Let $N$ and $M$ be $F$-matroids with respective ground sets $S$ and $T$ and $\Phi:N\to M$ an $F$-morphism. For subsets $A\subseteq S$ and $B\subseteq T$ with complements $A^c=S-A$ and $B^c=T-B$, we define the following submatrices of $\Phi$:
\[
 \Phi|(A\to B) \ = \ (\Phi_{i,j})_{(i,j)\in B\times A} \quad \text{and} \quad \Phi/(A\to B) \ = \ (\Phi_{i,j})_{(i,j)\in B^c\times A^c}.
\]

\begin{propA}[minors]\label{propF}
 If $\uPhi(\uA)\subseteq\uB$, then 
 \[
  \Phi\vert{(A \to B)}:N\vert A\to M\vert B \qquad \text{and} \qquad \Phi/{(A\to B)}:N/A\to M/B
 \]
 are $F$-morphisms (called the \emph{restriction} and \emph{contraction} of $\Phi$ to $A\to B$, respectively). They satisfy $\Phi^t|(B^c\to A^c) = \big(\Phi / (A\to B) \big)^t$ and $f_\ast\big(\Phi\vert{(A\to B)}\big)=(f_\ast\Phi)\vert{(A\to B)}$ for any morphism $f: F \to G$ of perfect idylls.
\end{propA}

\subsection*{Quiver matroids}

Let $F$ be a perfect idyll and $Q$ a quiver with vertex set $Q_0$ and arrow set $Q_1$. A \emph{$(Q,F)$-matroid} is a tuple $(M_v)_{v\in Q_0}$ of $F$-matroids $M_v$ together with a tuple $(M_\alpha)_{\alpha\in Q_1}$ of $F$-morphisms $M_\alpha:M_{s(\alpha)}\to M_{t(\alpha)}$ where $s(\alpha)$ is the source of $\alpha$ and $t(\alpha)$ is the target of $\alpha$. We write $M = \big((M_v),(M_\alpha)\big)$ for this $(Q,F)$-matroid.

\subsubsection*{Examples}
Particular instances of quiver matroids include:
\begin{itemize}
 \item Flag $F$-matroids (in the sense of \cite{Jarra-Lorscheid24}) are $(Q,F)$-matroids $M$ for which $Q$ is of the form
 \[
  \begin{tikzcd}
   1 & 2 \ar[l] & \ar[l] \dotsb  & s \ar[l]
  \end{tikzcd}
 \]
 and for which $M_\alpha$ is the square identity matrix (for every $\alpha$); see \autoref{flag matroids}.
 \item Let $k$ be a field and $Q$ a quiver. A $(Q,k)$-matroid corresponds to a subrepresentation of a $Q$-representation $\Lambda$ over $k$ together with a choice of basis $E_v$ of $\Lambda_v$ (for every $v\in Q_0$) such that $\Lambda_\alpha$ is submonomial with respect to the bases $E_{s(\alpha)}$ and $E_{t(\alpha)}$ (for every $\alpha\in Q_1$); see \autoref{ex: matroid morphisms}.
 \item In the case of the tropical hyperfield $\T$, we gain the notion of valuated quiver matroid, which is a tuple of tropical linear spaces together with submonomial matrices that preserve them; see \autoref{ex: matroid morphisms}. 
\end{itemize}

\subsubsection*{Duality}
Let $M$ be an $(Q,F)$-matroid and $Q^\ast$ the \emph{opposite quiver}, which has the same vertices and arrows, but with reversed direction. Then the \emph{dual of $M$} is the $(Q^\ast,F)$-matroid $M^\ast$ with $(M^\ast)_v:=(M_v)^\ast$ for $v\in Q_0^\ast=Q_0$ and $(M^\ast)_{\alpha}:=M_{\alpha}^t$ for $\alpha\in Q_1^\ast = Q_1$.

\subsubsection*{Minors}
Let $M$ be an $(Q,F)$-matroid and $E_v$ the underlying set of $M_v$ for $v\in Q_0$. Let $\Lambda_v\subseteq \underline{E_v}$ be a pointed subset (for every $v\in Q_0$) such that $\underline{M_\alpha}(\Lambda_s)\subseteq \Lambda_t$ for every $\alpha\in Q_1$ with source $s$ and target $t$. Define $A_v := \Lambda_v - \{0\}$ and $\Lambda := (\Lambda_v)_{v \in Q_0}$.

The \emph{restriction of $M$ to $\Lambda$} is the $(Q,F)$-matroid $M|\Lambda$ with $(M|\Lambda)_v=M_v|A_v$ (for $v\in Q_0$) and $(M|\Lambda)_\alpha=M_\alpha|(A_s\to A_t)$ (for $\alpha\in Q_1$ with source $s$ and target $t$).

The \emph{contraction of $M$ by $\Lambda$} is the $(Q,F)$-matroid $M/\Lambda$ with $(M/\Lambda)_v=M_v/A_v$ (for $v\in Q_0$) and $(M/\Lambda)_\alpha=M_\alpha/(A_s\to A_t)$ (for $\alpha\in Q_1$ with source $s$ and target $t$).

For more details, see \autoref{Section: Duality and minors of quiver matroids}.

\subsubsection*{Relation to \texorpdfstring{$\Fun$}{F1}-representations of quivers}
An $\Fun$-representation of $Q$ is a tuple $\Lambda=\big(({\Lambda}_v)_{v\in Q_0},(\Lambda_\alpha)_{\alpha\in Q_1}\big)$ where $\Lambda_v$ is a finite pointed set with base point $0$ (for $v\in Q_0$) and $\Lambda_\alpha:{\Lambda}_{s(\alpha)}\to{\Lambda}_{t(\alpha)}$ (for $\alpha\in Q_1$) is an \emph{$\Fun$-linear map}, i.e.\ $\Lambda_\alpha(0)=0$ and $\#\Lambda_\alpha^{-1}(e)\leq1$ for $e\neq0$.

It is a general phenomenon that (postulated) objects in $\Fun$-geometry surface as certain types of matroids, which are usually of a particularly simple shape. Quiver representations form a first such instance (for this paper): an $\Fun$-representation $\Lambda$ defines a $(Q,\K)$-matroid $\cat{M}_Q(\Lambda)$ whose matroids $\cat{M}_Q(\Lambda)_v$ (for $v\in Q_0$) are characterized by their vector sets $\cV_{\cat{M}_Q(\Lambda)_v}=\K^{\Lambda_v-\{0\}}$ and whose submonomial matrices $\cat{M}_Q(\Lambda)_\alpha$ (for $\alpha\in Q_1$) have coefficients
\[
\big(\cat{M}_Q(\Lambda_\alpha)\big)_{i,j} \ = \  \begin{cases}
                                       1 & \text{if }\Lambda_\alpha(j)=i, \\
                                       0 & \text{if }\Lambda_\alpha(j)\neq i.
                                      \end{cases}
\]
In other words, $\cat{M}_Q(\Lambda)_\alpha$ is the unique submonomial matrix with coefficients in $\K$ such that $\underline{\cat{M}_Q(\Lambda)_\alpha}=\Lambda_\alpha$.

Conversely, every $(Q,\K)$-matroid $M$ has an \emph{underlying $\Fun$-representation} $\cat{U}_Q(M)$ with pointed sets $\cat{U}_Q(M)_v=\underline{E_v}$ (for $v\in Q_0$), where $E_v$ is the ground set of $M_v$, and $\cat{U}_Q(M)_\alpha=\underline{M_\alpha}$. We have $\cat{U}_Q \circ \cat{M}_Q(\Lambda)=\Lambda$ for every $\Fun$-representation $\Lambda$ of $Q$, which, in fact, extends to an adjunction between the category of $\Fun$-representations of $Q$ and $(Q,\K)$-matroids with respect to a suitable notion of morphisms (see \autoref{Cor: adjunction}).

\subsection*{Quiver Grassmannians}
Let $\Lambda$ be an $\Fun$-representation of $Q$ and $\textbf{r}$ a \emph{rank vector for $Q$}, which is a tuple $\textbf{r}=(r_v)_{v\in Q_0}$ of integers $r_v\geq0$. The \emph{rank} of a $(Q,F)$-matroid $M$ is the tuple $\underline{\rk}(M)=\big(\rk(M_v)\big)_{v\in Q_0}$. 

As a point set over $F$, the quiver Grassmannian of $(\Lambda, F)$-matroids of rank $\textbf{r}$ is 
\[
 \Gr_{\textbf{r}}(\Lambda)(F) \ = \ \big\{ \text{$(Q,F)$-matroids $M$ of rank $\textbf{r}$ with $\cat{U}_Q(M)=\Lambda$} \big\}
\]
where $\cat{U}_Q(M)$ is the underlying $\Fun$-representation of $M$.

\subsubsection*{Relation to complex quiver Grassmannians}
An $\Fun$-representation $\Lambda$ determines a complex representation $\Lambda_\C$ of $Q$ with vector spaces $\Lambda_{\C,v}=\C^{E_v}$ where $E_v=\Lambda_v-\{0\}$ (for $v\in Q_0$) and complex matrices $\Lambda_{\C,\alpha}$ (for $\alpha\in Q_1$) with coefficients
\[
 (\Lambda_{\C,\alpha})_{i,j} \ = \ \begin{cases}
                                    1 & \text{if }\Lambda_\alpha(j)=i, \\
                                    0 & \text{if }\Lambda_\alpha(j)\neq i
                                   \end{cases}
\]
for $(i,j)\in E_{t(\alpha)}\times E_{s(\alpha)}$. The quiver Grassmannian $\Gr_{\textbf{r}}(\Lambda)(\C)$ in the above sense agrees with the usual complex quiver Grassmannian 
\[
 \Gr_{{\textbf{r}}^\ast}(\Lambda_\C)(\C) \ = \ \big\{ \text{subrepresentations $N$ of $\Lambda_\C$ with $\underline{\dim}(N)={\textbf{r}}^\ast$} \big\},
\]
where ${\textbf{r}}^\ast=(\# E_v-r_v)_{v\in Q_0}$. Note that we need to pass to the dimension vector ${\textbf{r}}^\ast$ since the rank of a $\C$-matroid $M$ with ground set $E$ equals the codimension of its vectors $\cV_{M}$ as a subspace of $\C^{E}$.

Note further that we do not need to invoke subrepresentations in the definition of $\Gr_{\textbf{r}}(\Lambda)(F)$ since (the vector set of) a $\C$-matroid $M$ with ground set $E$ is by definition a subspace $\cV_{M}$ of $\C^{E}$, i.e.\ a $(Q,\C)$-matroid $N$ in $\Gr_{\textbf{r}}(\Lambda)(\C)$ is naturally a \emph{subrepresentation} of $\Lambda_\C$, or a \emph{quotient} if we adhere to the terminological conventions of matroid theory.

In a similar vein, $\Gr_{\br}(\Lambda)(\T)$ stays in a natural bijection with the quiver Dressian of \cite{Iezzi-Schleis23}.

\subsubsection*{The quiver Grassmannian as a moduli space}

The push-forward along an perfect idyll morphism $f:F\to G$ defines a map 
\[
 f_\ast: \ \Gr_{\textbf{r}}(\Lambda)(F) \ \longrightarrow \ \Gr_{\textbf{r}}(\Lambda)(G), 
\]
which turns $\Gr_{\textbf{r}}(\Lambda)$ into a functor from $\Idylls$ to $\Sets$. This functor is represented by an $\Fun$-variety in a suitable sense. For this, we employ the theory of band schemes, as introduced in \cite{Baker-Jin-Lorscheid24}; see \autoref{section: The moduli space of quiver matroids} for details. 

More accurately, we first extend the functor $\Gr_\br(\Lambda)$ to band schemes in the following way. A \emph{matroid bundle} over a band scheme is a suitable generalization of a subbundle of $\cO_X^n$ in usual algebraic geometry, where $X$ is a scheme and $n$ is the cardinality of the ground set $E$. This notion extends to $\Lambda$-matroid bundles on a band scheme $X$ in the same way as $F$-matroids generalize to $(\Lambda,F)$-matroids. This leads to the definition of $\Gr_\br(\Lambda)(X)$ as the collection of all $\Lambda$-matroid bundles on $X$ with rank vector $\br$.

\begin{thmA}\label{thmG}
 The functor $\Gr_\br(\Lambda)$ has the structure of a band scheme. In other words, the band scheme $\Gr_{\mathbf{r}}(\Lambda)$ is the fine moduli space of $\Lambda$-matroid bundles. 
\end{thmA}

\subsection*{Euler characteristic and \texorpdfstring{$\Fun$}{F1}-rational points}

The following heuristic reasoning leads to the expectation that the number of $\Fun$-points of a sufficiently nice $\Fun$-variety $X$ is equal to the (topological) Euler characteristic of its $\C$-points $X(\C)$ (with respect to the complex topology). 

Namely, assume that $X$ is an $\Fun$-model of a smooth projective scheme $X_\Z$ for which the number $N(q)=\#X_\Z(\F_q)$ of $\F_q$-points (for prime powers $q$) is a polynomial in $q$. Then we expect that the number of $\Fun$-points of $X$ is
\[
 N(1) \ \ = \ \ \lim_{q\to1} \ N(q).
\]

As a consequence of the comparison result of singular cohomology with $l$-adic cohomology and of Deligne's proof of the Weil conjectures, 
\[
 N(q) \ = \ \sum_{i=0}^{\dim X} b_{2i} \cdot q^i
\]
where $b_i$ are the Betti numbers of the complex manifold $X(\C)$, with $b_{2i+1}=0$ for odd indices (see \cite[Section 7]{Lorscheid16} for more details). Thus the number $N(1)$ of $\Fun$-points of $X$ equals the Euler characteristic of $X(\C)$:
\[
 \chi\big(X(\C)\big) \ = \ \sum_{i=0}^{2\dim X} (-1)^i \cdot b_i \ = \ \sum_{i=0}^{\dim X} b_{2i} \cdot 1^i\ = \ N(1) \ = \ \lim_{q\to 1} \ \# X(\F_q).
\]

This heuristic is reflected by the following expectation on the Grassmannian over $\Fun$: an $r$-dimensional subspace of an $\Fun$-vector space $\Lambda$ (which is a pointed set) is a pointed subset $N\subseteq\Lambda$ with $\#(N-\{0\})=r$. In other words, $N-\{0\}$ is an $r$-subset of $\Lambda-\{0\}$. This leads to the expectation that 
\[
 \Gr(r,n)(\Fun) \ = \ \big\{ \text{$r$-subsets of $\{1,\dotsc,n\}$} \big\}
\]
has $\binom nr$ elements, which equals the Euler characteristic of $\Gr(r,n)(\C)$ as well as the limit
\[
 \lim_{q\to1} \ \# \Gr(r,n)(\F_q) \ \ = \ \ \lim_{q\to1} \ \ \frac{[n]_q!}{[r]_q!\cdot [n-r]_q!} \ \ = \ \ \binom nr,
\]
where $[n]_q=q^{n-1}+\dotsc+q+1$ is the $q$-number and $[n]_q!=[n]_q\cdots[1]_q$ is the $q$-factorial.

These expectations are not met, however, if we interpret the notion of $\Fun$-points of $X$ in the usual way, i.e.\ as morphisms $\Spec\Fun\to X$. For example, $X=\P^n_\Fun$ has $2^{n+1}-1$ such morphisms (say, when considered as monoid schemes\footnote{Note that finite type monoid schemes form a full subcategory of virtually any other approach to $\Fun$-geometry, so the number of $\Fun$-rational points of $\P^n$ does not depend on the notion of $\Fun$-geometry that we use.}), but the Euler characteristic of $\P^n(\C)$ is $n+1$.

In many instances, our expectations on $\Fun$-points are met by certain $\K$-points of a suitable $\Fun$-model. To explain, a band scheme $X$ comes with a set $X(\C)$ of $\C$-points and a set $X(\K)$ of $\K$-points, which are both equipped with a natural topology. The expected $\Fun$-points agree with the closed $\K$-points in the following examples (cf. \cite{Lorscheid-Thas23}, \cite[Section 3.4]{Baker-Jin-Lorscheid24}):
\begin{itemize}
 \item $\P^n$ has $n+1$ closed $\K$-points, which recovers the Euler characteristic of $\P^n(\C)$.
 \item The Weyl group $\mathfrak{S}_n$ of $\SL_n$ stays in bijection with the closed points of $\SL_n(\K)$.
 \item The closed $\K$-points of a flag variety correspond bijectively to the simplices (of the corresponding type) of the Coxeter complex for $\mathfrak{S}_n$, and their number agrees with the Euler characteristic of the corresponding complex flag variety.
\end{itemize}

We call the set of closed $\K$-points of $X$ the \emph{Tits space} and denote it by $X^\Tits$. The relation between the Euler characteristic and the Tits space generalizes to all quiver Grassmannians of sufficiently `nice' $\Fun$-representations. This includes the following cases (see \autoref{thm: Euler characteristics for tree modules} and \autoref{cor: Euler characteristic for quivers with at most one cycle}).

\begin{thmA}\label{thmH}
 Let $Q$ be a finite quiver and $\Lambda$ an $\Fun$-representation. Assume that one of the following conditions holds:
 \begin{enumerate}
  \item The coefficient quiver $\Gamma$ of $\Lambda$ is a tree (see \autoref{subsection: the Euler characteristic and the Tits space} for the definition of $\Gamma$).
  \item $Q$ is a quiver of simply laced (extended) Dynkin type and $\Lambda_\C$ is irreducible.
 \end{enumerate}
 Then
 \[
  \chi\big(\Gr_{\br}(\Lambda)(\C)\big) \ = \ \# \Gr_{\br}(\Lambda)^\Tits.
 \]
 for every dimension vector ${\br}$ for $Q$.
\end{thmA}

\subsubsection*{Geometric sketch of proof}

Every dimension ${\textbf{r}^*}$ subrepresentation of $\Lambda$, which is a collection of pointed subsets $\Omega_v$ of $\Lambda_v$ (for $v\in Q_0$) that is stable under the maps $\Lambda_\alpha$ (for $\alpha\in Q_1$), determines a so-called \emph{coordinate matroid} $M$, which has $B_v=\Lambda_v-\Omega_v$ as its unique basis at $v$. By \autoref{thmG}, $M$ corresponds to a point $x_M$ of $\Gr_{\textbf{r}}(\Lambda)(\K)$. Since $x_M$ has a unique nonzero coordinate at $B_v$ for each $v$, it is closed in $\Gr_{\textbf{r}}(\Lambda)(\K)$, which establishes an injection 
\[
 \big\{\text{dimension ${\textbf{r}^*}$ subrepresentations of $\Lambda$}\big\} \ \longrightarrow \ \Gr_{\textbf{r}}(\Lambda)^\Tits.
\]

Under the assumptions of \autoref{thmH}, there is a sequence of torus actions that allows us to establish an equality between these sets and between their cardinalities and the Euler characteristic of $\Gr_{\textbf{r}}(\Lambda)(\C)$. 

Namely, Jun and Sistko's work \cite{Jun-SistkoEuler} provides a \emph{nice sequence of gradings of $\Lambda$ that distinguishes elements}. These gradings determine the weights of actions 
\[
 \theta_i: \ \G_m \times X_i \ \longrightarrow \ X_i
\]
(for $i=1,\dotsc,n$) of the multiplicative group scheme $\G_m$ on $X_1=\Gr_{\textbf{r}}(\Lambda)$ and on the subspace $X_{i+1}$ of fixed points of $\theta_{i}$ (for $i=1,\dotsc,n$), with $X_{n+1}$ being finite.

By a theorem of Bia\l ynicki-Birula (\cite{Bialynicki-Birula73}), the Euler characteristic of $\Gr_{\textbf{r}}(\Lambda)(\C)$ is equal to the Euler characteristic of the $\C^\times$-fixed points $X_2(\C)$ of $\theta_{1,\C}$, and thus, after repeating the argument $n$ times, equal to the Euler characteristic, or cardinality, of the finite set $X_{n+1}(\C)$.

Since the torus actions are defined in terms of gradings of $\Lambda$, it follows from first principles that $X_{n+1}(\C)$ corresponds to the dimension ${\textbf{r}^*}$ subrepresentations of $\Lambda$. 

In order to show that every point of $\Gr_{\textbf{r}}(\Lambda)^\Tits$ stems from a subrepresentation of $\Lambda$, we consider an arbitrary point $x$ in $\Gr_{\textbf{r}}(\Lambda)(\K)$. Since $\K$ is contained in the tropical hyperfield $\T$, we can consider $x$ as an element of the tropical quiver Grassmannian $\Gr_{\textbf{r}}(\Lambda)(\T)$. If $x$ does not correspond to a point in $X_{n+1}$, then there is an $i$ such that $\G_m(\T)$ acts via $\theta_{i,\T}$ non-trivially on $x\in X_i(\T)$. Since $\G_m(\T)=\R_{>0}$ (as multiplicative group), we can consider the limit
\[
 x_0 \ = \ \lim_{t\to0} \ \theta_{i,\T}(t,x),
\]
which is a point of $\Gr_{\textbf{r}}(\Lambda)(\K)$ with strictly more zero coefficients than $x$. Thus $x_0\neq x$ is in the closure of $x$, which shows that $x$ is not closed and not in $\Gr_{\textbf{r}}(\Lambda)^\Tits$. \qed

\begin{rem*}
 As explained in this proof sketch, the finite set $X_{n+1}(\C)$ stays in a natural bijection with the Tits space $\Gr_\br(\Lambda)^\Tits$. This fact is not specific to $\C$: since the nice sequence of gradings of $\Lambda$ is defined over $\Fun$, also the subschemes $X_i$ of $\Gr_\br(\Lambda)$ and the torus actions $\theta_i:\G_m\times X_i\to X_i$ are defined over $\Fun$, which leads to a natural identification of $X_{n+1}(F)$ with $\Gr_\br(\Lambda)^\Tits$ for every field $F$. Therefore we find that
 \[
  \lim_{q\to1} \ \#\,\Gr_\br(\Lambda)(\F_q) \ \ = \ \ \lim_{q\to1} \big( \# \, X_{n+1}(\F_q) + n(q)\cdot(q-1) \big) \ \ = \ \ \# \, \Gr_\br(\Lambda)^\Tits,
 \]
 where $n(q)$ is the number of non-trivial torus orbits of the actions $\theta_1,\dotsc,\theta_n$, which each have $q-1$ elements. This establishes the expectations on the number of $\Fun$-rational points for the quiver Grassmannians $\Gr_\br(\Lambda)$ under consideration.
\end{rem*}

\subsection*{Remark on generalizations}
While the definitions and concepts that we develop in this paper are sufficient to illustrate the interplay of matroid theory, $\Fun$-geometry and quiver representations, they lean themselves to generalizations in different directions, which might be of independent interest for a broader framework for matroid theory and quiver Grassmannians. At this point, we restrict ourselves to the following brief comments.

\subsubsection*{Idylls with involutions}
Baker and Bowler introduce $F$-matroids for idylls (and more generally, tracts) $F$ together with an involution $i:F\to F$, which allows, for instance, to take Hermitian complements of complex matroids, opposed to orthogonal complements. All results of this paper generalize readily to perfect idylls and tracts with involution; also cf.\ \autoref{subsection: morphisms for idylls with involution}.

\subsubsection*{Morphisms of matroid bundles}
The notion of morphisms of matroid bundles over band schemes has to be seen as a provisional concept. The problem is that the composition of two such morphisms are in general not a morphism in the sense of this text anymore. This effect occurs already for non-perfect idylls; cf.\ \autoref{rem: composition of morphisms}.

The failure of composability can be fixed in terms of the so-called \emph{perfection} of an idyll $F$, which is a perfect idyll that has the same class of (weak) $F$-matroids, but more morphisms. In particular, the composition of $F$-morphisms is defined over the perfection of $F$. A similar, but more subtle, line of thought leads to a satisfactory notion of morphisms for matroid bundles. All this is part of follow-up work.

\subsubsection*{Arbitrary matrices as morphisms}
Even for perfect idylls $F$, there is a more ample class of morphisms of $F$-matroids, which are given by $F$-matrices that are not necessarily submonomial. This notion of morphism is required to preserve certain $F^+$-semimodules $\cV_M^+\subseteq (F^+)^E$ of ``additive'' vectors for an $F$-matroid $M$.

This larger class of morphisms contains all strong maps (and more) in the case of the Krasner hyperfield $F=\K$. In case of a field $F$, it is equal to the class of all linear maps that preserve the given subspaces of vectors. Also this is part of follow-up work.

\subsubsection*{Quiver Grassmannians for $F$-representations}
In the definition of quiver Grassmannians, it is not necessary that the chosen ambient representation $\Lambda$ is $\Fun$-linear. For a fixed quiver $Q$ and a perfect idyll $F$, any choice of a rank vector $\br=(r_v)_{v\in Q_0}$, of a collection $(\Lambda_v)_{v\in Q_0}$ of finite ground sets $\Lambda_v$ and of a collection $(\Lambda_\alpha)_{\alpha\in Q_1}$ of $\Lambda_{t(\alpha)}\times \Lambda_{s(\alpha)}$-matrices $\Lambda_\alpha$ with coefficients in $F$ leads to the definition of a quiver Grassmannian $\Gr_\br(\Lambda)$ in terms of the quiver Pl\"ucker relations (cf.\ \autoref{subsection: classical quiver Grassmannians}), which is a band scheme defined over $F$. 

If the matrices $\Lambda_\alpha$ are submonomial, then for every perfect idyll $P$ together with a structure map $F\to P$, the rational point set $\Gr_\br(\Lambda)(P)$ consists of all $(Q,P)$-matroids $M$ with rank vector $\underline{\rk}(M)=\br$ and with $M_\alpha=\Lambda_\alpha$ for all $\alpha\in Q_1$. 

Since there is a unique morphism $F\to\K$, this leads to the definition of the \emph{Tits space} $\Gr_\br(\Lambda)^\Tits$ as the subspace of closed points of $\Gr_\br(\Lambda)(\K)$, which should be thought of as the $\Fun$-points of $\Gr_\br(\Lambda)$.

This more ample notion of quiver Grassmannians over $F=\C$ applies, in particular, to irreducible representations $\Lambda_\C$ of tame quiver (of extended Dynkin type $ADE$). A \emph{basis for $\Lambda_\C$} is a collection of bases $\Lambda_v$ for $\Lambda_{\C,v}$ (for all $v\in Q_0$). A basis for $\Lambda_\C$ identifies $\Lambda_{\C,\alpha}$ with a matrix $\Lambda_\alpha$ (for all $\alpha\in Q_1$). The verification of sample cases leads us to expect the following generalization of \autoref{thmH}.

\begin{conjA}
 Let $Q$ be a tame quiver, $\br$ a rank vector for $Q$ and $\Lambda_\C$ an irreducible representation of $Q$. Then there exists a basis for $\Lambda_\C$ such that 
 \[
  \chi\big(\Gr_{\br^\ast}(\Lambda_\C)\big) \ = \ \# \Gr_\br(\Lambda)^\Tits.
 \]
\end{conjA}

\subsection*{Acknowledgements}
The authors thank Jaiung Jun and Alex Sistko for their explanations on their technique to compute Euler characteristics of complex quiver Grassmannians. They thank Matt Baker and Victoria Schleis for helpful conversations. E.~Vital was funded by the Deutsche Forschungsgemeinschaft (DFG, German Research Foundation) – Project-ID 491392403 – TRR 358, and by the  National Council for Scientific and Technological Development - CNPq - Brazil, Proc.\ 200268/2022-8.

\section{Background}

\subsection{Vector spaces over \texorpdfstring{$\F_1$}{F1}}
\label{F1-vector spaces}

A \emph{vector space over $\Fun$} is a pointed set, whose base point we denote typically by $0$. An \emph{$\Fun$-linear map} between $\Fun$-vector spaces $V$ and $W$ is a base point preserving map $f:V\to W$ such that $\#f^{-1}(w)\leq1$ for all $w\neq0$. This defines the category $\Vect_\Fun$.

The \emph{adjoint} of an $\F_1$-linear map $f: V \to W$ is the $\F_1$-linear map $f^t: W \to V$ given by 
\[
f^t(w) \ = \ \begin{cases}
            v & \text{if} \;\; w \neq 0 \text{ and } f(v) = w\\
            0 & \text{otherwise}.
            \end{cases}
\]

\subsection{Strong maps of classical matroids}
\label{subsection - Strong maps of classical matroids}

Let $M$ be a matroid with ground set $E$. We denote the sets of circuits of $M$ by $\cC_M$ and its set of cocircuits by $\cC_M^\ast$.

\begin{df}
 A \emph{vector} of $M$ is a union of circuits of $M$. A \emph{covector} of $M$ is a vector of the dual matroid $M^*$. We denote the set of vectors of $M$ by $\mathcal{V}_M$ and its set of covectors by $\mathcal{V}^*_M$. 
\end{df}

Since a circuit is the same as the complement of a cohyperplane (\cite[Prop.\ 2.1.6]{Oxley11}), a vector is nothing else than the complement of a coflat. For our comparison of usual matroid theory with linear algebra and Baker-Bowler theory (cf.\ \autoref{Baker-Bowler theory}), it is more convenient to work with vectors than with coflats.

An alternative characterization of vectors is as follows. Two subsets $A$ and $B$ of $E$ are \emph{orthogonal}, denoted by $A \perp B$, if $\# (A \cap B) \neq 1$. Two collections $\mathcal{X}$ and $\mathcal{Y}$ of subsets of $E$ are \emph{orthogonal}, denoted by $\mathcal{X} \perp \mathcal{Y}$, if $X \perp Y$ for all $X \in \mathcal{X}$ and $Y \in \mathcal{Y}$. For $\mathcal{X} \subseteq 2^E$, we define $\mathcal{X}^\perp := \{Z \subseteq E \mid X \perp Z \text{ for all } X \in \mathcal{X}\}$. Since circuits are orthogonal to cocircuits (\cite[Prop.\ 2.1.11]{Oxley11}), we have that
\[
 \mathcal{V}_M^\ast \ = \ \mathcal{C}_M^\perp \ = \ \cV_M^\perp.
\]

\begin{notation}\label{classical matroid extension by 0}
We denote the closure operator of $M$ by $\textup{cl}_M$ and define the matroid $\widetilde{M} := M \oplus U_1^0$ on the set $\underline{E}$. In other words, $\widetilde{M}$ is obtained from $M$ by adding a loop at $0$.
\end{notation}

Let $N$ and $M$ be classical matroids with respective underlying sets $S$ and $T$. A \emph{strong map} $\sigma$ from $N$ to $M$, denoted by $\sigma: N \rightarrow M$, is a base point preserving map $\sigma: \underline{S} \rightarrow \underline{T}$ such that $\sigma^{-1}(F)$ is a flat of $\widetilde{N}$ for any flat $F$ of $\widetilde{M}$ (see \cite[Def. 8.1.1]{Kung86}). If $S=T$ and $\sigma = \textup{id}_{\underline{S}}$ is a strong map, $M$ is called a \emph{quotient} of $N$. By \cite[Prop. 8.1.3]{Kung86}, a base point preserving map $\sigma: \underline{S} \rightarrow \underline{T}$ is a strong map if, and only if, $\sigma\big( \textup{cl}_N(A)\big) \subseteq \textup{cl}_{\widetilde{M}} \big(\sigma(A)\big)$ for every $A \subseteq S$.

\begin{df}
An \emph{$\F_1$-linear strong map} $\sigma: N \rightarrow M$ is a strong map  $\sigma: \underline{S} \rightarrow \underline{T}$ that is $\F_1$-linear.
\end{df}

Examples of strong maps are isomorphisms of matroids, restrictions $M\vert A\to M$, contractions $M\to M/A$, quotient maps and the identification of parallel elements. Every strong map is the composition of these elementary types of strong maps (see \cite[pp. 227-228]{Kung86}). Among these types of strong maps, isomorphisms, restrictions, contractions and quotient maps are $\Fun$-linear. Nontrivial identifications of parallel elements are not.

In fact, the factorization theorem for strong maps (\cite[Thm 8.2.7]{Kung86}) asserts that every quotient map is the composition of a restriction followed by a contraction. This fact does not generalize to other idylls (see \cite[Cor.\ 3.5]{Richter-Gebert93}), which means that quotient maps play a central role as a building block of matroid morphisms; also cf.\ \autoref{factorization of classical strong map by quotient} and \autoref{factorization of surjective strong maps through quotients}.

\begin{df}[{\cite[p.\ 228]{Kung86}}]
\label{pre-image matroid}
Let $\sigma: \underline{S} \rightarrow \underline{T}$ be a base point preserving map and $M$ a matroid on $T$. We define $\sigma^{-1}(M)$ as the matroid on $S$ whose rank function is given by
\[
\rk_{\sigma^{-1}(M)}(A) := \rk_{\widetilde{M}}\big(\sigma(A)\big).
\]
\end{df}

\begin{rem}
\label{pre-image and restriction}
Note that $M| (\im(\sigma) - \{0\})$ is obtained from $\sigma^{-1}(M)$ by identifying parallel elements, deleting loops and relabeling elements.
\end{rem}

\begin{prop}
\label{factorization of classical strong map by quotient}
Let $\sigma: \underline{S} \rightarrow \underline{T}$ be a base point preserving map. Then $\sigma: N \rightarrow M$ is a strong map if, and only if, $\sigma^{-1}(M)$ is a quotient of $N$.
\end{prop}

\begin{proof}
Using the fact that
\[
\big\{\text{flats of } \widetilde{\big(\sigma^{-1}(M)\big)}\big\} = \big\{\sigma^{-1}(F) \;\big|\; F \text{ is a flat of } \widetilde{M}\big\},
\]
the Proposition follows. 
\end{proof}

\begin{rem}
\label{F1-linear and pre-image}
By \autoref{factorization of classical strong map by quotient}, a strong map $\sigma: N \rightarrow M$ is $\F_1$-linear if, and only if, it does not identify parallel elements.
\end{rem}

The following result generalizes \cite[Prop. 8.1.6]{Kung86} from matroid quotients to all strong maps.

\begin{prop}
\label{cryptomorphism - classical strong maps}
Let $\sigma: \underline{S} \rightarrow \underline{T}$ be a base point preserving map. Then $\sigma: N \rightarrow M$ is a strong map if, and only if, $\sigma^{-1}: 2^T \rightarrow 2^S$ sends cocircuits of $M$ to covectors of $N$.
\end{prop}

\begin{proof}
Assume that $\sigma: N \rightarrow M$ is a strong map. Let $Z$ be a cocircuit of $M$. Next we show that $\sigma^{-1}(Z) \in \mathcal{V}^*_N$. Suppose that there exists $C \in \mathcal{C}_N$ such that $C \cap \sigma^{-1}(Z) = \{x\}$ is a singleton. As $\mathcal{C}^*_M = \mathcal{C}^*_{\widetilde{M}}$, the set $H := \underline{T} - Z$ is a hyperplane of $\widetilde{M}$ that contains $\sigma(C-\{x\})$. Thus
\[
\sigma(x) \in \sigma(C) \subseteq \sigma\big( \textup{cl}_N(C-\{x\})\big) \subseteq \textup{cl}_{\widetilde{M}} \big(\sigma(C-\{x\})\big) \subseteq H,
\]
which is a contradiction, because $\sigma(x) \in Z$.

Assume $\sigma^{-1}(\mathcal{C}^*_M) \subseteq \mathcal{V}^*_N$. Let $A \subseteq S$. Next we prove that $\sigma\big( \textup{cl}_N(A)\big) \subseteq \textup{cl}_{\widetilde{M}} \big(\sigma(A)\big)$. Suppose that there exists $x \in \textup{cl}_N(A)$ such that $\sigma(x) \in \underline{T} - \textup{cl}_{\widetilde{M}} \big(\sigma(A)\big)$. As $\textup{cl}_{\widetilde{M}} \big(\sigma(A)\big)$ is the intersection of all hyperplanes of $\widetilde{M}$ that contain $\sigma(A)$, there exists a hyperplane $H$ of $\widetilde{M}$ such that $\sigma(A) \subseteq H$ and $\sigma(x) \notin H$. Thus $Z:= \underline{T} - H$ is a cocircuit of $\widetilde{M}$ with $\sigma(x) \in Z$ and $\sigma(A) \cap Z  = \emptyset$. 

As $\mathcal{C}^*_{\widetilde{M}} = \mathcal{C}^*_M$, it follows that $\sigma^{-1}(Z) \in \mathcal{V}^*_N$, which implies that there exists a cocircuit $Y$ of $N$ such that $x \in Y \subseteq \sigma^{-1}(Z)$. Thus $L := S - Y$ is a hyperplane of $N$ with $x \notin L$. Because $A \subseteq L$, it follows that $\textup{cl}_N(A) \subseteq L$ and, consequently, $x \notin \textup{cl}_N(A)$, which is a contradiction.
\end{proof}

\begin{cor} \label{duality - classical strong maps}
Let $\sigma: \underline{S} \rightarrow \underline{T}$ be an $\F_1$-linear map. Then the following are equivalent:
\begin{enumerate}
    \item \label{duality - classical strong maps 1} $\sigma: N \rightarrow M$ is a strong map;
    \item \label{duality - classical strong maps 2} $\sigma^t: M^* \rightarrow N^*$ is a strong map;
    \item \label{duality - classical strong maps 3} $\sigma(\cC_N) \perp \cC^*_{M}$.
\end{enumerate}
\end{cor}

\begin{proof} 
As $\sigma$ is $\F_1$-linear
\[
\#\big(X \cap \sigma^{-1}(Y) \big) = \#\big((\sigma^t)^{-1}(X) \cap Y\big)
\]
for all $X \subseteq S$ and $Y \subseteq T$. In particular, $\sigma^{-1}(\mathcal{C}^*_M) \perp \mathcal{C}_N$ if, and only if, $\mathcal{C}_{M^*} \perp (\sigma^\mathfrak{t})^{-1}(\mathcal{C}^*_{N^*})$. Thus by \autoref{cryptomorphism - classical strong maps}, we conclude that \eqref{duality - classical strong maps 1} is equivalent to \eqref{duality - classical strong maps 2}.

Now, note that $\sigma(X) - \{0_T\} = (\sigma^\mathfrak{t})^{-1}(X)$ for each $X \subseteq S$. Therefore the equivalence between (\ref{duality - classical strong maps 2}) and (\ref{duality - classical strong maps 3}) follows from \autoref{cryptomorphism - classical strong maps}.
\end{proof}

\subsection{Baker-Bowler theory}
\label{Baker-Bowler theory}
We recall the theory of matroids with coefficients introduced by Baker and Bowler in \cite{Baker-Bowler19}.

\medskip\noindent\textbf{Idylls.} 
A \emph{pointed abelian group} is a (multiplicatively written) commutative monoid $F$ with a distinguished element $0$ such that $F^\times=F-\{0\}$ and $0\cdot a=0$ for all $a\in F$. An \emph{idyll} is a pointed abelian group $F$ together with a subset $N_F$ (called the \emph{nullset}) of the group semiring $\N[F^\times]$ that satisfies the following axioms:
    \begin{enumerate}[label=({I\arabic*})]
     \item $N_F$ is a proper ideal of $\N[F^\times]$, i.e. $0\in N_F$, \ $1 \notin N_F$, \ $a+b\in N_F$ and $ac\in N_F$ for all $a,b\in N_F$ and $c\in\N[F^\times]$.\label{idyll: T3}
     \item There is a unique element $-1 \in F^\times$ with $1+(-1) \in N_F$; \label{idyll: T2}
\end{enumerate}
Note that these axioms imply that every element $a\in F$ has an \emph{additive inverse} $-a=(-1)\cdot a$, which is the unique element with $a+(-a)\in N_F$. This further implies that $(-1)^2=1$ and $a\in N_F$ only if $a=0$.

A \emph{morphism of idylls} is a multiplicative map $f: F \rightarrow G$ with $f(0)=0$ and $f(1)=1$ such that $\sum f(a_i)\in N_G$ for all $\sum a_i\in N_F$. This defines the category $\Idylls$.
    
\begin{ex}[Fields as idylls]\label{exa: fields as idylls}
 A field $k$ is naturally an idyll if we replace its addition by the nullset
    $$ \textstyle
        N_k:=\left\{ \sum a_i \in \N[k^\times] \mid \sum a_i = 0 \text{ in } k\right\}.
    $$
 By abuse of notation, we denote this idyll by the same symbol $k$. More to the point, the category of fields embeds as a full subcategory into $\Idylls$, which allows us to consider fields as particular types of idylls.   
\end{ex}
    
\begin{ex}\label{exam: regular partial field}
The \emph{regular partial field} is the multiplicative pointed group 
$$
\F_1^{\pm}=\{0,1,-1\} \quad \text{ with nullset }  \quad  
N_{\F_1^{\pm}}=\{m\cdot 1 + n\cdot (-1) \mid m=n\}.
$$
This idyll is the initial object of the category $\Idylls$. Given an idyll $F$, the unique morphism $\F_1^\pm\rightarrow F$ is determined by $1_{\F_1^\pm} \mapsto 1_F$ and $-1_{\F_1^\pm}\mapsto -1_F$.
\end{ex}        

\begin{ex}[Hyperfields as idylls] Roughly speaking, a \emph{hyperfield} is a field where the standard addition is replaced by a multivaluated \emph{hyperaddition}, see \cite{Davvaz2007} for details. Given a hyperfield $H$ with hyperaddition $\boxplus : H \times H \rightarrow 2^H$, we consider $H$ as an idyll with respect to the nullset $N_H := \big\{ \sum a_i \in \N[H^\times] \mid 0 \in \hypersum a_i\big\}$.

Some examples of interest are:
\begin{itemize}
    \item The \emph{Krasner hyperfield} is $\K = \{0, 1\}$ with hyperaddition given by $a\boxplus 0 = \{a\}$ and $1\boxplus 1 = \{0,1\}$. As an idyll, $-1_{\K} = 1$ and its nullset is $\N - \{1\}$. This hyperfield is the final object in the category $\Idylls$.
    \item The \emph{tropical hyperfield} is $\T = \R_{\geq 0}$ with the usual multiplication, whose hyperaddition is given by $a \boxplus b = \textup{max}\{a, b\}$ for $b \neq a$, and $a \boxplus a = [0, a]$. As an idyll, $-1_{\T} = 1$ and its nullset is 
    \[\textstyle
    N_\T = \big\{ \sum a_i \, \big| \, \text{ the maximum of } \{a_i\} \text{ appears at least twice}\big\}.
    \]
    \item The \emph{phase hyperfield} is $\P = S^1 \cup \{0\}$, where $S^1 = \{z \in \C \mid z \text{ has norm } 1\}$, with usual multiplication and hyperaddition given by
    \[
    a \boxplus b = \left\{ \frac{\alpha a + \beta b}{||\alpha a + \beta b ||} \; \bigg| \; \alpha,\beta \in \R_{>0}\right\}
    \]
    for $b \neq -a$, and $a \boxplus -a = \{-a, 0, a\}$. As an idyll, $-1_{\P} = -1$ and its nullset is
    $$ \textstyle
        N_{\P}=\left\{\sum x_i \in \N[S^1] \;\big|\; \sum \alpha_i x_i = 0 \text{ in } \C, \text{ for some } \alpha_i \text{ in } \R_{>0}\right\}.
    $$ 
\end{itemize}
\end{ex}

\begin{rem}
     Historically, the notion of \emph{hypergroup} was introduced by Marty in 1934 \cite{Marty1934}. Around 1940, more general aspects of \emph{hyperstructure} were considered, when Krasner defined the notion of \emph{hyperring} and \emph{hyperfield}. For instance, Jun develop algebraic geometry over hyperrings in \cite{JUN2018}, and the second author studies tropical geometry over the tropical hyperfield in \cite{Lorscheid19}.
\end{rem}

\begin{notation}
\label{notation - tuples}
For an idyll $F$, a set $E$, an integer $r$, a tuple ${\bf x} = (x_1,\dots,x_r) \in E^r$ and a permutation $\sigma\in \mathfrak{S}_r$, we define:
\begin{itemize}
    \item $|{\bf x}|:=\{x_1,\dots,x_r\} \subseteq E$; 
    \item the sign $\sign(\sigma)\in \{\pm1\}$, considered as an element of $F$;
    \item  ${\bf x}^\sigma:=(x_{\sigma(1)},\dots,x_{\sigma(r)})$;
    \item ${\bf x}_{\widehat{k}}:=(x_1,\dots,\widehat{x}_k,\dots,x_r) \in E^{r-1}$.
\end{itemize}
\end{notation}

\medskip\noindent\textbf{Matroids over idylls.} Let $F$ be an idyll, $E$ a finite set and $r$ an integer. A \emph{Grassmann-Pl\"ucker function} of rank $r$ on $E$ with coefficients in $F$ is a nonzero function $\varphi: E^r \rightarrow F$ satisfying the following two properties:
\begin{enumerate}[label=({GP\arabic*})]
 \item $\varphi$ is \emph{alternating}, i.e. $ \varphi({\bf x})  = \sign(\sigma)\cdot \varphi({\bf x}^{\sigma})$ for all $\sigma\in \mathfrak{S}_r$ and $\varphi({\bf x})=0$ whenever $\#\norm\bx<r$;\label{def: mat GP1}
 \item for all $x_2,\dotsc,x_r,y_0,\dotsc,y_r\in E$,
    $$
        \sum_{k=0}^{r} \ (-1)^k \cdot \varphi(y_0,\dotsc,\widehat{y_k},\dotsc,y_r)\cdot\varphi(y_k,x_2,\dotsc,x_r) \ \in \ N_F.
    $$ \label{def: mat GP2}
\end{enumerate}
The relations in \ref{def: mat GP2} are called \emph{Pl\"ucker relations}. Two Grassmann-Pl\"ucker functions $\varphi_1$ and $\varphi_2$ are \emph{equivalent} if 
$\varphi_1 = a\cdot\varphi_2$ for some $a$ in $F^\times$. An \emph{$F$-matroid} of rank $r$ on $E$ over $F$ is an equivalence class $[\varphi]$ of a Grassmann-Pl\"ucker function $\varphi: E^r \to F$. 

\begin{ex}
\label{ex: usual matroid and Krasner-matroid} 
Given a $\K$-matroid $M = [\varphi]$ of rank $r$ on the set $E$, the collection
$$
    \cB = \big\{\{x_1,\dots,x_r\} \subseteq E \mid \varphi(x_1,\dots,x_r) \neq 0\big\}
$$
is the set of bases of a classical matroid, called the \emph{underlying matroid} of $M$ and denoted by $\underline{M}$. By \cite[Remark 3.19]{Baker-Bowler19}, the map
\[
\begin{array}{ccc}
\{\K\text{-matroids on }E\} &\longrightarrow& \{\text{classical matroids on }E\}\\
M & \longmapsto & \underline{M}
\end{array}
\]
is a bijection. This shows that matroids over idylls are generalizations of classical matroids. 
\end{ex}

\medskip\noindent\textbf{Functoriality.}
Given a morphism of idylls $f: F \rightarrow G$ and a Grassmann-Pl\"ucker function $\varphi: E^r \rightarrow F$ with coefficients in $F$, the composition $f \circ \varphi: E^r \rightarrow G$ is a Grassmann-Pl\"ucker function with coefficients in $G$. For an $F$-matroid $M = [\varphi]$, we define the \emph{push-forward} of $M$ \emph{along} $f$ as the $G$-matroid $f_*(M) := [f \circ \varphi]$.

\begin{df}
Given an $F$-matroid $M$, its \emph{underlying matroid} is the classical matroid $\underline{M} := \underline{f_*(M)}$, where $f:F\to \K$ the unique idyll morphism into $\K$. 
\end{df}

\medskip\noindent\textbf{Duality.} Let $F$ be an idyll and $M=[\varphi]$ an $F$-matroid of rank $r$ on the set $E$. Fix a total order on $E$. For ${\bf y} \in E^n$ with $|\mathbf{y}| = E$, let $\sigma$ be the permutation for which $y_{\sigma(1)}<\cdots<y_{\sigma(n)}$, and define $\sign({\bf y}) := \sign(\sigma) \in \{1, -1\}$.

The \emph{dual} of $M$ is the $F$-matroid $M^* := [\varphi^*]$, where $\varphi^*: E^{n-r}\rightarrow F$ is the Grassmann-Pl\"ucker function given by
\[
\varphi^*({\bf x}) \ := \ \begin{cases}
                        0, \text{ if there exists } i \neq j \text{ with } x_i = x_j\\
                        \sign({\bf x},{\bf x}')\cdot \varphi({\bf x}'), \text{ otherwise},
                        \end{cases}
\]
where ${\bf x}' \in E^{r}$ is such that $E=|{\bf x}| \sqcup |{\bf x}'|$. The dual of $M$ is independent of the order on $E$ and satisfies $(M^*)^* = M$.

\medskip\noindent\textbf{Circuits and vectors.} 
Let $X=(X_e)_{e\in E}$ be an element of $F^E$. The \emph{support} of $X$ is the set $\underline{X} := \{e\in E \mid X_e \neq 0\}$. More generally, for a subset $\cS$ of $F^E$, its support is $\underline{\cS}:=\{\underline{X}\mid X \in \cS\}$. 

The \emph{inner product} of two elements $X$ and $Y$ of $F^E$ is defined as
\[
 X\cdot Y \ = \ \sum X_e\cdot Y_e \in \N[F^\times].
\]
We say that $X$ is \emph{orthogonal} to $Y$, denoted by $X\perp Y$, if $X\cdot Y$ is in $N_F$. Two subsets $\cX$ and $\cY$ of $F^E$ are \emph{orthogonal}, written as $\cX \perp \cY$, if $X\perp Y$ for each $X$ in $\cX$ and $Y$ in $\cY$. The \emph{orthogonal complement} of a subset $\cS \subseteq F^E$ is $\cS^\perp:=\big\{ X \in F^E \mid \{X\} \perp \cS \big\}$. 

Let $M=[\varphi:E^r\to F]$ be an $F$-matroid and define $\Xi_M$ as the collection of all $\by=(y_0,\dotsc,y_r)\in E^{r+1}$ for which $\#\norm\by=r+1$ and for which $\varphi(y_0,\dotsc,\widehat{y_k},\dotsc,y_r)\neq0$ for some $k=0,\dotsc,r$. We define the \emph{fundamental circuit of $\by\in\Xi_M$ with respect to $\varphi$} as
\[
 X_{\by}(x) \ = \ \begin{cases}
                    (-1)^i\cdot\varphi(y_0,\dotsc,\widehat{y_i},\dotsc,y_r) & \text{if }x=y_i, \\
                    0                                                       & \text{otherwise}.
                    \end{cases}
\]
The set of \emph{$F$-circuits} of $M$ is 
$$
    \cC_M \ = \ \big\{ a\cdot X_{\by} \, \big| \, a \in F^\times, \ \by\in \Xi_M \big\}.
$$
The set of \emph{$F$-cocircuits} of $M$ is $\cC^*_M := \cC_{M^*}$. We define the set of \emph{$F$-vectors} and of \emph{$F$-covectors} of $M$ as
$$
    \cV_M:= (\cC^*_M)^\perp \quad \text{and} \quad \cV^*_M:=\cC_M^\perp,
$$
respectively. The circuits of $\underline M$ are exactly the supports 
\[
\underline{X} \ = \ \{e\in E\mid X_e\neq0\}
\]
of the $F$-circuits $X\in \cC_M$ of $M$, i.e. the elements of $\underline{\cC_M}$.

An idyll $F$ is \emph{perfect} if $\cV_M \perp \cV^*_M$ for all $F$-matroids $M$. Examples of perfect idylls are: fields, the regular partial field $\F^\pm_1$, and the hyperfields $\K$ and $\T$. The phase hyperfield $\P$ is not perfect. For more details, see \cite{Baker-Bowler19}.

\begin{ex}
Let $k$ be a field, $E$ a finite set and $r$ an integer. By \cite[Example 3.30]{Baker-Bowler19}, the following maps are bijections:
\[
\begin{array}{ccccc}
\left\{\begin{matrix} \text{subspaces of } k^E\\ \text{ of dimension } r \end{matrix}\right\}  
& \longleftarrow 
& \left\{\begin{matrix} k\text{-matroids on } $E$ \\ \text{ of rank } r
        \end{matrix}\right\} 
        & \longrightarrow 
        & \left\{\begin{matrix} \text{subspaces of } k^E \text{ of}\\ \text{dimension } \#E-r \end{matrix}\right\}\\
        &&\\
\mathcal{V}^*_M = (\mathcal{V}_M)^\perp & \longmapsfrom & M & \longmapsto & \mathcal{V}_M.
\end{array}
\] 
\end{ex}

\medskip\noindent\textbf{Minors.}
Let $M = [\varphi]$ be an $F$-matroid of rank $r$ on $E$ and $A \subseteq E$ a subset.
\begin{enumerate}
    \item (Contraction) Let $\ell$ be the rank of $A$ in $\underline{M}$ and $\{a_1, \dotsc, a_\ell\}$ a maximal independent subset of $A$. We define $\varphi/A: E^{r-\ell} \rightarrow F$ by
    \[
    (\varphi/A) (x_1,\dotsc,x_{r-\ell}) \ = \ \varphi(x_1,\dotsc,x_{r-\ell},a_1, \dotsc, a_\ell).
    \]
    \item (Deletion) Let $k$ be the rank of $E\backslash A$ in $\underline{M}$ and $\{b_1, \dotsc, b_{r-k}\} \subseteq A$ a basis of $\underline{M} / (E \backslash A)$. We define $\varphi \backslash A: (E\backslash A)^k \rightarrow F$ by
    \[
    (\varphi \backslash A)(y_1,\dotsc,y_k) \ = \ \varphi(y_1,\dotsc,y_k,b_1, \dotsc, b_{r-k}).
    \]
    \item (Restriction) Let $q$ be the rank of $A$ in $\underline{M}$. Define $\varphi|A := \varphi \backslash (E \backslash A): A^q \rightarrow F$.
\end{enumerate}
By \cite[Lemma 4.4]{Baker-Bowler19}, $\varphi/A$, $\varphi \backslash A$ and $\varphi|A$ are Grassmann-Pl\"ucker functions, and the $F$-matroids $M/A:=[\varphi/A]$ (\emph{contraction} of $M$ by $A$), $M\backslash A := [\varphi \backslash A]$ (\emph{deletion} of $A$ from $M$) and $M|A := [\varphi |A]$ (\emph{restriction} of $M$ to $A$) do not depend on any choices. By \cite[Thm. 3.29]{Baker-Bowler19}, $(M/A)^* = M^* \backslash A$ and $(M \backslash A)^* = M^* / A$.

\section{Category of \texorpdfstring{$F$}{F}-matroids}
\label{section: category of F-matroids}

\subsection{Submonomial matrices}
\label{subsection: submonomial matrices}

Let $F$ be a perfect idyll and $T$ and $S$ finite sets. 

\begin{df}
 An {$F$}\emph{-matrix} (indexed by $T\times S$) is a matrix $\Phi=(\Phi_{i,j})_{(i,j)\in T\times S}$ with entries $\Phi_{i,j}\in F$. An $F$-matrix $\Phi$ is \emph{submonomial} if every row and column of $\Phi$ contains at most one nonzero entry. We denote the \emph{transpose} of $\Phi$ by $\Phi^t$, which has entries $\Phi^t_{i,j} :=\Phi_{j,i}$.
\end{df}

Note that the product $\Phi \cdot X$ of a submonomial $F$-matrix $\Phi$ indexed by $T\times S$ with $X \in F^S$ is a well-defined element of $F^T$, since the sum in the definition of $(\Phi \cdot X)_i=\sum_{j\in S} \Phi_{i,j} X_j$ contains at most one nonzero term. Due to this, we also consider a submonomial $F$-matrix $\Phi$ as a map $\Phi: F^S \to F^T$. For a subset $\mathcal{X} \subseteq F^S$, we use $\Phi \cdot \mathcal{X}$ to denote the image $\{\Phi\cdot X \mid X \in \mathcal{X}\} \subseteq F^T$.

Let $\Phi$ and $\Psi$ be submonomial $F$-matrices indexed by $T\times S$ and $W\times T$, respectively. The product $\Psi\cdot\Phi$ is well-defined and corresponds to the composition $\Psi \circ \Phi: F^S\to F^W$ of the maps $\Phi:F^S\to F^T$ and $\Psi:F^T\to F^W$. The transpose $\Phi^t$ of $\Phi$ is also submonomial, and thus defines a map $\Phi^t:F^T\to F^S$.

\begin{df}
\label{def: underlying map of submonomial matrix}
The \emph{underlying map} of a submonomial $F$-matrix $\Phi$ indexed by $T \times S$ is the $\F_1$-linear map $\uline\Phi:\uline{S}\to \uline{T}$ given by
\[
\uline\Phi(j) \ := \ \begin{cases}
                    i & \text{if $j\neq0$ and $\Phi_{i,j}\neq 0$;}\\
                    0 & \text{if $j=0$, or if $j\neq0$ and $\Phi_{i,j}=0$ for all $i\in T$.}
                    \end{cases}
\]
which satisfies $(\underline{\Phi})^t = \underline{(\Phi^t)}$ (see \autoref{F1-vector spaces}). For $s \in \underline{S}$, we define
\[
\Phi_s \ := \ \begin{cases}
            \Phi_{\underline{\Phi}(s), s} & \text{if $s \neq 0$ and $\Phi_{\underline{\Phi}(s),s} \neq 0$;}\\
            0 &\text{if $s=0$, or if $s\neq0$ and $\Phi_{i,s}=0$ for all $i\in T$.}
            \end{cases}
\]
\end{df}

\begin{rem}\label{conta tecnica}
Note that since $\Phi$ is submonomial, there is at most one $i$ for every given $j$ such that $\Phi_{i,j}\neq 0$. Note further that $(\Phi^t \cdot Z)_s = \Phi_s \cdot Z_{\underline{\Phi}(s)}$ for $Z \in F^T$ and $s \in S$.
\end{rem}

\subsection{Morphisms of matroids with coefficients}
\label{subsection: morphisms of matroids with coefficients} 

\begin{df}[Category of $F$-matroids]
Let $N$ and $M$ be $F$-matroids with respective underlying sets $S$ and $T$. A \emph{morphism} of $F$-matroids from $N$ to $M$ is a submonomial $F$-matrix $\Phi$ indexed by $T \times S$ such that $\Phi \cdot \cV_N \subseteq \cV_M$. We write $\Phi: N \to M$ for a morphism $\Phi$ from $N$ to $M$. This defines the category $\Mat_F$ of $F$-matroids.
\end{df}

\begin{ex}\label{ex: matroid morphisms}
 A morphism of $\K$-matroids is a particular type of strong map, which includes isomorphisms, restrictions, contractions and quotients, but not the identification of parallel elements. A rigorous discussion of this relation can be found in \autoref{subsection: Morphisms of K-matroids and strong maps}.

 In the case of a field $k$, a submonomial matrix $\Phi$ describes a linear map $\Phi: k^S \rightarrow k^T$ that sends coordinate vectors of $k^S$ (i.e.\ vectors with only one nonzero coefficient) to coordinate vectors of $k^T$ or $0$. Therefore a morphism $\Phi:N\to M$ of $k$-matroids (with respective ground sets $S$ and $T$) is a submonomial linear map $\Phi: k^S \rightarrow k^T$ with $\Phi\cdot\cV_N\subseteq\cV_M$, where $\cV_N$ and $\cV_M$ are the vector sets of $M$ and $N$, respectively.
 
 Similarly, a morphism $\Phi:N\to M$ of $\T$-matroids (with respective ground sets $S$ and $T$) is a submonomial linear map $\T^S\to\T^T$ such that $\Phi\cdot\cV_N\subseteq\cV_M$.
 
 A quotient $M$ of an oriented matroid $N$ is the same thing as an $\S$-morphism $\Phi:N\to M$ that is given by a square identity matrix $\Phi$. Examples of other $\S$-morphisms are
 \[
  \smallmat 01{-1}0: M \to M' \qquad \text{and} \qquad \smallmat1000: M \to N
 \]
 where $M$, $M'$ and $N$ are the $\S$-matroids with respective vector sets
 \[\textstyle
  \cV_M \ = \ \big\{ \smallvector00, \smallvector 11, \smallvector{-1}{-1} \big\}, \quad \cV_{M'} \ = \ \big\{ \smallvector00, \smallvector 1{-1}, \smallvector{-1}1 \big\}, \quad \cV_N \ = \ \big\{ \smallvector00, \smallvector 10, \smallvector{-1}0 \big\}.
 \]
 Thus the underlying matroid of both $M$ and $M'$ is $U^1_2$ and the underlying matroid of $N$ is $U^1_1\oplus U^0_1$.
\end{ex}
 
Recall that the inner product of two tuples $X,Y\in F^E$ is defined as the element $X\cdot Y=\sum_{e\in E} X_eY_e$ of $\N[F^\times]$.
 
\begin{lemma}
\label{inner product}
Let $\Phi$ be a submonomial $F$-matrix indexed by $T \times S$. Let $X \in F^S$ and $Y \in F^T$. Then $(\Phi\cdot X)\cdot Y = X \cdot (\Phi^t \cdot Y)$.
\end{lemma}

\begin{proof}
 This follows from a direct computation:
 \[
  (\Phi\cdot X)\cdot Y  \ = \ 
  \sum_{e\in T} \Phi_{e,\underline\Phi^t(e)} \cdot X_{\underline\Phi^t(e)} \cdot Y_{e} \ = \ 
  \sum_{e\in S} X_e\cdot\Phi^t_{e,\underline{\Phi}(e)} \cdot Y_{\underline\Phi(e)} \ = \ 
  X\cdot (\Phi^t\cdot Y).
  \qedhere
 \]
\end{proof}

\begin{df}
\label{df: pointed Grassmann-Plucker function}
Given a function $\mu: E^r \rightarrow F$ into a perfect idyll $F$, we define
\[
\begin{array}{clll}
\widetilde{\mu}: & \underline{E}^r & \longrightarrow & \quad \; F \\
& {\bf x} & \longmapsto & \begin{cases}
                        \mu({\bf x}) &\text{if }  x_i \neq 0 \text{ for all } i\\
                        0 &\text{otherwise}.
                        \end{cases}     
\end{array}
\]
If $M = [\mu]$ is an $F$-matroid on $E$, then $\widetilde{M} := [\widetilde{\mu}]$ is an $F$-matroid on $\underline{E}$ with $\underline{\widetilde{M}} = \underline{M} \oplus U_1^0$, i.e. $\widetilde{M}$ is obtained from $M$ by adding a loop (cf. \autoref{classical matroid extension by 0}).
\end{df}

\begin{thm}\label{thm: cryptomorphisms}
Let $N = [\nu]$ be an $F$-matroid of rank $w$ on $S$. Let $M = [\mu]$ be an $F$-matroid of rank $r$ on $T$ and $\Phi$ a submonomial $F$-matrix indexed by $T\times S$. Then the following are equivalent:
 \begin{enumerate}
  \item \label{crypto1} $\Phi$ is a morphism from $N$ to $M$;
  \item \label{crypto2} $\Phi \cdot \cC_N\perp\cC_M^\ast$;
  \item \label{crypto3} for all $y_0,\dotsc,y_w\in S$ and $x_2,\dotsc,x_r \in T$,
  \[
   \sum_{k=0}^w \ (-1)^k \cdot \nu(y_0,\dotsc,\widehat{y_k},\dotsc,y_w) \cdot {\Phi_{\uPhi(y_k),y_k}} \cdot \widetilde{\mu} \big(\underline{\Phi}(y_k),x_2,\dotsc,x_r\big) \ \in \ N_F.
  \]
 \end{enumerate}
\end{thm}

\begin{proof}
We first prove the equivalence between \eqref{crypto1} and \eqref{crypto2}. Assume \eqref{crypto1}. As $\cC_N \subseteq \cV_N$, it follows that $\Phi \cdot \cC_N \subseteq \cV_M = (\cC_M^*)^\perp$, which implies \eqref{crypto2}.

Assume \eqref{crypto2}. By \autoref{inner product}, $\cC_N \perp \Phi^t \cdot \cC_M^*$, which implies $\cV_N^* = (\cC_N)^\perp \supseteq \Phi^t \cdot \cC_M^*$. As $F$ is perfect, one has $\cV_N = (\cV_N^*)^\perp \subseteq (\Phi^t \cdot \cC_M^*)^\perp$, i.e. $\cV_N \perp \Phi^t \cdot \cC_M^*$. Again by \autoref{inner product}, $\Phi\cdot \cV_N \perp \cC_M^*$, which implies $\Phi\cdot \cV_N \subseteq (\cC_M^*)^\perp = \cV_M$. Thus \eqref{crypto1} follows.

Next we prove the equivalence between \eqref{crypto2} and \eqref{crypto3}. Assume \eqref{crypto3}. Let $Z \in \mathcal{C}^*_M = \mathcal{C}_{M^*}$. By \cite[p.\ 841]{Baker-Bowler19}, there exists $z_0 \in \underline{Z}$, a basis $\{z_1, \dotsc, z_{n-r}\}$ of $\underline{M}^*$ containing $\underline{Z} - z_0$, and $\alpha \in F^\times$ such that
\[
 Z_{z_i} = \alpha\cdot (-1)^i \cdot \mu^*(z_0, \dotsc, \widehat{z_i}, \dotsc, z_{n-r})
\]
for all $i$. Let $\textbf{z} := (z_0, \dotsc, z_{n-r}) \in T^{n-r+1}$ and $\textbf{x} \in T^{r-1}$ such that $|{\bf z}| \sqcup |{\bf x}| = T$. Similarly, for $Y \in \mathcal{C}_N$, there exist $y_0 \in \underline{Y}$, a basis $B = \{y_1, \dotsc, y_w\}$ of $\underline{N}$ containing $\underline{Y} - y_0$, and $\chi \in F^{\times}$ such that
\[
Y_{y_j} = \chi\cdot (-1)^j \cdot \nu({\bf y}_{\widehat{j}})
\]
for all $j = 0, \dotsc, w$, where $\textbf{y} := (y_0, \dotsc, y_w)$. Let $U := |{\bf z}| \cap \underline{\Phi}\big(|{\bf y}|)$ and define
\[
  I \ := \ \big\{i \in \{0,\dotsc,n-r\}\, \big| \, z_i\in U\big\} \qquad \text{and} \qquad J \ := \ \big\{j \in \{0,\dotsc,w\} \,\big|\, \underline{\Phi}(y_j) \in U\big\}.  
\]
As $\underline{\Phi}$ is $\F_1$-linear, it induces a bijection $\mathfrak{b}: J\rightarrow I$ such that $\underline{\Phi}(y_j) = z_{\mathfrak{b}(j)}$, for all $j\in J$. By \autoref{inner product} and \autoref{conta tecnica}, we have (using the shorthand notation from \autoref{notation - tuples} and \autoref{def: underlying map of submonomial matrix}):

\[
\begin{aligned}
      & (\chi {\alpha})^{-1} \cdot \big(Z \cdot (\Phi \cdot Y) \big) = (\chi {\alpha})^{-1} \cdot \big(Y \cdot (\Phi^t \cdot Z) \big)
\\      
        = & (\chi {\alpha})^{-1} \cdot \underset{s \in S}{\sum} Y_s \cdot (\Phi^t \cdot Z)_s = (\chi {\alpha})^{-1} \cdot \underset{j \in J}{\sum} \  Y_j \cdot (\Phi^t \cdot Z)_j
\\
        = & (\chi \alpha)^{-1} \cdot \sum_{j \in J} {\bigl(} {\chi\cdot (-1)^j \nu({\textbf{y}}_{\widehat{j}})}{\bigr)} \cdot {\bigl(}\Phi_{y_j} \cdot \alpha \cdot (-1)^{\fb(j)}\mu^*(\bfz_{\widehat{\fb(j)}}){\bigr)}
\\
       = & \underset{j \in J}{\sum} \ (-1)^{\big(j + \fb(j) \big)} \cdot \nu({\textbf{y}}_{\widehat{j}}) \cdot {\Phi_{y_j}} \cdot \sign({\textbf{z}}_{\widehat{{\mathfrak{b}(j)}}}, z_{\mathfrak{b}(j)},\bfx) \cdot \mu( z_{\mathfrak{b}(j)}, \bfx)
\\
      = & \underset{j\in J}{\sum} \ (-1)^{\big( j + \text{$\mathfrak{b}$}(j)\big)} \cdot \sign({\bf z}_{\widehat{\text{$\mathfrak{b}$}(j)}}, z_{\text{$\mathfrak{b}$}(j)}, {\bf x}) \cdot \nu({\bf y}_{\widehat{j}}) \cdot {\Phi_{y_j}} \cdot \mu( \underline{\Phi}(y_j), {\bf x}),
\end{aligned}
\]
because $Y_s \cdot (\Phi^t \cdot Z)_s = 0$ for $s \notin J$. Furthermore, for all $j\in J$ it follows that 
\[
\begin{aligned}
      (-1)^{(n-r)} \cdot \sign({\bf z},{\bf x}) & = (-1)^{2(n-r) - \mathfrak{b}(j)} \cdot \sign({\bf z}_{\widehat{\mathfrak{b}(j)}}, z_{\mathfrak{b}(j)}, {\bf x})
      \\
       &=  (-1)^{\mathfrak{b}(j)} \cdot \sign({\bf z}_{\widehat{\mathfrak{b}(j)}},z_{\mathfrak{b}(j)},{\bf x}).
\end{aligned}
\]
Thus
\[
\begin{aligned}
      & (-1)^{(n-r)} \cdot \sign({\bf z},{\bf x}) \cdot (\chi {\alpha})^{-1} \cdot\big( Y \cdot (\Phi^t\cdot Z) \big)
\\
      = & \underset{j \in J}{\sum} (-1)^j \cdot \nu({\bf y}_{\widehat{j}}) \cdot {\Phi_{y_j}} \cdot \mu\big( \underline{\Phi}(y_j),{\bf x}\big)
\\
      = & \underset{j=0}{\overset{w}{\sum}} (-1)^j \cdot \nu({\bf y}_{\widehat{j}}) \cdot {\Phi_{y_j}} \cdot \widetilde{\mu}\big(\underline{\Phi}(y_j),{\bf x}\big).
\end{aligned}
\]
This implies that $Z \cdot (\Phi \cdot Y) = Y \cdot (\Phi^t \cdot Z) \in N_F$. Therefore \eqref{crypto2} follows.

Assume \eqref{crypto2}. Let ${\bf y} =(y_0, \dotsc, y_w) \in S^{w+1}$ and ${\bf x} = (x_1, \dotsc, x_{r-1}) \in T^{r-1}$. There are two cases to analyze:

\medskip\noindent\textbf{Case 1.} 
If there is no $\ell$ in $\{0, \dotsc, w\}$ such that $\big\{\underline{\Phi}(y_\ell)\big\}\cup|{\bf x}|$ is a basis of $\underline{M}$ and $|{\bf y}_{\widehat{\ell}}|$ is a basis of $\underline{N}$, then $\nu({\bf y}_{\widehat{j}}) \cdot \widetilde{\mu}\big(\underline{\Phi}(y_\ell),{\bf x}\big) = 0$ for all $j \in \{0, \dotsc, w\}$. Thus

\[
\underset{k = 0}{\overset{w}{\sum}} (-1)^k ~
\nu({\bf y}_{\widehat{k}}) ~ {\Phi_{y_k}} ~ \widetilde{\mu} \big(\underline{\Phi}(y_k),{\bf x}\big) =  0 \ \in \ N_F.
\]

\medskip\noindent\textbf{Case 2.} 
If there is an $\ell$ in $\{0, \dotsc, w\}$ such that $\big\{\underline{\Phi}(y_\ell)\big\}\cup|{\bf x}|$ is a basis of $\underline{M}$ and $|{\bf y}_{\widehat{\ell}}|$ is a basis of $\underline{N}$, then $\{z_1, \dotsc, z_{n-r}\} := T-\big\{\{\underline{\Phi}(y_\ell)\}\cup |{\bf x}|\big\}$ is a basis of $\underline{M}^*$. Define $z_0 := \underline{\Phi}(y_\ell)$ and ${\bf z} := (z_0, \dotsc, z_{n-r}) \in T^{n-r+1}$. Then
\[
H_z \ = \ \begin{cases}
        (-1)^i \cdot \mu^*(\textbf{z}_{\widehat{i}}) & \text{if}\ z=z_i \text{ for some}\ i\in\{0, \dotsc, n-r\}\\
        0 &\text{otherwise},
        \end{cases}
\]
defines a circuit of $M^*$, and 
\[
G_y \ = \ \begin{cases}
        (-1)^j \cdot \nu({\bf y}_{\widehat{j}}) & \text{if}\ y=y_j \text{ for some}\ j\in\{0, \dotsc, w\}\\
        0 &\text{otherwise}, 
        \end{cases}
\]
defines a circuit of $N$.

Note that 
\[
H_{z_i} = (-1)^{i} \cdot \sign(\textbf{z}_{\widehat{i}}, z_i, {\bf x}) \cdot \mu(z_i, {\bf x}),
\]
for all $i \in \{0, \dotsc, n-r\}$. Let $U := |\textbf{z}| \cap \underline{\Phi}(|{\bf y}|)$ and define
\[
  I \ := \ \big\{i \in \{0,\dotsc,n-r\}\, \big| \, z_i\in U\big\} \qquad \text{and} \qquad J \ := \ \big\{j \in \{0,\dotsc,w\} \,\big|\, \underline{\Phi}(y_j)\in U\big\}.  
\]
There exists a bijection $\mathfrak{b}: J\rightarrow I$ such that $\underline{\Phi}(y_j)=z_{\mathfrak{b}(j)}$, for each $j\in J$. Because $G \cdot (\Phi^t \cdot H) = H \cdot (\Phi \cdot G) \in N_F$ (see \autoref{inner product}), it follows that
\[
\begin{aligned}
    \underset{k = 0}{\overset{w}{\sum}}  (-1)^k ~ \nu({\bfy}_{\widehat{k}}) ~ {\Phi_{y_j}} ~ \widetilde{\mu}\big(\underline{\Phi}(y_k), \bfx \big) 
    \; = \; & \underset{j\in J}{\sum}  (-1)^j ~ \nu(\textbf{y}_{\widehat{j}}) ~ {\Phi_{y_j}} ~ \mu(z_{\fb(j)}, {\bfx})
    \\
    =  \; & (-1)^{(n-r)} ~ \sign(\textbf{z}, \textbf{x}) ~ \underset{j\in J}{\sum} G_{y_j} ~ (\Phi^t \cdot H)_{z_{\fb(j)}} 
    \\
    = \; & (-1)^{(n-r)} ~ \sign(\mathbf{z},{\bf x}) ~ \underset{s \in S}{\sum} G_s ~ { (\Phi^t \cdot H)_s} \\
    =  \; & (-1)^{(n-r)} ~ \sign({\bf z}, {\bf x}) ~ G \cdot (\Phi^t \cdot H) \ \in \ N_F,
\end{aligned}
\]
which concludes the proof.  
\end{proof}

\begin{cor}
\label{duality for strong maps}
 Let $N$ be an $F$-matroid on $S$ and $M$ an $F$-matroid on $T$. Let $\Phi$ be a submonomial $F$-matrix indexed by $T\times S$. Then the map $\Phi$ is a morphism from $N$ to $M$ if, and only if, $\Phi^t$ is a morphism from $M^*$ to $N^*$.
\end{cor}

\begin{proof}
Let $Z \in \mathcal{C}^*_M$ and $Y \in \mathcal{C}_N$. By \autoref{inner product}, one has ${Y \cdot (\Phi^t \cdot Z)} = Z \cdot (\Phi \cdot Y)$. In particular, $Y \perp (\Phi^t \cdot Z)$ if, and only if, $Z \perp (\Phi \cdot Y)$. Thus the result follows from \autoref{thm: cryptomorphisms}.
\end{proof}

\begin{ex}[Quotients]
\label{ex: quotients}
 Let $N$ and $M$ be $F$-matroids on the same ground set $S=T$. If $\Phi = \id_S$ is the identity matrix indexed by $S \times S$, then the condition \eqref{crypto3} of \autoref{thm: cryptomorphisms} correspond to the \textit{flag Pl\"ucker relations} present in \cite[Def.~2.1]{Jarra-Lorscheid24}. Thus $\id_S$ is a morphism from $N$ to $M$ if, and only if, $M$ is a \textit{quotient} of $N$. In particular, this recovers quotients of classical matroids for $F = \K$, of oriented matroids for $F=\S$ and of valuated matroids for $F=\T$.
\end{ex}

\begin{ex}[Contractions and restrictions]
\label{ex: contractions and restrictions}
Let $M$ be an $F$-matroid on $E$ and $A\subseteq E$.
\begin{enumerate}
    \item \label{contraction} (Contraction)  Let $c_A := (\delta_{ij})_{(i,j) \in (E-A) \times E}$, where $\delta_{ij} = 1$ for $i = j$, and $\delta_{ij} = 0$ for $i \neq j$. Note that $c_A$ is a submonomial $F$-matrix. Let $V = (v_e)_{e \in E} \in \cV_M$. By \cite[Prop.~4.4]{Anderson19}, $c_A \cdot V = (v_e)_{e \in E - A}$ is an element of $\cV_{M / A}$. Therefore $c_A: M \rightarrow M/A$ is a morphism.
    \item \label{restriction} (Restriction) Let $r_A := (\delta_{ij})_{(i,j) \in E \times (E-A)}$. Then $r_A$ is a submonomial matrix and a morphism $r_A: M|A \to M$, which follows from \autoref{duality for strong maps} since $r_A=c_A^t$.
\end{enumerate}
\end{ex}

\subsection{Morphisms of \texorpdfstring{$\K$}{K}-matroids and strong maps}
\label{subsection: Morphisms of K-matroids and strong maps}

By \autoref{ex: usual matroid and Krasner-matroid}, the map
\[
\begin{array}{ccc}
\{\K\text{-matroids}\} &\longrightarrow& \{\text{classical matroids}\}\\
M & \longmapsto & \underline{M}
\end{array}
\]
is a bijection. Next we show a similar bijection between morphisms of $\K$-matroids and $\F_1$-linear strong maps of classical matroids.

\begin{lemma}
\label{classical strong maps}
Let $N$ and $M$ be $\K$-matroids on the sets $S$ and $T$, respectively, and $\Phi: N \rightarrow M$ a morphism. Then the $\F_1$-linear map $\underline{\Phi}: \underline{S} \rightarrow \underline{T}$ is a classical strong map $\underline{\Phi}: \underline{N} \rightarrow \underline{M}$.
\end{lemma}

\begin{proof}
Two tuples $X, Y$ in $\K^T$ are orthogonal if, and only if, $\#(\underline{X} \cap \underline{Y}) \neq 1$, i.e. if, and only if, $\underline{X} \perp \underline{Y}$ (cf. \autoref{subsection - Strong maps of classical matroids}). As $\underline{\mathcal{C}_P} = \mathcal{C}_{\underline{P}}$ for any $\K$-matroid $P$ and $\underline{\Phi \cdot Z} = \underline{\Phi}(\underline{Z}) \cap T$ for any $Z \in \K^S$, the Lemma follows from \autoref{thm: cryptomorphisms} and \autoref{duality - classical strong maps}.
\end{proof}

\begin{thm}
\label{thm: K-matroids and F1-linear spaces}
Let $N$ and $M$ be $\K$-matroids on the sets $S$ and $T$, respectively. Then the map
\[
\begin{array}{rcl}
\{\text{morphisms } N \rightarrow M\} &\longrightarrow& \{\F_1\text{-linear strong maps } \underline{N} \rightarrow \underline{M}\}\\
&&\\
\Phi: N \rightarrow M & \longmapsto & \underline{\Phi}: \underline{N} \rightarrow \underline{M}
\end{array}
\]
is a bijection.
\end{thm}

\begin{proof}
Note that for a morphism $\Phi: N  \rightarrow M$, 
\[
\Phi_{i,j} \ =  \ \begin{cases}
                1 &\text{if } \; \underline{\Phi}(j)  = i \neq 0\\
                0 &\text{otherwise.}
                \end{cases}
\]
Thus if $\Phi: N \rightarrow M$ and $\Phi': N \rightarrow M$ are morphisms with $\underline{\Phi} = \underline{\Phi'}$, then $\Phi = \Phi'$.

Fix an $\F_1$-linear strong map $\sigma: \underline{N} \rightarrow \underline{M}$ and, for $(i,j) \in T \times S$, let
\[
\Psi_{(i,j)} \ := \ \begin{cases}
                    1 &\text{if } \sigma(j) = i \neq 0\\
                    0 &\text{otherwise.}
                    \end{cases} 
\]
Note that $\Psi := (\Psi_{i,j})_{(i,j) \in T \times S}$ is a submonomial $\K$-matrix that satisfies $\underline{\Psi \cdot X} = \sigma(\underline{X}) \cap T$ for each $X$ in $\K^S$. As $\underline{\cC_N} = \cC_{\underline{N}}$ and $\underline{\cV_M} = \cV_{\underline{M}}$, by \autoref{duality - classical strong maps} and \autoref{thm: cryptomorphisms}, $\Psi: N \rightarrow M$ is a morphism. As $\underline{\Psi} = \sigma$, the result follows. 
\end{proof}

\subsection{Factorization of morphisms through quotients}
\label{subsection: pre-image}

In this section we prove \autoref{factorization of surjective strong maps through quotients}, which is the analogue of \autoref{factorization of classical strong map by quotient} for morphisms of $F$-matroids.

\begin{df}
Let $\Phi = (\Phi_{i, j})_{(i, j) \in T \times S}$ be a submonomial $F$-matrix and $\mu$ a Grassmann-Pl\"ucker function on the set $T$ with coefficients in $F$. Let $d$ be the rank of $\underline{[\mu]}|\big(\im(\underline{\Phi}) \cap T\big)$ and define $\Phi^{-1}(\mu): S^d \rightarrow F$ by
\[
\Phi^{-1}(\mu)(z_1,\dotsc,z_r) \ := \ \Big(\underset{\ell = 1}{\overset{d}{\prod}} \Phi_{\uPhi(z_\ell),z_\ell} \Big) \cdot \widetilde{\big(\mu|_{\im(\underline{\Phi}) \cap T}\big)}\big(\underline{\Phi}(z_1),\dots,\underline{\Phi}(z_d)\big).
\]
\end{df}

\begin{prop}
$\Phi^{-1}(\mu)$ is a Grassmann-Pl\"ucker function on $S$.
\end{prop}

\begin{proof}
Note that $\Phi^{-1}(\mu)$ is not identically zero when $d>0$ and is alternating. Let $\psi := \widetilde{\big(\mu|_{\im(\underline{\Phi}) \cap T}\big)}$, ${\bf y} \in S^{d+1}$ and ${\bf x} \in S^{d-1}$. For the purpose of this proof, if ${\bf z} = (z_1, \dotsc, z_m)$ is a tuple of elements of $S$, we write $\underline{\Phi}({\bf z})$ for $\big(\underline{\Phi}(z_1),\dotsc,\underline{\Phi}(z_m)\big)$. Then
\[
\begin{aligned}
    & ~ \sum_{k=1}^{d+1}(-1)^k ~ \Phi^{-1}(\mu)({\bf y}_{\widehat{k}}) ~ \Phi^{-1}(\mu)(y_k,{\bf x}) ~ 
\\
    = & ~ \sum_{k=1}^{d+1}(-1)^k ~ \Big(\underset{\ell \neq k}{\prod} \Phi_{y_\ell} \Big) ~ \psi\big(\underline{\Phi}({\bf y}_{\widehat{k}})\big) ~ \Phi_{y_k} ~ \Big(\underset{m = 1}{\overset{d-1}{\prod}} \Phi_{x_m} \Big) ~ \psi\big(\underline{\Phi}(y_k),\underline{\Phi}({\bf x})\big)
\\
    = & ~ \Big(\underset{\ell = 1}{\overset{d+1}{\prod}} \Phi_{y_\ell} \Big) ~ \Big(\underset{m = 1}{\overset{d-1}{\prod}} \Phi_{x_m} \Big) ~ \sum_{k=1}^{d+1}(-1)^k ~ \psi\big(\underline{\Phi}({\bf y}_{\widehat{k}})\big) ~ \psi\big(\underline{\Phi}(y_k),\underline{\Phi}({\bf x})\big),
\end{aligned}
\]
which is in $N_F$, because $\psi$ is a Grassmann-Pl\"ucker function.
\end{proof}

\begin{df}
\label{df: pre-image}
Let $M$ be the $F$-matroid $[\mu]$ on $T$. The \textit{pre-image} of $M$ under $\Phi$ is the $F$-matroid $\Phi^{-1}(M) := [\Phi^{-1}(\mu)]$ on $S$. 
\end{df}

\begin{rem}
Note that $\Phi^{-1}(M)$ is well defined, because $\alpha \Phi^{-1}(\mu) = \Phi^{-1}( \alpha \mu)$ for all $\alpha \in F^\times$.
\end{rem}

\begin{prop}
The underlying matroids $\underline{\Phi^{-1}(M)}$ and $\underline{\Phi}^{-1}(\underline{M})$ are equal.
\end{prop}

\begin{proof}
It follows from the fact that $\{z_1, \dotsc, z_d\}$ is a basis of $\underline{\Phi^{-1}(M)}$ if, and only if, $\{\underline{\Phi}(z_1), \dotsc, \underline{\Phi}(z_d)\}$ is a basis of $\underline{M}|\big(\im(\underline{\Phi})\cap T\big)$.
\end{proof}

\begin{thm}
\label{factorization of surjective strong maps through quotients}
Let $\Phi = (\Phi_{i,j})_{(i,j) \in T \times S}$ be a submonomial $F$-matrix, and $N$, $M$ two $F$-matroids with respective underlying sets $S$ and $T$. Then $\Phi: N \rightarrow M$ is a morphism if, and only if, $\Phi^{-1}(M)$ is a quotient of $N$.
\end{thm}

\begin{proof}
    Let $w$ be the rank of $N$ and $d$ the rank of $\Phi^{-1}(M)$. Assume that $\Phi: N \rightarrow M$ is a morphism. Let ${\bf y} \in S^{w+1}$ and ${\bf z} \in S^{d-1}$. Thus
\[
\begin{aligned}
    \underset{k = 1}{\overset{w+1}{\sum}} (-1)^k ~ \nu({\bf y}_{\widehat{k}}) ~ \Phi^{-1}(\mu)(y_k,{\bf z}) ~ = ~ \Big(\underset{\ell = 1}{\overset{d-1}{\prod}} \Phi_{z_\ell} \Big) ~ \underset{k = 1}{\overset{w+1}{\sum}} (-1)^k ~ \nu({\bf y}_{\widehat{k}}) ~ \Phi_{y_k} ~ \widetilde{\mu}\big(\underline{\Phi}(y_k), \underline{\Phi}({\bf z})\big)
\end{aligned}
\]
is in $N_F$, which implies that $\Phi^{-1}(M)$ is a quotient of $N$.

Conversely, assume that $\Phi^{-1}(M)$ is a quotient of $N$. The proof is divided into two cases.

\medskip\noindent\textbf{Case 1:} Assume $T \subseteq \im(\underline{\Phi})$. Let ${\bf y} \in S^{w+1}$ and ${\bf x} \in T^{d-1}$. Note that $\underline{\Phi^t}(\mathbf{x}) \in S^{d-1}$ satisfies $\underline{\Phi}\big(\underline{\Phi^t}(\mathbf{x})\big) = \mathbf{x}$.
\[
\begin{aligned}
    &\underset{k = 1}{\overset{w+1}{\sum}} (-1)^k ~ \nu({\bf y}_{\widehat{k}}) ~ \Phi_{y_k} ~ \widetilde{\mu}\big(\underline{\Phi}(y_k), {\bf x}\big) \\ ~ = ~ & \Big(\underset{\ell = 1}{\overset{d-1}{\prod}} \Phi_{\underline{\Phi^t}(x_\ell)} \Big)^{-1} ~ \underset{k = 1}{\overset{w+1}{\sum}} (-1)^k ~ \nu({\bf y}_{\widehat{k}}) ~ \Phi^{-1}(\mu)\big(y_k, \underline{\Phi^t}({\bf x})\big)
\end{aligned}
\]
is in $N_F$, which implies that $\Phi: N \rightarrow M$ is a morphism.

\medskip\noindent\textbf{Case 2:} Assume $T \not\subseteq \im(\underline{\Phi})$. Let $\Psi := \Phi|\big(S \rightarrow \im(\underline{\Phi}) \cap T\big)$. Using \textbf{Case 1}, we conclude that  $\Psi: N \to M|\big(\im(\underline{\Phi}) \cap T\big)$ is a morphism. By \autoref{ex: contractions and restrictions} \eqref{restriction}, the map $r_{\im(\underline{\Phi}) \cap T}: M|\big(\im(\underline{\Phi})\cap T\big) \rightarrow M$ is a morphism. Then $\Phi = r_{\im(\underline{\Phi})\cap T} \cdot \Psi: N \rightarrow M$ is a morphism.
\end{proof}

Next result justifies the choice of the term ``pre-image'' for $\Phi^{-1}(M)$ in \autoref{df: pre-image}.

\begin{prop}
\label{vectors of pre-image}
Let $\Phi = (\Phi_{i, j})_{(i, j) \in T \times S}$ be a submonomial $F$-matrix and $M$ an $F$-matrix on $T$. Then $\mathcal{V}_{\Phi^{-1}(M)} = \Phi^{-1}(\mathcal{V}_M)$.
\end{prop}

\begin{proof}
As any $F$-matroid is a quotient of itself, by \autoref{factorization of surjective strong maps through quotients}, it follows that the map $\Phi: \Phi^{-1}(M) \rightarrow M$ is a morphism. Thus $\Phi \cdot \mathcal{V}_{\Phi^{-1}(M)} \subseteq \mathcal{V}_M$.

Let $X \in \Phi^{-1}(\mathcal{V}_M) - \{0\}$. By \cite[Thm.~3.12 and Thm.~3.21]{Baker-Bowler19}, there exists an $F$-matroid $N$ such that $\mathcal{C}_N = \{\alpha X \mid \alpha \in F^\times\}$. As $\Phi \cdot \mathcal{C}_N \subseteq \mathcal{V}_M$, which is equivalent to $\Phi \cdot \mathcal{C}_N \perp \mathcal{C}^*_M$, we conclude, by \autoref{thm: cryptomorphisms}, that $\Phi: N \rightarrow M$ is a morphism. Thus by \autoref{factorization of surjective strong maps through quotients}, $\Phi^{-1}(M)$ is a quotient of $N$. By \cite[Rem.~2.12]{Jarra-Lorscheid24}, we have $\mathcal{C}_N \subseteq \mathcal{V}_{\Phi^{-1}(M)}$.
\end{proof}

\subsection{Minors of morphisms}
In this subsection we show how to construct minors of morphisms.

\begin{df}
    Let $\Phi$ be a submonomial $F$-matrix indexed by $T\times S$, and $A \subseteq S$ and $B\subseteq T$ subsets such that $\underline{\Phi}(\underline{A}) \subseteq \underline{B}$. We define:
    \[
    \begin{array}{lccl}
    \text{(\textit{Contraction} of }\Phi \text{ by }A\to B\text{)} & \Phi / (A\to B) & := &   (\Phi_{i,j})_{(i,j) \in B^c \times A^c};\\
    \text{(\textit{Restriction} of }\Phi\text{ to }A \rightarrow B\text{)} & \Phi | (A\to B) & := & (\Phi_{i,j})_{(i,j) \in B \times A}.\\
    \end{array}
    \]
\end{df}

\begin{rem}
\label{rem: contraction and deletion} 
The deletion of a morphism is the restriction of its transpose to the complementary subsets. More explicitly, consider a submonomial $F$-matrix $\Phi$ indexed by $T\times S$ and subsets $A \subseteq S$ and $B\subseteq T$ such that $\underline{\Phi}(\underline{A}) \subseteq \underline{B}$. Since $\underline{\Phi}^t(\underline{B^c}) \subseteq \underline{A^c}$ for the respective complements $A^c=S-A$ and $B^c=T-B$ of $A$ and $B$, we define:
\[
    \begin{array}{lccl}
    \text{(\textit{Deletion} of }A \to B \text{ from }\Phi^t\text{)} & \Phi^t \backslash (A\to B) & := &  \Phi^t|(B^c \to A^c).
    \end{array}
    \]
Taking transposes yields the following alternative presentations of the deletion
\[
 \Phi^t \backslash (A\to B) \ = \ \big(\Phi|(A^c\to B^c)\big)^t \ = \ \big(\Phi /  (A\to B)\big)^t.
\]
\end{rem}

\begin{thm}
\label{thm: minors of morphisms}
Let $N$ and $M$ be two $F$-matroids on $S$ and $T$, respectively, $\Phi: N \rightarrow M$ a morphism, and $A \subseteq S$ and $B \subseteq T$ subsets such that $\underline{\Phi}(\underline{A}) \subseteq \underline{B}$. Then
\begin{enumerate}
    \item \label{contraction of morphism} $\Phi/(A \rightarrow B): N/A \rightarrow M/B$ is a morphism;
    \item \label{restriction of morphism} $\Phi|(A \rightarrow B): N|A \rightarrow M|B$ is a morphism;
    \item \label{deletion of morphism} $\Phi^t\backslash(A \rightarrow B): M^* \backslash B \rightarrow N^* \backslash A$ is a morphism.   
\end{enumerate}
\end{thm}

\begin{proof}
We first prove \eqref{restriction of morphism}. Let $V = (v_a)_{a \in A} \in \cV_{N|A}$. By \cite[Prop. 4.5]{Anderson19}, there exists $X = (x_s)_{s \in S} \in \cV_N$ such that $x_e = v_e$ for $e \in A$, and $x_e = 0$ for $e \in S-A$. As $\Phi: N \to M$ is a morphism, $\Phi \cdot V \in \cV_M$. Let $Y = (y_t)_{t \in T} := \Phi \cdot V$. As $\underline{\Phi}(\underline{A}) \subseteq \underline{B}$, one has that $y_t = 0$ for $t \in T-B$. Again, by \cite[Prop. 4.5]{Anderson19}, we conclude that $Z := (y_b)_{b \in B} \in \cV_{M|B}$. As $Z = \big(\Phi| (A \to B)\big) \cdot V$, the result follows.

Item \eqref{deletion of morphism} follows from \autoref{duality for strong maps} and \eqref{restriction of morphism}.

Next we prove \eqref{contraction of morphism}. By \autoref{duality for strong maps}, $\Phi^t: M^* \to N^*$ is a morphism. As $\underline{\Phi^t}(\underline{B^c}) \subseteq \underline{A^c}$, it follows that $\Phi^t|(B^c \to A^c): M^* | B^c \to N^* | A^c$ is a morphism, by \eqref{restriction of morphism}. Thus the result follows from \autoref{rem: contraction and deletion} and another application of \autoref{duality for strong maps}.
\end{proof}

\subsection{Functoriality}
The next Proposition shows that morphisms of matroids with coefficients are well behaved with respect to idyll morphisms.

Let $f: F \rightarrow G$ be a morphism of perfect idylls and $\Phi$ a submonomial $F$-matrix indexed by $T\times S$. 

\begin{df}
The \textit{push-forward} of $\Phi$ by $f$ is the submonomial $G$-matrix
\[
f_*(\Phi) := \big(f(\Phi_{i,j})\big)_{(i,j) \in T \times S}.
\]
\end{df}

\begin{prop}
\label{prop-functoriality}
Let $N$ and $M$ be $F$-matroids with respective underlying sets $S$ and $T$. If $\Phi: N \rightarrow M$ is a morphism of $F$-matroids, then $f_*(\Phi): f_*(N) \rightarrow f_*(M)$ is a morphism of $G$-matroids. Therefore $f_*$ defines a functor from $\cat{Mat}_F$ to $\cat{Mat}_G$.
\end{prop}

\begin{proof}
Let $f^+: \N[F^\times] \to \N[G^\times]$ be the morphism of semirings induced by $f$. For a set $E$, we also use $f_*$ to denote the map $F^E \to G^E$ given by $f_*\big( (x_e)_{e \in E}\big) := \big( f(x_e)\big)_{e \in E}$. 

Note that 
\[
f^+\big(Z \cdot (\Phi \cdot Y)\big) \; =\; f_*(Z) \cdot \big(f_*(\Phi) \cdot f_*(Y) \big)
\]
for all $Z \in F^T$ and $Y \in F^S$. In particular, if $Z \perp \Phi\cdot Y$, then $f_*(Z) \perp f_*(\Phi) \cdot f_*(Y)$. Therefore the result follows from \autoref{thm: cryptomorphisms} and \cite[Lemma 3.39]{Baker-Bowler19}.
\end{proof}

\begin{cor}
\label{functoriality and classical strong maps}
Let $M$ and $N$ be $F$-matroids on the sets $T$ and $S$, respectively, and $\Phi: N \rightarrow M$ a morphism. Then $\underline{\Phi}: \underline{S} \rightarrow \underline{T}$ is an $\F_1$-linear strong map $\underline{\Phi}: \underline{N} \rightarrow \underline{M}$.
\end{cor}

\begin{proof}
Let $f: F \rightarrow \K$ be the unique idyll morphism from $F$ to $\K$. By \autoref{prop-functoriality}, $f_*(\Phi): f_*(N) \rightarrow f_*(M)$ is a morphism. As $\underline{\Phi} = \underline{f_*(\Phi)}$ and $\underline{f_*(P)} = \underline{P}$ for any $F$-matroid $P$, the Corollary follows from \autoref{classical strong maps}.
\end{proof}

\subsection{Remark on morphisms for idylls with involution}
\label{subsection: morphisms for idylls with involution}

Baker and Bowler formulate their theory of matroids with coefficients for idylls $F$ with an involution $x\mapsto\overline x$ (which is an automorphism $F\to F$ of order $2$). The most prominent example for such an involution are the complex numbers $\C$ together with the complex conjugate function $z \mapsto \overline z$. In this case, the dual of a $\C$-matroid $M$ with vectors $\cV_M\subseteq \C^T$ is the $\C$-matroid $M^{\ast}$ whose vectors $\cV_{M^{\ast}}$ is the Hermitian complement of $\cV_M$ in $\C^T$.

The presence of an involution changes the notion of orthogonality and therefore the notion of the dual of an $F$-matroid. This extends to morphisms of $F$-matroids: the dual of a morphism $\Phi:N\to M$ of $F$-matroids (with respect to an involution $x\mapsto \overline x$) is described by the \emph{conjugate transpose} $\overline\Phi^t: M^{\ast}\to N^{\ast}$ of $\Phi$.

It is readily verified that the conjugate transpose is indeed a morphism of $F$-matroids, and that the cryptomorphic descriptions of morphisms and minor constructions extend to matroids over idylls with involution. We omit further details.


\section{Quiver \texorpdfstring{$F$}{F}-matroids}
\label{section: quiver matroids}

\subsection{Quiver representations}
\label{subsection: quiver representations}
A \emph{quiver} is a directed graph which allows loops and multiples arrows. More precisely, a quiver $Q$ is a tuple $(Q_0, Q_1, s, t)$ where $Q_0$ is the set of \emph{vertices}, $Q_1$ is the set of \emph{arrows}, and $s,t: Q_1 \rightarrow Q_0$ are functions mapping an arrow to its \emph{source} and \emph{target}, respectively. We denote an arrow of $Q$ by $s(\alpha) \xrightarrow{\;\;\alpha\;\;} t(\alpha)$.

Given a quiver $Q$, let $\cat{C}(Q)$ be the free category generated by $Q$. A \emph{representation} of $Q$ over a category $\cat{C}$ is a (covariant) functor 
$$
    \Gamma: \cat{C}(Q) \longrightarrow \cat{C}.
$$
Let $\Lambda:\cat{C}(Q)\rightarrow \cat{C}$ and $\Gamma:\cat{C}(Q)\rightarrow \cat{C}$ be two representations of $Q$. A \emph{morphism} of representations is a natural transformation $\Lambda \rightarrow \Gamma$. A \emph{subrepresentation} of $\Gamma$ is an isomorphism class of subfunctors of $\Gamma$ \footnote{A subfunctor of $\Gamma$ is a pair $(\Gamma', \tau)$, where $\Gamma'$ is a functor and $\tau: \Gamma' \rightarrow \Gamma$ a natural transformation such that $\tau_v: \Gamma'v \rightarrow \Gamma v$ is a monomorphism for all $v \in \textup{ob}\big(\cat{C}(Q)\big) = Q_0$. Two subfunctors $(\Gamma'_1, \tau_1)$ and $(\Gamma'_2, \tau_2)$ are isomorphic if there exists a natural isomorphism $T: \Gamma'_1 \rightarrow \Gamma'_2$ such that $\tau_1 = \tau_2 \circ T$.}.

We denote by $\cat{Rep}_{\cat{C}}(Q)$ the category of representations of $Q$ over $\cat{C}$.

\begin{ex}[Representations of quivers over vector spaces]
Let $k$ be a field, $\Vect_k$ the category of $k$-vector spaces, and $Q$ a quiver. A representation of $Q$ over $\Vect_k$ corresponds to the data $\Gamma=\big((V_v)_{v\in Q_0},(\varphi_\alpha)_{\alpha\in Q_1}\big)$, where $V_v$ is a $k$-vector space for each $v \in Q_0$ and $\varphi_\alpha: V_{s(\alpha)}\rightarrow V_{t(\alpha)}$ is a $k$-linear map for each $\alpha \in Q_1$. The \emph{dimension vector} of $\Gamma$ is the tuple $\underline{\dim}(\Gamma) := (\dim V_v)_{v \in Q_0}$.
    
Let $\Lambda = (V'_v,\varphi'_\alpha)$ and $\Gamma=(V_v,\varphi_\alpha)$ be two representations of $Q$ over $\Vect_k$. A morphism of representations $f: \Lambda \rightarrow \Gamma$ is a collection $(f_v:V'_v\rightarrow V_v)_{v\in Q_0}$ of $k$-linear maps, such that $f_{t(\alpha)}\circ \varphi'_\alpha = \varphi_\alpha \circ f_{s(\alpha)}$ for each arrow $\alpha$ in $Q_1$. The morphism $f=(f_v)$, of representations, is a monomorphism, an epimorphism or an isomorphism if each $f_v$ is, respectively. 

Each subrepresentation of $\Gamma = (V_v,\varphi_\alpha)$ is isomorphic to a unique representation $\Gamma' = (V'_v,\varphi'_\alpha)$ for which $V_v' \subseteq V_v$ is a linear subspace for all $v \in Q_0$ and for which $\varphi'_\alpha$ is the restriction of $\varphi_\alpha$ to a map $V_{s(\alpha)}' \rightarrow V_{t(\alpha)}'$ for all $\alpha \in Q_1$. For more details on representations of quivers over vector spaces, we refer the reader to \cite{Kirillov2016} and  \cite{Schiffler2014}. 
\end{ex}

\subsection{Quiver \texorpdfstring{$F$}{F}-matroids}

\begin{df}[Quiver $F$-Matroid]\label{def: Quiver F-Matroid}
Let $F$ be a perfect idyll and $Q=(Q_0,Q_1,s,t)$ a quiver. A \emph{quiver $F$-matroid of type $Q$}, or \emph{$(Q, F)$-matroid} for short, is a representation of $Q$ in $\cat{Mat}_F$. More explicitly, this is a tuple 
\[
M = \big((M_v)_{v\in Q_0}, (M_\alpha)_{\alpha \in Q_1}),
\]
where $M_v$ is an $F$-matroid for all $v \in Q_0$, and $M_\alpha: M_{s(\alpha)} \rightarrow M_{t(\alpha)}$ is a morphism of $F$-matroids for all $\alpha \in Q_1$. The \emph{rank vector} of $M$ is the tuple $\underline{\rk}(M) := \big(\rk(M_v)\big)_{v \in Q_0}$.
\end{df}

\begin{notation}
For $v \in Q_0$, we use $\mu_v: E_v^{r_v} \rightarrow F$ to denote a Grassmann-Pl\"ucker function such that $M_v = [\mu_v]$. In particular, $E_v$ is the underlying set of $M_v$ and $r_v$ is the rank of $M_v$.
\end{notation}

By \autoref{thm: cryptomorphisms}, the term
\[
 \sum_{k=0}^{r_{s(\alpha)}} (-1)^k \cdot \mu_{s(\alpha)}(y_0,\dotsc,\widehat{y_k},\dotsc,y_{r_{s(\alpha)}}) \cdot {M_{\alpha,\underline{M_{\alpha}}(y_k),y_k}} \cdot \widetilde{\mu_{t(\alpha)}} \big(\underline{M_{\alpha}}(y_k),x_2,\dotsc,x_{r_{t(\alpha)}}\big)
\]
is contained in $N_F$ for all $\alpha \in Q_1$, all $y_0,\dotsc,y_{r_{s(\alpha)}} \in E_{s(\alpha)}$ and all $x_2,\dotsc,x_{r_{t(\alpha)}} \in E_{t(\alpha)}$.

\begin{ex}[Flag $F$-matroids as quiver $F$-matroids] \label{flag matroids}
Let $F$ be a perfect idyll and 
$$
    M = (M_1, \dots, M_s)
$$
a sequence of $F$-matroids on the same set $E$. The sequence $M$ is a \emph{flag $F$-matroid} if $M_i$ is a quotient of $M_j$ for all $i < j$. By \cite[Thm.~2.17]{Jarra-Lorscheid24} and \autoref{ex: quotients}, $M$ is a flag $F$-matroid if, and only if,
$$
    \begin{tikzcd}[column sep =40pt]
         M_1 & \arrow[swap, "\id_E"]{l} M_2 & \arrow[swap,"\id_E"]{l} \cdots & \arrow[swap, "\id_E"]{l} M_{s-1} & \arrow[swap, "\id_E"]{l}   M_s,
    \end{tikzcd}
$$
is a quiver $F$-matroid.
\end{ex}

\subsection{Relation to quiver representations over \texorpdfstring{$\Fun$}{F1}}
\label{subsction: relation do quiver representations over f1}
Let $F$ be a perfect idyll. In the following, we denote the underlying set of an $F$-matroid $M$ by $E_M$.

\begin{df}
The assignment
\begin{itemize}  
    \item $M \mapsto \underline{E_M}$ for $F$-matroids
    \item $(\Phi: N \rightarrow M) \mapsto (\underline{\Phi}: \underline{E_N} \rightarrow \underline{E_M})$ for morphisms
\end{itemize}
defines the \emph{underlying set functor} $\cat{U}_F: \cat{Mat}_F \rightarrow \Vect_{\F_1}$.

The \emph{underlying $\Fun$-representation} of a $(Q, F)$-matroid $M$ is the representation of $Q$ over $\cat{Vect}_{\F_1}$ given by $\cat{U}_F(M) := \big((\underline{E_{M_v}}), (\underline{M_\alpha})\big)$.
\end{df}

Let $V$ be a pointed finite set, $E= V-\{0\}$ and $n=\# E$. Let $\mu_V: E^n\rightarrow \K$ be the unique Grassmann-Pl\"ucker function of rank $n$ on $E$, which sends ${\bf x}$ to $1$ if $|{\bf x}| = E$, and to $0$ otherwise. Given an $\F_1$-linear map $g: V \to W$, let $g_\K$ be the unique submonomial $\K$-matrix such that $\underline{g_\K} = g$, which is indexed by $(W - \{0\}) \times (V - \{0\})$ and whose entries are
\[
g_{\K, w, v} \ := \ \begin{cases}
                    1 & \text{if } g(v) = w,\\
                    0 &\text{if } g(v) \neq w.
                    \end{cases}
\]

\begin{lemma}
Let $g: V \to W$ be an $\F_1$-linear map and $M = [\mu]$ a $\K$-matroid on $W - \{0\}$. Then $g_\K: [\mu_V] \rightarrow M$ is a morphism of $\K$-matroids.
\end{lemma}

\begin{proof}
Let $r$ be the rank of $M$. Let $\mathbf{y} \in (V - \{0\})^{\#V}$ and $\mathbf{x} \in (W - \{0\})^{r-1}$. Note that there exist $\ell \neq m$ such that $y_\ell = y_m$, as $\mu_V(\mathbf{y}_{\widehat{k}}) = 0$ for $k \neq \ell, m$, it follows that
\[
\begin{aligned}
    & \; \underset{k = 1}{\overset{\#V}{\sum}} \mu_V(\mathbf{y}_{\widehat{k}}) ~ {g_{\K, y_k}} ~ \widetilde{\mu} \big(g(y_k), \mathbf{x}\big)
\\
    = & \; \mu_V(\mathbf{y}_{\widehat{\ell}}) ~ {g_{\K, y_\ell}} ~ \widetilde{\mu} \big(g(y_\ell), \mathbf{x}\big)  \; + \; \mu_V(\mathbf{y}_{\widehat{m}}) ~ {g_{\K, y_m}} ~ \widetilde{\mu} \big(g(y_m), \mathbf{x}\big)
\\
    = & \; 2 \cdot \; \mu_V(\mathbf{y}_{\widehat{\ell}}) ~ {g_{\K, y_\ell}} ~ \widetilde{\mu} \big(g(y_\ell), \mathbf{x}\big),
\end{aligned}
\]
which is in $N_\K = \N - \{1\}$. Thus the result follows.
\end{proof}

\begin{df}
The assignment
\begin{itemize}
    \item $V \mapsto [\mu_V]$ for finite pointed sets
    \item $(g: V \to W) \mapsto (g_\K: [\mu_V] \rightarrow [\mu_W])$ for $\F_1$-linear maps
\end{itemize}
defines the functor $\cat{M}: \Vect_{\F_1} \rightarrow \cat{Mat}_\K$.
\end{df}

\begin{thm}\label{thm: adjunction}
The functor $\cat{M}$ is left adjoint to $\cat{U}_\K$.
\end{thm}
\begin{proof}
Let $M$ be a $\K$-matroid and $V$ a finite pointed set. We define the maps 
\begin{equation*}
    \begin{tikzcd}
        \textup{Hom}\big(\cat{M}(V), M\big)
        \arrow[r, shift left, "\alpha"] &
        \arrow[l, shift left, "\gamma"] 
        \textup{Hom}\big(V, \cat{U}_\K(M)\big)
    \end{tikzcd}
\end{equation*}
by
\begin{equation}\label{eq: adjunction cat}
    \alpha(\Phi) := \underline{\Phi}: V \rightarrow \underline{E_{M}} = \cat{U}_\K(M) \quad \text{and} \quad \gamma(g) := g_\K: \cat{M}(V) \rightarrow M.
\end{equation}
As $\alpha$ and $\gamma$ are mutually inverse, the result follows.
\end{proof}

\begin{rem}
The functor $\cat{U}_\K$ also has a right adjoint $\cat{M}_0$, which sends a finite pointed set $V$ to the unique $\K$-matroid of rank $0$ on $V - \{0\}$. Both $\cat{M}$ and $\cat{M}_0$ are right inverses to $\cat{U}_\K$.
\end{rem}

Let $Q$ be a quiver. Then the functors $\cat{U}_\K$ and $\cat{M}$ induce functors 
$$
    \begin{tikzcd}[row sep = -3pt, column sep = 15pt, %
    /tikz/column 1/.append style={anchor=base east}
    ,/tikz/column 2/.append style={anchor=base west}]
        \cat{U}_{Q}:\rep_{\cat{Mat}_\K}(Q)\arrow[r]  & \rep_{\cat{Vect}_\Fun}(Q)
        \\
        \Gamma  \arrow[r,mapsto] & \cat{U}_\K\circ \Gamma
    \end{tikzcd}
    \quad \text{ and } \quad 
    \begin{tikzcd}[row sep = -3pt, column sep = 15pt, %
    /tikz/column 1/.append style={anchor=base east}
    ,/tikz/column 2/.append style={anchor=base west}]
        \cat{M}_Q:\rep_{\cat{Vect}_\Fun}(Q) \arrow[r] &\rep_{\cat{Mat}_\K}(Q)
        \\
        \Lambda \arrow[r,mapsto] & \cat{M}\circ \Lambda.
    \end{tikzcd}
$$
The functors $\cat{U}_Q$ and $\cat{M}_Q$ satisfy the following corollary.

\begin{cor}\label{Cor: adjunction}
The functor $\cat{M}_Q$ is left adjoint to $\cat{U}_Q$.
\end{cor}

\begin{proof} Let $\Lambda$ be a representation of $Q$ over $\cat{Vect}_\Fun$ and $\Gamma$ a $(Q, \K)$-matroid. We define 
$$
    \begin{tikzcd}
        \textup{Hom}\big(\cat{M}_Q(\Lambda), \Gamma\big)
        \arrow[r, shift left, "\alpha_Q"] &
        \arrow[l, shift left, "\gamma_Q"] 
       \textup{Hom}\big(\Lambda, \cat{U}_Q(\Gamma)\big) 
    \end{tikzcd}
$$
by $\alpha_Q = \big(\alpha_v)_{v\in Q_0}$ and $\gamma_Q = \big(\gamma_v)_{v\in Q_0}$, with $\alpha_v$ and $\gamma_v$ defined as in \eqref{eq: adjunction cat}, for each $v$. As $\gamma_v\circ \alpha _v$ and $\alpha _v\circ\gamma_v$ are identity maps, it follows that  $\alpha_Q\circ \gamma_Q$ and $\gamma_Q\circ\alpha_Q$ are identity maps as well, which finishes the proof.   
\end{proof}

\begin{ex}[Quiver matroids of type $D_4$]\label{ex: quiver matroids of type D4}
Let $Q$ be a $D_4$-quiver and $\Lambda$ the $\Fun$-representation of $Q$ as illustrated in \autoref{fig: F1-representation Lambda}: the points vertically above $v_i$ (for $i=0, 1, 2, 3$) correspond the nonzero elements of $\Lambda_{v_i}$; the arrows above $\alpha_i$ (for $i=1,2,3$) illustrate the map $\Lambda_{\alpha_i}:\Lambda_{v_0}\to\Lambda_{v_i}$; the elements of $\Lambda_{v_0}$ without an emerging arrow above $\alpha_i$ are mapped to $0$. More to the point, \autoref{fig: F1-representation Lambda} illustrates the coefficient quiver of $\Lambda$; cf.\ \autoref{ex - last example}.

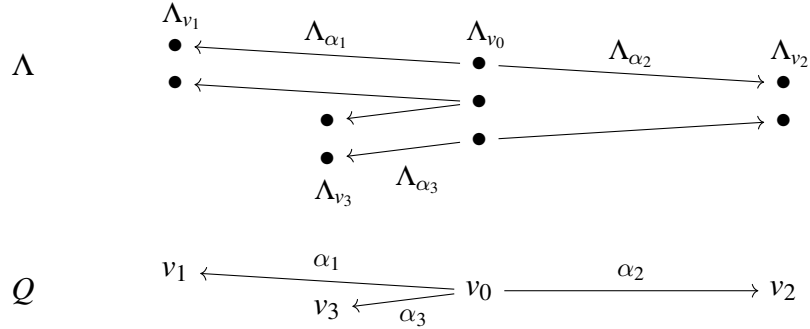
\begin{figure}[htb] 
 \[
  \begin{tikzpicture}[scale=1, x=2cm, y=0.5cm, font=\normalsize, vertices/.style={draw, fill=black, circle, inner sep=0pt},decoration={markings,mark=at position 0.5 with {\arrow{>}}}]
   \draw (1,6.5) node (2) {$\bullet$};
   \draw (1,5.5) node (5) {$\bullet$};
   \draw (2,4.5) node (7) {$\bullet$};
   \draw (2,3.5) node (9) {$\bullet$};
   \draw (3,6.0) node (1) {$\bullet$};
   \draw (3,5.0) node (4) {$\bullet$};
   \draw (3,4.0) node (6) {$\bullet$};
   \draw (5,5.5) node (3) {$\bullet$};
   \draw (5,4.5) node (8) {$\bullet$};
   \draw[->] (1) -- (2);
   \draw[->] (1) -- (3);
   \draw[->] (4) -- (5);
   \draw[->] (4) -- (7);
   \draw[->] (6) -- (8);
   \draw[->] (6) -- (9);
   \draw (3,0.0) node (v0) {$v_0$};
   \draw (1,0.5) node (v1) {$v_1$};
   \draw (5,0.0) node (v2) {$v_2$};
   \draw (2,-0.5) node (v3) {$v_3$};
   \draw[->] (v0) -- node[above] {\footnotesize $\alpha_1$} (v1);
   \draw[->] (v0) -- node[above] {\footnotesize $\alpha_2$} (v2);
   \draw[->] (v0) -- node[pos=0.4,below] {\footnotesize $\alpha_3$} (v3);
   \draw (0,6) node (G) {$\Lambda$};
   \draw (0,0) node (Q) {$Q$};
   \draw (3.05,6.8) node (G) {\small $\Lambda_{v_0}$};
   \draw (1.05,7.3) node (G) {\small $\Lambda_{v_1}$};
   \draw (5.05,6.3) node (G) {\small $\Lambda_{v_2}$};
   \draw (2.05,2.6) node (G) {\small $\Lambda_{v_3}$};
   \draw (2,6.8) node (G) {\small $\Lambda_{\alpha_1}$};
   \draw (4,6.3) node (G) {\small $\Lambda_{\alpha_2}$};
   \draw (2.6,2.95) node (G) {\small $\Lambda_{\alpha_3}$};
   \end{tikzpicture}
  \] 
  \caption{The $\Fun$-representation $\Lambda$ of type $D_4$} 
 \label{fig: F1-representation Lambda}     
\end{figure}

Next we determine all $(\Lambda,\K)$-matroids $M$ with rank vector $\br=(r_0,r_1,r_2,r_3)=(2,1,1,1)$.  Thus $M_{v_i}$ has corank $1$ for every $i=0, 1, 2, 3$ and therefore a unique circuit, which is the unique nonzero vector of $M_{v_i}$. In the following, we list the different choices for nonzero vectors (written as row vectors), ordered by the isomorphism type of $M_{v_i}$ (for $i\neq0$ and $i=0$, respectively): 
\[
 \begin{tabular}{l|ll||l|l}
  \text{type of $M_{v_i}$} & \text{nonzero vectors}
  & \qquad &
  \text{type of $M_{v_0}$} & \text{nonzero vectors} 
  \\
  \hline &&&& \ \\[-10pt]
  $U_1^1 \oplus U_1^0$    & $(1, 0), \quad (0, 1)$ 
  & \qquad &
  $U_2^2 \oplus U_1^0$    & $(1,0,0), \quad (0,1,0), \quad (0,0,1)$ 
  \\[5pt]
  $U_2^1$                 & $(1, 1)$ 
  & \qquad &
  $U_2^1 \oplus U_1^1$    & $(1,1,0), \quad (1,0,1), \quad (0,1,1)$ 
  \\[5pt]
  && \qquad &
  $U_3^2$                 & $(1,1,1)$ 
 \end{tabular}
\]

Let $\Phi_i$ be the unique submonomial $\K$-matrix such that $\underline{\Phi_i} = \Lambda_{\alpha_i}$ (for $i = 1, 2, 3$), i.e.\ 
\[
 \Phi_1 \ = \ \big[\begin{smallmatrix} 1&0&0\\0&1&0 \end{smallmatrix}\big], \qquad
 \Phi_2 \ = \ \big[\begin{smallmatrix} 1&0&0\\0&0&1 \end{smallmatrix}\big], \qquad
 \Phi_3 \ = \ \big[\begin{smallmatrix} 0&1&0\\0&0&1 \end{smallmatrix}\big].
\]
A quadruple of corank $1$ matroids $(M_{v_0},M_{v_1}, M_{v_2},M_{v_3})$ with respective nonzero vectors $(w_0,w_1, w_2, w_3)$ as above is a $(\Lambda,\K)$-matroid if and only if $\Phi_i\cdot w_0\subseteq\cV_{M_{v_i}}=\{0,w_i\}$. We illustrate such quadruples in \autoref{fig: quiver matroids of type D4} as follows: each vertex $w$ within the colored area labeled by $v_i$ represents a corank $1$ matroid $M_w$ on $\Lambda_{v_i}-\{0\}$ whose label $w$ is the unique nonzero vector of $M_w$. An arrow from a vertex with label $w$ in the green area (labeled by $v_0$) to a vertex with label $w'$ in a blue area (labeled by some $v_i$) indicates that $\Phi_i: M_w \to M_{w'}$ is a morphism. Therefore $(\Lambda, \K)$-matroids of rank $\textbf{r}$ correspond to connected full subquivers of $Q'$ with exactly one vertex per colored area.

\newcounter{tikz-counter}
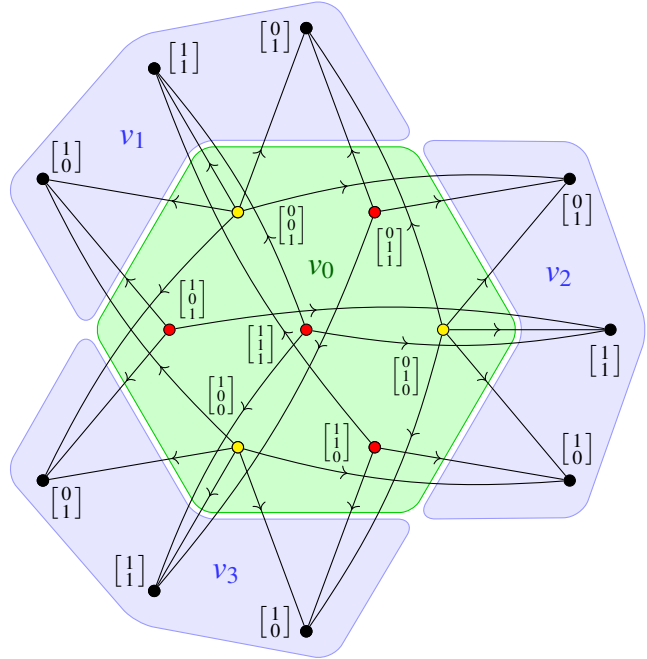
\begin{figure}[ht] 
  \[
   \begin{tikzpicture}[scale=1, font=\footnotesize, vertices/.style={draw, fill=black, circle, inner sep=0pt},decoration={markings,mark=at position 0.33 with {\arrow{>}}}]
    \draw[fill=green!20!white,draw=green!80!black,rounded corners=5pt] (0:2.8cm) -- (60:2.8cm) -- (120:2.8cm) -- (180:2.8cm) -- (240:2.8cm) -- (300:2.8cm) -- cycle;
    \node[color=green!40!black] at (0.2,0.8) {\normalsize $v_0$};
    \draw[fill=blue!10!white,draw=blue!40!white,rounded corners=7pt] (300:2.9cm) -- (0:2.9cm) -- (60:2.9cm) -- (34.5:4.4cm) -- (0:4.5cm) -- (325.5:4.4cm) -- cycle;
    \node[color=blue!80!white] at (12:3.4) {\normalsize $v_2$};
    \draw[fill=blue!10!white,draw=blue!40!white,rounded corners=7pt] (60:2.9cm) -- (120:2.9cm) -- (180:2.9cm) -- (154.5:4.4cm) -- (120:4.5cm) -- (85.5:4.4cm) -- cycle;
    \node[color=blue!80!white] at (132:3.4) {\normalsize $v_1$};
    \draw[fill=blue!10!white,draw=blue!40!white,rounded corners=7pt] (180:2.9cm) -- (240:2.9cm) -- (300:2.9cm) -- (274.5:4.4cm) -- (240:4.5cm) -- (205.5:4.4cm) -- cycle;
    \node[color=blue!80!white] at (252:3.4) {\normalsize $v_3$};
    \draw (0,0) node [draw,circle,inner sep=1.5pt,fill=red] (0) {};
    \foreach \a in {1,...,5}{\draw (0+\a*120: 1.8cm) node [draw,circle,inner sep=1.5pt,fill=yellow] (a\a) {};
                             \draw (180+\a*120: 1.8cm) node [draw,circle,inner sep=1.5pt,fill=red] (b\a) {};
                             \draw (\a*120: 4cm) node [draw,circle,inner sep=1.5pt,fill=black] (c\a) {};
                             \draw (30+\a*120: 4cm) node [draw,circle,inner sep=1.5pt,fill=black] (d\a) {};
                             \draw (-30+\a*120: 4cm) node [draw,circle,inner sep=1.5pt,fill=black] (e\a) {};
                            }
    \foreach \a in {1,...,3} {\setcounter{tikz-counter}{\a};
                              \draw (0) edge [-,postaction={decorate},bend right=10] (c\arabic{tikz-counter});
                              \draw (b\a) edge [-,postaction={decorate},bend left=10] (c\arabic{tikz-counter});
                              \draw (a\a) edge [-,postaction={decorate}] (c\arabic{tikz-counter});
                              \draw (a\a) edge [-,postaction={decorate}] (d\arabic{tikz-counter});
                              \draw (a\a) edge [-,postaction={decorate}] (e\arabic{tikz-counter});
                              \addtocounter{tikz-counter}{1};
                              \draw (b\a) edge [-,postaction={decorate},bend right=0] (d\arabic{tikz-counter});
                              \draw (a\a) edge [-,postaction={decorate},bend right=10] (e\arabic{tikz-counter});
                              \addtocounter{tikz-counter}{1};
                              \draw (a\a) edge [-,postaction={decorate},bend left=10] (d\arabic{tikz-counter});
                              \draw (b\a) edge [-,postaction={decorate},bend left=0] (e\arabic{tikz-counter});
                              }
    \draw (-0.6,-0.2)  node {\tiny $\smalltrivector111$};
    \draw (-25:1.45cm) node {\tiny $\smalltrivector010$};
    \draw ( 45:1.55cm) node {\tiny $\smalltrivector011$};
    \draw ( 98:1.4cm)  node {\tiny $\smalltrivector001$};
    \draw (165:1.55cm) node {\tiny $\smalltrivector101$};
    \draw (218:1.42cm) node {\tiny $\smalltrivector100$};
    \draw (285:1.55cm) node {\tiny $\smalltrivector110$};
    \draw (336:3.9cm)  node {$\smallvector10$};
    \draw (354:3.95cm) node {$\smallvector11$};
    \draw ( 24:3.9cm)  node {$\smallvector01$};
    \draw ( 96:3.9cm)  node {$\smallvector01$};
    \draw (114:3.95cm) node {$\smallvector11$};
    \draw (144:3.9cm)  node {$\smallvector10$};
    \draw (216:3.9cm)  node {$\smallvector01$};
    \draw (234:3.95cm) node {$\smallvector11$};
    \draw (264:3.9cm)  node {$\smallvector10$};
   \end{tikzpicture}
  \] 
  \caption{The quiver $Q'$, which illustrates all $(\Lambda,\K)$-matroids of rank $\br$} 
 \label{fig: quiver matroids of type D4}     
\end{figure}

Note that each red vertex $w$ of $Q'$ determines a unique $(\Lambda, \K)$-matroid $M$ of rank $\textbf{r}$ with $M_{v_0} = M_w$, while for each yellow vertex $w'$ there exist $3$ $(\Lambda, \K)$-matroids $M$ of rank $\textbf{r}$ with $M_{v_0} = M_{w'}$. In total, there are $13$ $(\Lambda, \K)$-matroids of rank $\textbf{r}$.
\end{ex}

\subsection{Duality and minors of quiver matroids}\label{Section: Duality and minors of quiver matroids}
Let $Q = (Q_0,Q_1,s,t)$ be a quiver. The \emph{opposite quiver of $Q$} is the quiver $Q^\ast$ with
\[
 Q_0^* := Q_0, \qquad Q_1^* := Q_1, \qquad s^* := t, \qquad t^* := s.
\]
Let $F$ be a perfect idyll and $M$ a $(Q, F)$-matroid. By \autoref{duality for strong maps}, the transpose $\Phi_\alpha^t: M_{t(\alpha)}^* \rightarrow M_{s(\alpha)}^*$ is a morphism for all $\alpha \in Q_1$.

\begin{df}
\label{duality for quiver matroids}
The \emph{dual} of $M$ is $M^* := \big((M_v^*)_{v\in Q_0^*}, (M_\alpha^t)_{\alpha \in Q_1^*}\big)$, which is a quiver $F$-matroid of type $Q^*$.
\end{df}

Let $\cat{U}_\K(M) = \big((\underline{E_v}), (\underline{M_\alpha})\big)$ be the underlying representation of $M$ and $\Lambda$ a subrepresentation of $\cat{U}_\K(M)$ where $\Lambda_v \subseteq \underline{E_v}$ for all $v \in Q_0$. Define $A_v := \Lambda_v - \{0\}$. By \autoref{thm: minors of morphisms}, for any $\alpha \in Q_1$, the following are morphisms:
\[
\begin{aligned}
M_\alpha/(A_{s(\alpha)} \rightarrow A_{t(\alpha)})&: M_{s(\alpha)}/A_{s(\alpha)} \rightarrow M_{t(\alpha)}/A_{t(\alpha)},
\\
M_\alpha|(A_{s(\alpha)} \rightarrow A_{t(\alpha)})&: M_{s(\alpha)}|A_{s(\alpha)} \rightarrow M_{t(\alpha)}|A_{t(\alpha)},
\\
M_\alpha^t \backslash (A_{s(\alpha)} \rightarrow A_{t(\alpha)})&: M_{t(\alpha)}^* \backslash A_{t(\alpha)} \rightarrow M_{s(\alpha)}^* \backslash A_{s(\alpha)}.
\end{aligned}
\]

\begin{df} We define the following quiver $F$-matroids:
\[
\begin{array}{lccc}
\text{(\emph{Contraction} of $M$ by $\Lambda$)} \; & M / \Lambda & := &   \big((M_v / A_v) \;,\; \big(M_\alpha/(A_{s(\alpha)} \rightarrow A_{t(\alpha)})\big)\big),
\\
\text{(\emph{Restriction} of $M$ to $\Lambda$)} \; & M | \Lambda & := &   \big((M_v | A_v) \;,\; \big(M_\alpha|(A_{s(\alpha)} \rightarrow A_{t(\alpha)})\big)\big),
\\
\text{(\emph{Deletion} of $\Lambda$ from $M^*$)} \; & M^* \backslash \Lambda & := & \big((M_v^* \backslash A_v) \;,\; \big(M_\alpha^t \backslash (A_{t(\alpha)} \rightarrow A_{s(\alpha)})\big)\big).
\end{array}
\]
\end{df}

\begin{rem}
Note that $M / \Lambda$ and $M | \Lambda$ are of type $Q$, and $M^* \backslash \Lambda$ is of type $Q^*$.
\end{rem}


\section{Bands and band schemes}
\label{secion: band schemes}

The vessel of our approach to moduli spaces of quiver matroids are band schemes, which form a streamlined approach to $\Fun$-geometry that is developed in \cite{Baker-Jin-Lorscheid24}. We review this theory in brevity in this section. 

In particular, we tailor certain results about ordered blueprints to bands, which is possible since bands are ordered blueprints of a particularly well-behaved type (see~\autoref{subsubsection: bands as ordered blueprints})---morally speaking, a band is the part of an ordered blueprint that is sensitive to matroid theory. For further details, we refer the reader to \cite{Baker-Jin-Lorscheid24}.

\subsection{Bands}
A band is a generalization of a ring, in the same way that an idyll is a generalization of a field. 

A \emph{pointed monoid} $B$ is a commutative monoid $B$ together with an element $0$ such that $0 \cdot a = 0$ for all $a \in B$. The \emph{ambient semiring} of a pointed monoid $B$ is the semiring
\[\textstyle
 B^+ \ := \ \N[B]\,/\,\gen{0} \ = \ \big\{ \sum a_i \, \big| \, a_i\in B-\{0\}\big\}.
\]
Note that $B$ embeds as a submonoid of $B^+$. An \emph{ideal} of $B^+$ is a subset $I$ such that \ $0 \in I$, \ \ $I + I = I$ \ and \ $B \cdot I = I$. 

A \emph{band} is a pointed monoid $B$ together with an ideal $N_B \subseteq B^+$ (the \emph{null set of $B$}) such that for every $a \in B$, there exists a unique element $b \in B$ such that $a+b\in N_B$, which we call the \emph{additive inverse of $a$} and which we denote by $-a$. We write $a-b$ for $a+(-b)$. A \emph{band morphism} is a multiplicative map $f:B\to C$ of bands that preserves $0$ and $1$ such that $\sum a_i \in N_B$ implies $\sum f(a_i) \in N_C$. We denote the category of bands by $\Bands$.

The null set $N_B$ of a band $B$ is \emph{generated by a subset $S$ of $B^+$} if it is generated by $S$ and $1-1$ as an ideal of $B^+$. We write $N_B=\gen S_{B^+}$ in this case.

\begin{ex}
 The main examples of interest to us are the following:
 \begin{enumerate}
  \item Every idyll is naturally a band. More to the point, an idyll is the same as a band $F$ with $F^\times=F-\{0\}$.
  \item A ring $R$ is a band with null set $N_R=\{\sum a_i\mid \sum a_i=0\text{ in }R\}$.
 \end{enumerate}
\end{ex}

\subsubsection{Free algebras}
Let $k$ be a band. A \emph{$k$-algebra} is a band $B$ together with a band morphism $\alpha_B:k\to B$ (called the \emph{structure map}). A \emph{$k$-linear morphism of $k$-algebras} from $B$ to $C$ is a band morphism $f:B\to C$ such that $\alpha_C=f\circ\alpha_B$. 

The \emph{free $k$-algebra in $T_1,\dotsc,T_n$} is the multiplicative submonoid
\[
 k[T_1,\dotsc,T_n] \ = \ \big\{ a\cdot T_1^{e_1}\dotsb T_n^{e_n} \, \big| \, a\in k,\, e_1,\dotsc,e_n\in\N \big\}
\]
of the polynomial semiring $k[T_1,\dotsc,T_n]^+$ over $k^+$. It contains $k$ as a submonoid and gains the structure of a band with respect to the null set generated by $N_k$. The free algebra $k[T_1,\dotsc,T_n]$ in $T_1,\dotsc,T_n$ satisfies the expected universal property: every map $f_0:\{T_1,\dotsc,T_n\}\to B$ into a $k$-algebra $B$ extends uniquely to a $k$-linear morphism $f:k[T_1,\dotsc,T_n]\to B$.

\subsubsection{Quotients}
A \emph{null ideal} of a band $B$ is an ideal $I$ of the ambient semiring $B^+$ that contains the null set $N_B$ and satisfies the substitution rule:
\begin{enumerate}[label=({SR})]
 \item if $a-c\in I$ and $c+\sum b_j\in I$, then $a+\sum b_j\in I$. \label{eq: axion null ideal}
\end{enumerate}
Given a subset $S$ of $B^+$, we denote the smallest null ideal that contains $S$ by $\gen{S}$. 

The \emph{quotient of $B$ by a null ideal $I$} is the following band $\bandquot BI$: as a set, $\bandquot BI$ is the quotient $B/\sim$ by the equivalence relation given by $a\sim b$ if and only if $a-b\in I$. It is a pointed monoid with product $[a]\cdot[b]=[ab]$. The null set of $\bandquot BI$ is defined as
\[\textstyle
 N_{\bandquot BI} \ = \ \big\{ \sum [a_i] \, \big| \, \sum a_i\in I \big\}.
\]
The association $a\mapsto[a]$ defines a band morphism $\pi_I:B\to\bandquot BI$, which satisfies the following universal property: every band morphism $f:B\to C$ with $\sum f(a_i)\in N_C$ for every $\sum a_i\in I$ factors into $f=\bar f\circ\pi_I$ for a uniquely determined morphism $\bar f:\bandquot BI\to C$.

\subsubsection{Localizations}
A \emph{multiplicative subset} of a band $B$ is a subset $S$ that contains $1$ and is closed under multiplication. Given such $S$, we define the \emph{localization of $B$ at $S$} as 
\[
 S^{-1}B \ = \ (S\times B) \, / \, \sim
\]
where $(s,a)\sim(s',a')$ if and only if $tsa'=ts'a$ for some $t\in S$. We denote by $\frac as$ the class of $(s,a)$ in $S^{-1}B$. The product $\frac as\cdot \frac bt=\frac{ab}{st}$ turns $S^{-1}B$ into a pointed monoid, and the null set 
\[\textstyle
 N_{S^{-1}B} \ = \ \big\{ \sum \frac{a_i}s \, \big| \, s\in S,\ \sum a_i\in N_B \big\}
\]
endows $S^{-1}B$ with the structure of a band.

The association $a\mapsto\frac a1$ defines a band morphism $\iota_S:B\to S^{-1}B$, which satisfies the following universal property: every band morphism $f:B\to C$ with $f(S)\subseteq C^\times$ factors into $f=f_S\circ\iota_S$ for a uniquely determined morphism $f_S:S^{-1}B\to C$.

A \emph{finite localization of $B$} is a localization of the form $B[h^{-1}]=S^{-1}B$ where $h\in B$ and $S=\{h^i\}_{i\in \N}$.

\subsubsection{Ideals}
An \emph{$m$-ideal} of a band $B$ is a subset $I$ of $B$ that contains $0$ such that $IB=I$. An $m$-ideal $\fp$ is \emph{prime} if $S=B-\fp$ is a multiplicative subset of $B$. The \emph{localization of $B$ at $\fp$} is $B_\fp=S^{-1}B$.

Every nonzero band $B$ has a unique maximal $m$-ideal, which is $\fm_B=B-B^\times$ and which is prime.

\subsubsection{Limits and colimits}
The category of a bands contains all limits and colimits. The regular partial field $\Funpm$ is initial and the zero band $\mathbf0=\{0\}$ is terminal in $\Bands$. The product of two bands $B$ and $C$ is the Cartesian product $B\times C$ together with the inherited structure as a band. The tensor product $B\otimes_k C$ of two $k$-algebras $B$ and $C$ is defined as the colimit of $B\leftarrow k\to C$.

\subsubsection{The universal ring}
The \emph{universal ring} of a band $B$ is the ring
\[
 B^+_\Z \ := \ \Z[B] / \gen{N_B} 
\]
where $\gen{N_B}$ is the ideal of $\Z[B]$ generated by all expression in $N_B$, interpreted as a subset of the monoid algebra $\Z[B]$. If $R$ is a ring with associated band $B$, then $B^+_\Z=R$.

The universal ring comes together with a map $B\to B^+_\Z$, which is universal among all maps $f:B\to R$ into a ring $R$ with the property that $\sum f(a_i)=0$ for all $\sum a_i\in N_B$. In other words, $\Rings$ form a reflective subcategory of $\Bands$.  

Note that the canonical map $B\to B^+_\Z$ is in general not injective. For example, $\K^+_\Z$ is the trivial ring $\mathbf0=\{0\}$.

\subsubsection{Bands as ordered blueprints}
\label{subsubsection: bands as ordered blueprints}

For the sake of this section, we assume that the reader is familiar with the notion of an ordered blueprint. We employ the notation from \cite{Baker-Lorscheid21} and \cite{Baker-Jin-Lorscheid24}. 

A band $B$ can be viewed as the ordered blueprint $B^\oblpr$ with underlying monoid $B$, ambient semiring $B^+$ and partial order generated by the relations $0\leq \sum a_i$ for which $\sum a_i\in N_B$. This association defines a fully faithful embedding $\Bands\to\OBlpr$. In fact, this embedding is a coreflection: every ordered blueprint $B=(B^\bullet,B^+,\leq)$ has an associated band $B^\band=B^\bullet$ whose nullset is
\[\textstyle
 N_{B^\band} \ = \ \big\{ \sum a_i \, \big| \, 0\leq \sum a_i \text{ in }B\big\},
\]
and the canonical map $(B^\band)^\oblpr\to B$ establishes an adjunction
\[
 \Hom_{\Bands}(B, \ C^\band) \quad \stackrel\sim\longrightarrow \quad \Hom_{\OBlpr}(B^\oblpr,\ C).
\]

This tight connection between bands and ordered blueprints makes it possible to specialize and transfer many results from ordered blueprints to bands.

\subsection{Band schemes}

The spectrum of a band is a topological space that comes with a structure sheaf of coordinate bands. A band scheme is a space that locally looks like the spectrum of a band. We review these definitions in the following.

\subsubsection{Band spaces}
A \emph{band space} is a topological space $X$ together with a sheaf $\cO_X$ in $\Bands$. Every point $x\in X$ comes with a \emph{stalk}, which is $\cO_{X,x}=\colim\cO_X(U)$ where $U$ ranges over all open neighbourhoods $U$ of $x$. A \emph{morphism of band spaces} is a continuous map $\varphi:X\to Y$ together with a morphism $\varphi^\#:\cO_Y\to\varphi_\ast\cO_X$ of sheaves that is \emph{local} in the sense that for every $x\in X$ and $y=\varphi(x)$, the induced morphism of stalks $\varphi^\#_x:\cO_{Y,y}\to\cO_{X,x}$ maps the maximal ideal $\fm_y$ of $\cO_{Y,y}$ to the maximal ideal $\fm_x$ of $\cO_{X,x}$. This defines the category $\BandSpaces$ of band spaces.

Every open subset $U$ of a band space $X$ is naturally a band space with respect to the restriction of $\cO_X$ to $U$.

\subsubsection{The spectrum}
Let $B$ be a band. The \emph{spectrum of $B$} is the set $X=\Spec B$ of all prime $m$-ideals $\fp$ of $B$, endowed with the topology generated by the \emph{principal opens}
\[
 U_h \ = \ \{ \fp\in X \mid h\notin \fp\}
\]
for $h\in B$ and with the sheaf $\cO_X$ with values $\cO_X(U_h)=B[h^{-1}]$ and stalks $\cO_{X,\fp}=B_\fp$.

Note that $\cO_X$ is uniquely determined by its values on principal opens since the $U_h$ form a basis for the topology of $\Spec B$. The proof of the existence of $\cO_X$ is easier than in usual algebraic geometry since every non-empty spectrum $\Spec B$ contains a unique closed point, which is the maximal ideal $\fm_B=B-B^\times$ of $B$.

A band morphism $f:B\to C$ induces a morphism $\varphi=f^\ast:\Spec C\to\Spec B$ with $f^\ast(\fq)=f^{-1}(\fq)$ and the sheaf morphism $\varphi^\#$ determined by 
\[
 \begin{array}{cccc}
  \varphi^\#(U_h): & B[h^{-1}] & \longrightarrow & C[f(h)^{-1}]. \\[5pt]
                   & \frac{a}{h^i} & \longmapsto & \frac{f(a)}{f(h)^i}
 \end{array}
\]
This defines a contravariant functor $\Spec:\Bands\to\BandSpaces$, which is fully faithful. 

\begin{ex}
 The following are first examples of spectra:
 \begin{enumerate}
  \item An idyll $F$ has a unique prime $m$-ideal, which is $\eta=\{0\}$. Thus $\Spec F=\{\eta\}$ is a singleton.
  \item The \emph{affine space} over a band $B$ is $\A^n=\Spec B[T_1,\dotsc,T_n]$. For an idyll $F$, the prime $m$-ideals of $\fp$ of $F[T_1,\dotsc,T_n]$ are of the form $\fp_I=\gen{T_i\mid i\in I}$ for subsets $I$ of $E=\{1,\dotsc,n\}$. Thus $\A^n_F=\{\fp_I\mid I\subseteq E\}$ as a set, and $\fp_I$ is in the closure of $\fp_J$ if and only if $J\subseteq I$. 
 \end{enumerate}
\end{ex}

\subsubsection{Band schemes}
\label{subsection: band schemes}
An \emph{affine band scheme} is a band space that is isomorphic to the spectrum of a band. A \emph{band scheme} is a band space that has an open covering by \emph{affine opens}, which are open subsets that are affine band schemes. This defines the categories $\BAff$ of affine band schemes and $\BSch$ of band schemes.

Taking global sections defines a (contravariant) functor $\Gamma:\BSch\to\Bands$ that is adjoint to $\Spec$, i.e. $\Hom(B,\Gamma X) = \Hom(X,\Spec B)$ for all bands $B$ and band schemes $X$. The restrictions $\Gamma:\BAff\to\Bands$ and $\Spec:\Bands\to\BAff$ are mutually inverse anti-equivalences of categories.

\begin{rem}[Local nature of a band scheme]
 Since every band has a unique maximal ideal, every affine band scheme has a unique closed point. This fact has strong consequences:
 \begin{enumerate}
  \item Every closed point $x$ of a band scheme $X$ is contained in a unique affine open neighbourhood, which is $U_x=\Spec \cO_{X,x}$. In particular, if $X$ is finite, then the affine opens of the form $U_x$ (with $x\in X$ closed) form the unique minimal affine open covering of $X$.
  \item If a band morphism $f:B\to C$ induces an open immersion $f^\ast:\Spec C\to \Spec B$ of affine band scheme, then $f$ is a finite localization, i.e. $C=B[h^{-1}]$ for some $h\in B$.
  \item Let $\varphi:X\to Y$ be a morphism of band schemes and $X=\bigcup_{i\in I} U_i$ and $Y=\bigcup_{j\in J} V_j$ affine open coverings. Then there is a map $\varphi_0:I\to J$ such that $\varphi(U_i)\subseteq V_{\varphi_0(i)}$ for every $i\in I$.
 \end{enumerate}
\end{rem}

\subsubsection{Base extension to $\Z$}
\label{subsubsection: base extension to Z}
A band scheme $X$ comes with a universal scheme $X^+$ and a canonical inclusion of $R$-rational point sets $X(R)\subseteq X^+(R)$ for every ring $R$, which is a bijection if $R$ is local.

The universal scheme $X^+$ is defined as follows. If $X=\Spec B$ is affine, we define $X^+_\Z:=\Spec B^+_\Z$. For arbitrary $X$, we let $\cU$ be the diagram of all affine opens $U$ of $X$ together with all inclusions $\iota_{V,U}:V\hookrightarrow U$ as subsets of $X$, which are open immersions of affine band schemes. Since $(-)^+_\Z$ is functorial, this yields a diagram $\cU^+_\Z$ of affine schemes. We define
\[
 X^+_\Z \ = \ \colim \cU^+_\Z.
\]
As a matter of fact, all morphisms of $\cU^+_\Z$ are principal open immersions and the canonical inclusions $U^+_\Z\to X^+_\Z$ are open immersions for $U\in\cU$.

The functoriality of this construction defines an injection
\[
 X(R) \ = \ \Hom(\Spec R, X) \ \longrightarrow \ \Hom(\Spec R, X^+_\Z) \ = \ X^+_\Z(R)
\]
for every ring $R$, which is bijective if $R$ is local since in this case every morphism $\Spec R\to X^+_\Z$ factors through some affine open of the form $U^+_\Z$ with $U\in\cU$.

\subsubsection{Fibre products}
The categorical fibre product of two morphisms $\varphi_X:X\to Z$ and $\varphi_Y:Y\to Z$ is represented by a band scheme $X\times_ZY$, which comes with two projections $\pi_X:X\times_ZY\to X$ and $\pi_Y:X\times_ZY\to Y$. In the affine case $X=\Spec B$, $Y=\Spec C$ and $Z=\Spec k$, we have $X\times_ZY=\Spec (B\otimes_kC)$. In general, the fibre product $X\times_ZY$ is covered by affine opens of this form.

\subsubsection{Open and closed subschemes}
An \emph{open subscheme} of a band scheme $X$ is an open subset $U$ together with the restriction of the structure sheaf $\cO_X$ of $X$ to $U$. An \emph{open immersion} is a morphism of band schemes $\varphi:Y\to X$ that factors through an isomorphism $Y \to U$ with $U$ an open subscheme of $X$.

An \emph{affine morphism} is a morphism $\varphi:Y\to X$ of band schemes such that for every affine open $U$ of $X$, the inverse image $\varphi^{-1}(U)=Y\times_XU$ is affine. A \emph{closed immersion} is an affine morphism $\varphi:Y\to X$ such that for every affine open $U$ of $X$ and $V=\varphi^{-1}(U)$, the map $\varphi^\#(U):\Gamma U\to\Gamma V$ is surjective. A \emph{closed subscheme of $X$} is an isomorphism class of closed immersions $\varphi:Y\to X$.

\subsection{Projective geometry}
Similar to usual schemes, projective geometry for band schemes is governed by homogeneous coordinate algebras and the $\Proj$-construction. We introduce all relevant notions in the following.

\subsubsection{Projective space}
Let $k$ be a band. The free algebra $k[T_0,\dotsc,T_n]$ is graded by $\deg T_i=1$ for each $i$, and this grading is inherited by all of its localizations at elements\footnote{Note that all elements of $k[T_0,\dotsc,T_n]$ are homogeneous since they are monomials in $T_0,\dotsc,T_n$.} $h\in k[T_0,\dotsc,T_n]$. 

The \emph{projective $n$-space over $k$} is the band scheme $\P^n_k$ that consists of all prime $m$-ideals of $k[T_0,\dotsc,T_n]$ that do not contain all of $T_0,\dotsc,T_n$, endowed with the topology generated by the \emph{principal opens}
\[
 U_h \ = \ \{\fp\in \P^n_k \mid h\notin\fp\}
\]
for $h\in k[T_0,\dotsc,T_n]$, and whose structure sheaf is characterized by the values
\[
 \cO_{\P^n_k}(U_h) \ = \ \big(k[T_0,\dotsc,T_n][h^{-1}]\big)_0,
\]
on principal opens, where $(-)_0$ stands for the degree $0$ part. The projective $n$-space is covered by the canonical affine opens
\[
 U_i \ = \ U_{T_i} \ = \ \{\fp\in \P^n_k \mid T_i\notin\fp\},
\]
which are isomorphic to the affine space $\A^n_k$ via the morphism associated to the band isomorphism
\[
 \begin{array}{ccccc}
  \Gamma(\A^n_k) & = &k[T_0/T_i,\dotsc,T_n/T_i] & \longrightarrow & \big(k[T_0,\dotsc,T_n][T_i^{-1}]\big)_0 \\
& & T_j/T_i & \longmapsto & T_j/T_i.
 \end{array}
\]

\subsubsection{Projective band schemes}

A \emph{homogeneous element of $B=k[T_0,\dotsc,T_n]^+$ of degree $d$} is a polynomial $\sum a_{\underline{e}} T_1^{e_1}\dotsb T_n^{e_n}$ such that $e_1+\dotsb+e_n=d$ for all indices $\underline{e} =(e_i) \in \N^n$ with $a_{\underline{e}}\neq0$. A \emph{homogeneous null ideal of $B$} is a null ideal $I$ that is generated by homogeneous elements (of possibly different degrees).

A homogeneous null ideal $I$ of $B$ defines the closed subscheme $X=\Proj(\bandquot{B}{I})$ of $\P^n_k$ whose underlying space is the subspace
\[
 X \ = \ \big\{\fp\in\P^n_k \, \big| \, \text{if $a-b\in I$ and $a\in\fp$, then also $b\in\fp$}\big\}
\]
of $\P^n_k$ and whose structure sheaf is determined by its values
\[
 \cO_X\big(U_h\cap X\big) \ = \ \big((\bandquot{B}{I})[h^{-1}]\big)_0
\]
on the intersections with principal opens $U_h$. The inclusion $\iota:X\to\P^n_k$ together with the surjection $\iota^\#:\cO_{\P^n_k}\to\cO_X$ of sheaves is a closed immersion. Conversely, every closed subscheme of $\P^n_k$ is of this form.

\begin{ex}[Grassmannian]
\label{ex: Grassmannian}
 An example of primary interest are Grassmannians $\Gr(r,E)$ over $\Funpm$ for $E=\{1,\dotsc,n\}$ and $r$ an integer. The ambient projective space 
 \[
  \P^{n^r-1} \ = \ \Proj\big(\Funpm[T_\textbf{x}\mid \textbf{x}\in E^r]\big)
 \]
 has homogeneous coordinates $T_\textbf{x}$, indexed by $r$-tuples $\textbf{x}=(x_1,\dotsc,x_r)$ of elements of $E$. The homogeneous null ideal $I$ that defines $\Gr(r,E)$ as a closed subscheme of $\P^{n^r-1}$ is generated by the set of \emph{Pl\"ucker relations} $\mathpzc{Pl}(r,E)$, which consists of the following homogeneous elements (cf. \ref{def: mat GP1} and \ref{def: mat GP2}):
 \begin{align*}
  & T_{\textbf{x}}     && \text{whenever }\#|\textbf{x}|<r; \\
  & T_{\textbf{x}^\sigma} \ - \ \sign(\sigma)\cdot T_{\textbf{x}}   && \text{for every $\sigma\in \mathfrak{S}_r$}; \\
  & \sum_{k=1}^{r+1} (-1)^k ~ T_{\textbf{y}_{\widehat{k}}} \cdot T_{(y_k, \textbf{x})} && \text{for all } \textbf{y} \in E^{r+1} \text{ and } \textbf{x} \in E^{r-1}.
 \end{align*}
 
 We call the closed immersion $\iota:\Gr(r,E)\to\P^{n^r-1}$ the \emph{Pl\"ucker embedding} in analogy with the classical theory.
 
 Note that $\Gr(r,E)(k)$ corresponds to the classical Grassmann variety in the case of a field $k$ (cf.\ \autoref{subsubsection: base extension to Z}).
\end{ex}

\subsubsection{Multigraded bands}
Let $S$ be a finite set and $n_v\in\N$ for $v\in S$. The band $B=k[T_{v,i}\mid v\in S, i=0,\dotsc,n_v]$ is graded in $\Z^{S}$ with multidegree $\deg T_{v,i}=(\delta_{v,w})_{w\in S}$. A \emph{homogeneous element of $B^+$ with multidegree $(d_v)_{v\in S}$} is a polynomial $\sum a_e\prod T_{v,i}^{e_{v,i}}$ with $e_{v,0}+\dotsb+e_{v,n_v}=d_v$ for all indices $e=(e_{v,i})$ with $a_e\neq0$. A \emph{multi-homogeneous null ideal of $B$} is a null ideal $I$ of $B$ that is generated by multi-homogeneous elements.

A multi-homogeneous null ideal $I$ of $B$ defines the following closed subscheme $X=\Proj(\bandquot BI)$ of $\prod_{v\in S} \P^{n_v}_k$. As a subspace of $\prod \P^{n_v}_k$, 
\[\textstyle
 X \ = \ \big\{\fp\in\prod \P^{n_v}_k \, \big| \, \text{if $a-b\in I$ and $a\in\fp$, then also $b\in\fp$}\big\}.
\]
Its structure sheaf $\cO_X$ is determined by its values
\[
 \cO_X(U_{(h_v)}) \ = \ \big((\bandquot BI)[h_v^{-1}\mid v\in S]\big)_0
\]
on principal opens $U_{(h_v)}=\prod U_{h_v}$ of $\prod\P^{n_v}_k$ where $(-)_0$ stands for the subband of multidegree $0$ elements. The inclusion $\iota:X\to\prod\P^{n_v}_k$ together with the canonical projection $\iota^\#:\cO_{\prod\P^{n_v}_k}\to\cO_X$ of sheaves is a closed immersion and turns $X$ into a closed subscheme of $\prod \P^{n_v}_k$.

\begin{rem}
 The Segre embedding of usual algebraic geometry generalizes to band schemes and defines a closed immersion $\iota:\P^n_k\times\P^m_k\to \P^{nm+n+m}_k$, which identifies $\P^n_k\times\P^m_k$ with a closed subscheme of $\P^{nm+n+m}_k$. To be more precise, the Segre embedding $\iota$ is induced by the band morphism $k[T_{i,j}]\to k[T_i,T_j]$ that sends $T_{i,j}$ to $T_i\cdot T_j$ and taking $C$-points (for a $k$-algebra $C$), yields the map
 \[
  \begin{array}{cccc}
   \iota_C: & \P^n_k(C)\times\P^m_k(C) & \longrightarrow  & \P^{nm+n+m}_k(C) \\
            & \big([a_0:\dotsc:a_n],[b_0:\dotsc:b_m]\big) & \longmapsto     & [a_0b_0:\dotsc:a_nb_m].
  \end{array}
 \]
 The homogeneous null ideal of the image of $\iota$ is generated by the quadratic polynomials $T_{i,j}T_{k,l}-T_{i,l}T_{k,j}$ for $i,k\in\{0,\dotsc,n\}$ and $j,l\in\{0,\dotsc,m\}$.
 
 In particular, this shows that closed subvarieties of products of projective spaces are themselves projective varieties.
\end{rem}

\begin{ex}[Flag varieties]\label{ex: flag varieties}
 Flag varieties $\Fl(\textbf{r},E)$ over $\Funpm$ (where $\textbf{r}=(r_1,\dots,r_s)$ is the type of the flag) were introduced in \cite{Jarra-Lorscheid24} as closed subvarieties of $\prod_{i=1}^s \Gr(r_i,E)$. Quiver Grassmannians, as defined in \autoref{subsection: quiver Grassmannian} are a vast generalization of flag varieties.
 
 We exemplify this in the simplest nontrivial example, which is the flag variety $\Fl(\textbf{r},E)$ for $E=\{1,2,3\}$ and $\textbf{r}=(1,2)$. The ambient Grassmannians $\Gr(1,E)$ and $\Gr(2,E)$ are both projective planes $\P^2$ over $\Funpm$, with respective homogeneous coordinates $T_i$ (with $i\in E$ for $\Gr(1,E)$) and $T_{k,l}$ (with $k,l\in E$ and $k\neq l$ for $\Gr(2,E)$). The flag variety $\Fl(\textbf{r},E)$ is the closed subvariety of $\P^2\times\P^2$ defined by (the multi-homogeneous null ideal generated by) the single polynomial
 \[
  T_1\cdot T_{2,3} \ - \ T_2\cdot T_{1,3} \ + \ T_3\cdot T_{1,2}.
 \]
\end{ex}

\subsection{Line bundles and maps to projective space} 

In this section, we introduce $\cO_X$-modules and line bundles, and we explain that global sections that generate a line bundle define a map to a projective space.

\subsubsection{\texorpdfstring{$\cO_X$}{O_X}-modules}

Let $B$ be a band. A $B$-module is a pointed set $M$, with base point $0_M$, together with a {subset} $N_M$ of the free semigroup $M^+=\N[M-\{0_M\}]$ (called the \emph{null set of $M$}) and a map $\theta:B\times M\to M$, written as $a.m=\theta(a,m)$, such that for all $a,a_i,b\in B$ and $m,m_j,n\in M$ (where we identify $0_M$ with the empty sum in $M^+$):
\begin{enumerate}
 \item $1.m=m$, \quad $0.m=0_M$, \quad \text{and} \quad $a.0_M=0_M$;
 \item $(ab).m=a.(b.m)$;
 \item $\sum a_i.m_j\in N_M$ whenever $\sum a_i\in N_B$ or $\sum m_j\in N_M$;
 \item $m+n\in N_M$ if and only if $n=-m:=(-1).m$.
\end{enumerate}
A \emph{$B$-linear morphism between $B$-modules $M$ and $M'$} is a map $f:M\to M'$ such that $f(a.m)=a.f(m)$ and $\sum f(m_j)\in N_{M'}$ whenever $\sum m_j\in N_M$. We denote the category of $B$-modules by $\Mod_B$. Note that an $\Funpm$-module is the same as a pointed set $M$ together with a subset $N_M$ of $M^+$ that contains $0_M$ and such that for every $m\in M$, there is a unique $n\in M$ such that $m+n\in N_M$.

Let $X$ be a band scheme. An \emph{$\cO_X$-module} is a sheaf $\cM$ in $\Mod_\Funpm$ together with a morphism $\theta:\cO_X\times\cM\to\cM$ of sheaves such that $\theta(U):\cO_X(U)\times\cM(U)\to\cM(U)$ endows $\cM(U)$ with the structure of an $\cO_X(U)$-module and such that the restriction maps $\res:\cM(U)\to\cM(V)$ for opens $V\subseteq U$ are $\cO_X(U)$-linear.

\subsubsection{Line bundles}

A \emph{line bundle} is an $\cO_X$-module $\cL$ on $X$ that is locally isomorphic to $\cO_X$, i.e.\ there is an open covering $X=\bigcup U_i$ such that $\cL|_{U_i}\simeq\cO_X|_{U_i}$ as $\cO_X|_{U_i}$-modules for all $i$. We write $\Gamma\cL$ for the $\Gamma X$-module $\Gamma(X,\cL)$. 

The sheaf pullback of a line bundle $\cL$ over $X$ along a morphism $\varphi:Y\to X$ is a line bundle $\varphi^\ast\cL$ over $Y$. The pullback of global sections defines a $\Gamma X$-linear map $\Gamma\varphi:\Gamma\cL\to \Gamma(\varphi^\ast\cL)$.

\begin{ex}[Line bundles on points]\label{ex: line bundles on Spec F}
 Let $F$ be an idyll and $X=\Spec F$ its spectrum, whose unique point is the trivial ideal $\eta=\{0\}$ of $F$. A line bundle $\cL$ over $X$ is characterized by its global sections $\Gamma\cL=\cL(\{\eta\})$, which are isomorphic to $\Gamma X=F$ by the definition of a line bundle. In particular, every line bundle over $X$ is isomorphic to the trivial line bundle $\cO_X$.
\end{ex}

\begin{ex}[Twisted sheaves on projective space] \label{ex: twisted sheaves on projective space}
 Let $d\in\Z$. The \emph{twisted line bundle $\cO(d)$} over $\P^n=\Proj\Funpm[T_0,\dotsc,T_n]$ is defined as the $\cO_{\P^n}$-module with values
 \[
  \cO(d)(U_h) \ = \ \big(\Funpm[T_0,\dotsc,T_n][h^{-1}]\big)_d
 \]
 for homogeneous $h\in \F_1^\pm[T_0,\dotsc,T_n]^+$, where $(-)_d$ refers to the $\cO_{\P^n}(U_h)$-submodule of degree $d$ elements. For example, the \emph{canonical sections} $T_0,\dotsc, T_n$ of $\P^n$ are global sections of the \emph{twisting sheaf $\cO(1)$}.
\end{ex}

\subsubsection{Morphisms to projective space}

A subset $\{s_i\}\subseteq\Gamma \cL$ of global sections is said to \emph{generate $\cL$} if for every $x\in X$, there is an $i$ such that the class of $s_i$ is a unit in $\cO_{X,x}$. Let $\cL$ and $\cL'$ be two line bundles together with respective ordered sets of global sections $(s_0,\dotsc,s_n)$ and $(s_0',\dotsc,s_n')$. We say that $\big(\cL,(s_i)\big)$ and $\big(\cL',(s_i')\big)$ are equivalent if there is an isomorphism $\theta:\cL\to\cL'$ of line bundles such that $\theta(s_i)=s_i'$ for all $i$. We denote by $\mathpzc{M}_n(X)$ the set of all equivalence classes of such $(\cL,(s_i))$ for which $\{s_0,\dotsc,s_n\}$ generates $\cL$.
 
Let $\cO(1)$ be the twisting sheaf on $\P^n=\Proj   (\Funpm[T_0,\dotsc,T_n])$ and $T_0,\dotsc,T_n\in\Gamma\cO(1)$ the canonical sections (see \autoref{ex: twisted sheaves on projective space}). The following result is \cite[Thm.\ 4.20]{Baker-Lorscheid21}, specialized from ordered blue schemes to band schemes.

\begin{prop}\label{prop: characterization of morphisms to projective space by global sections of line bundles}
 Let $X$ be a band scheme and $T_0,\dotsc,T_n$ the canonical global sections of $\P^n$. Then pulling back $T_0,\dotsc,T_n$ along a morphism $\chi:X\to\P^n$ yields a functorial bijection
 \[
  \begin{array}{ccc}
   \Hom(X,\P^n)  & \longrightarrow & \mathpzc{M}_n(X) \\
   \chi:X\to\P^n & \longmapsto     & \big[\chi^\ast\cO(1),(\Gamma\chi(T_i))\big].
  \end{array}
 \]
\end{prop}

\section{Matroid bundles}

Matroid bundles for ordered blue schemes are introduced in \cite{Baker-Lorscheid21}. In this section, we transfer this definition to band schemes, reformulate the moduli property of Grassmannians in this setting and introduce morphisms for matroid bundles, which is a novel concept.

\subsection{Definition}
Let $X$ be a band scheme, $r$ an integer and $E=\{1,\dotsc,n\}$. A \emph{Grassmann-Pl\"ucker function over $X$ (of rank $r$ on $E$)} is a line bundle $\cL$ over $X$ together with a function
\[
 \Delta: \ E^r \ \longrightarrow \ \Gamma\cL
\]
that satisfies the following properties:
\begin{enumerate}
 \item $\{\Delta(\textbf{x}) \mid \textbf{x} \in E^r\}$ generates $\cL$; \\[-10pt]
 \item $\Delta(\textbf{x})=0$ if $\#|\textbf{x}|<r$; \\[-10pt]
 \item $\Delta(\textbf{x})=\sign(\sigma)\cdot \Delta(\textbf{x}^\sigma)$ for every $\sigma\in \mathfrak{S}_r$; \\[-10pt]
 \item for all $x_2,\dotsc,x_r,y_0,\dotsc,y_r\in E$, 
 \[
  \sum_{k=0}^r \ (-1)^k \cdot \Delta(y_0,\dotsc,\widehat{y_k},\dotsc,y_r) \otimes \Delta(y_k,x_2,\dotsc,x_r) \quad \in \quad N_{\Gamma\cL^{\otimes2}}.
 \]
\end{enumerate}

Two Grassmann-Pl\"ucker functions $\mu:E^r\to\Gamma\cL$ and $\mu':E^{r}\to\Gamma\cL'$ (on the same set $E$, over the same band scheme $X$ and with same rank $r$) are \emph{equivalent} if there exists an isomorphism $\varphi:\cL\to\cL'$ such that $\mu'=\Gamma\varphi\circ\mu$. A \emph{matroid bundle over $X$} is an equivalence class $\cM=[\mu:E^r\to\Gamma\cL]$ of a Grassmann-Pl\"ucker function $\mu$ over $X$.

Given a matroid bundle $\cM=[\mu:E^r\to\Gamma\cL]$ over $X$ and a morphism of band schemes $\theta: Y \rightarrow X$, we define the \emph{pullback} of $\cM$ along $\theta$ as the matroid bundle $\theta^*\cM$ over $Y$ given by the equivalence class of the composition 
\[
E^r \overset{\mu}{\longrightarrow} \Gamma\cL \overset{\Gamma \theta}{\longrightarrow} \Gamma (\theta^*\cL).
\]

\begin{ex}[Matroid bundles on points]\label{ex: matroid bundles on points}
 Let $F$ be an idyll, $X=\Spec F$ its spectrum and $\cM=[\mu:E^r\to \Gamma\cL]$ a matroid bundle over $X$. As discussed in \autoref{ex: line bundles on Spec F}, every line bundle over $X$ is isomorphic to $\cO_X$. Thus we may assume that $\cL=\cO_X$ and $\Gamma\cL=F$, which identifies $\mu:E^r\to F$ as a Grassmann-Pl\"ucker function over $F$ (see \autoref{Baker-Bowler theory}). Since every automorphism of $\cO_X$ is given by multiplication with an element in $F^\times$, the class of $\mu$ as a matroid bundle over $X$ and that as an $F$-matroid coincide.
 
 This shows that matroid bundles over $\Spec F$ correspond bijectively and in a canonical way to $F$-matroids.
\end{ex}

\subsection{The moduli property of the Grassmannian}

Let $r$ be an integer, $E:=\{1,\dotsc,n\}$ and $\Gr(r,E)$ the Grassmannian over $\Funpm$, which comes with the Pl\"ucker embedding $\iota:\Gr(r,E)\to\P^{n^r-1}$ into projective space (see \autoref{ex: Grassmannian}). Let $\cO(1)$ be the twisting sheaf on $\P^{n^r-1}$ and $\cL^\univ=\iota^\ast\cO(1)$ its pullback to $\Gr(r,E)$, which we call the \emph{universal line bundle of $\Gr(r,E)$}. The \emph{canonical sections of $\cL^\univ$} are the global sections $s_{\textbf{x}}=\Gamma\iota(T_{\textbf{x}})$ of $\cL^\univ$ for ${\textbf{x}}\in E^r$. 

The following result identifies the Grassmannian $\Gr(r,E)$ with the fine moduli space of matroid bundles of rank $r$ on $E$ (cf.\ \cite[Thm.\ 5.5]{Baker-Lorscheid21}).

\begin{thm}\label{thm: Grassmannian as moduli space of matroid bundles}
 Let $X$ be a band scheme. Sending a morphism $\chi:X\to\Gr(r,E)$ to the matroid bundle $\big[\mu_\chi:E^r\to \Gamma\chi^\ast\cL^\univ\big]$ over $X$, given by $\mu_\chi({\mathbf{x}})=\Gamma\chi(s_{\mathbf{x}})$, defines a bijection
 \[
  \begin{array}{cccc}
   \Phi: & \Hom\big(X,\, \Gr(r,E)\big) & \longrightarrow & \Big\{ \begin{array}{c}\text{matroid bundles of rank $r$}\\ \text{with ground set $E$ over $X$}\end{array} \Big\} 
  \end{array}
 \]
 that is functorial in $X$. If $\Phi(\chi)=\cM$, then we call $\chi$ the \emph{characteristic morphism of $\cM$}.
\end{thm}

\begin{rem}\label{rem: matroid bundles over fields}
 Let $k$ be a field, considered as an idyll, and $X$ a band $k$-scheme. By the functoriality of the base extension $(-)^+$ from band $k$-schemes to (usual) $k$-schemes, we have a canonical injection
 \[
  \Hom_{\cat{BSch}}\big(X,\, \Gr(r,E)\big) \ \longrightarrow \ \Hom_{k}\big( X^+, \, \Gr(r,E)^+_k\big).
 \]
 where $\Gr(r,E)^+_k$ is the usual Grassmann variety over $k$. Together with \autoref{thm: Grassmannian as moduli space of matroid bundles}, this identifies rank $r$ matroid bundles over $X$ on $E=\{1,\dotsc,n\}$ with rank $r$ subbundles of the trivial rank $n$ vector bundle $\cO^n_{X^+}$ over $X^+$.
\end{rem}

\subsection{Morphisms of matroid bundles}
\label{subsection: morphisms of matroid bundles}

Since the notion of vectors for matroid bundles is intricate (see \autoref{rem: composition of morphisms}), it is better for our purposes to describe morphism of matroid bundles in terms of Grassmann-Pl\"ucker functions (cf. \autoref{thm: cryptomorphisms}).

In analogy to the case of perfect idylls, we consider submonomial matrices with entries in a band $B$ in the following. Each submonomial $B$-matrix $\Phi$ induces an $\F_1$-linear map $\underline{\Phi}$ with $\uPhi(j)=i$ if $\Phi_{i,j}\neq0$ and $\uPhi(j)=0$ if $j=0$ or if $\uPhi_{i,j}=0$ for all $i$. Conversely, if $B \neq 0$, each $\F_1$-linear map $f$ induces a unique submonomial $B$-matrix $f_B$ with entries in $\{0,1\}$ that satisfies $\underline{f_B} = f$.

Let $\cN=[\nu:S^w\to\Gamma\cL_\cN]$ and $\cM=[\mu: T^r\to \Gamma\cL_\cM]$ be matroid bundles over $X$. A \emph{morphism $\Phi:\cN\to\cM$ of matroid bundles} is a submonomial $\Gamma X$-matrix $\Phi$ indexed by $T\times S$ such that
\[
 \sum_{k=0}^{w} (-1)^k \cdot \Phi_{\uPhi(y_k),y_k} \cdot \nu(y_0,\dotsc,\widehat{y_k},\dotsc,y_w) \otimes \widetilde{\mu}\big(\underline{\Phi}(y_k),x_2,\dotsc,x_r\big) \quad \in \quad N_{\Gamma (\cL_\cN\otimes\cL_\cM)}
\]
for all $y_0,\dotsc,y_w\in S$ and $x_2,\dotsc,x_r\in T$, where $\widetilde{\mu}: \underline{T}^r \to \Gamma \cL_{\cM}$ is given by $\widetilde{\mu}(\textbf{z}) := \mu(\textbf{z})$ if $\textbf{z} \in T^{r}$ and $\widetilde{\mu}(\textbf{z}):= 0$ if $0 \in |\textbf{z}|$ (cf.~\autoref{df: pointed Grassmann-Plucker function}).

\begin{ex}[Morphisms over points]
\label{ex: morphisms over points}
 Let $F$ be a perfect idyll, $X=\Spec F$ and $\cN=[\nu:S^w\to\Gamma\cO_X]$ and $\cM=[\mu: T^r\to \Gamma\cO_X]$ two matroid bundles over $X$ (recall from \autoref{ex: matroid bundles on points} that every line bundle over $X$ is trivial). Since $\Gamma\cO_X=F$, the Grassmann-Pl\"ucker function $\nu:S^w\to F$ defines an $F$-matroid $N$ and $\mu:T^r\to F$ defines an $F$-matroid $M$. By the very definition of morphisms for $F$-matroids and for matroid bundles, a submonomial matrix $\Phi$ indexed by $T\times S$ is a morphism of matroid bundles $\Phi:\cN\to\cM$ if and only if it is a morphism of $F$-matroids $\Phi:N\to M$.
\end{ex}

\begin{rem}[Composition of morphisms]\label{rem: composition of morphisms}
 Morphisms of matroid bundles are in general not composable. This fails already for matroids over idylls that are not perfect (see \cite[Ex. 2.19]{Jarra-Lorscheid24}).
  
 This failure goes back to the fact that vectors are in general not orthogonal to covectors. This issue can be resolved by an intricate categorical construction, which is the theme of a forthcoming work. In particular, this leads to a richer class of morphisms, which are composable and which contain morphisms as defined in this text as a subclass.
\end{rem}

\subsection{The dual of a matroid bundle}
\label{subsection: the dual of a matroid bundle}
Let $E = \{1,\dotsc,n\}$ and $r$ an integer. For each $\textbf{x} \in E^{r^*}$ with $\#|\textbf{x}| = r^* := n-r$, we fix $\textbf{x}' \in E^r$ such that $|\textbf{x}| \sqcup |\textbf{x}'| = E$. Let $X$ be a band scheme and $\cM=[\mu:E^r\to\Gamma\cL]$ a matroid bundle over $X$. The dual of $\cM$ is the matroid bundle $\cM^\vee:=[\mu^\vee: E^{r^\ast}\to\Gamma\cL]$ represented by the Grassmann-Pl\"ucker function $\mu^\vee$ given by
\[
\mu^\vee({\bf x}) \ := \ \begin{cases}
                        0 & \text{ if } \#|\textbf{x}| < r^*\\
                        \sign({\bf x},{\bf x}')\cdot \varphi({\bf x}') &\text{ otherwise}
\end{cases}
\]
where, for $\textbf{z} \in E^n$, $\sign(\textbf{z})$ is the sign of the unique permutation $\sigma \in \mathfrak{S}_n$ such that $z_{\sigma(1)}<\dots<z_{\sigma(n)}$ if $|\textbf{z}| = E$, and $0$ if $\#|\textbf{z}| < n$. 

Duality of matroid bundles is induced by an isomorphism between Grassmannians. The isomorphism
\[
\begin{array}{ccc}
\bandgenquot{\Funpm[T_{\bf x}\mid {\bf x}\in E^{r^\ast}]}{\mathpzc{Pl}(r^\ast,E)} & \longrightarrow & \bandgenquot{\Funpm[T_{\bf x}\mid {\bf x}\in E^r]}{\mathpzc{Pl}(r,E)} \\
&&\\
T_{\bx} & \longmapsto & \begin{cases}
                        0 & \text{if } \#|\bx| < r^\ast\\
                        \sign(\bx,\bx')\cdot T_{\bx'} & \text{otherwise}
                        \end{cases}
\end{array}
\]
of graded bands induces an isomorphism $\delta_{r,E}:\Gr(r,E)\to\Gr(r^\ast,E)$ of band schemes. The following is \cite[Prop.\ 5.8]{Baker-Lorscheid21}.

\begin{prop}\label{prop: duality of matroid bundles and their characteristic morphisms}
 Let $X$ be a band scheme and $\cM$ a matroid bundle with characteristic morphism $\chi_\cM:X\to\Gr(r,E)$. Let $\cM^\vee$ be the dual of $\cM$. Then the characteristic morphism of $\cM^\vee$ is $\chi_{\cM^\vee}=\delta_{r,E}\circ\chi_\cM$, i.e.\ the diagram
 \[
  \begin{tikzcd}[column sep=80pt, row sep=5pt]
                                                           & \Gr(r,E) \ar[dd,"\delta_{r,E}"] \\
   X \ar[ru,"\chi_\cM"] \ar[rd,pos=0.6,"\chi_{\cM^\vee}"'] \\
                                                           & \Gr(r^\ast,E)
  \end{tikzcd}
 \]
 commutes.
\end{prop}

\subsection{The dual of a morphism}

Duality for matroid bundles extends to a duality on morphisms.

\begin{prop}\label{prop: duality for morphisms of matroid bundles}
 Let $X$ be a band scheme and $\Phi:\cN\to\cM$ a morphism of matroid bundles from $\cN=[\nu:S^w\to\Gamma\cL_\cN]$ to $\cM=[\mu:T^r\to\Gamma\cL_\cM]$ over $X$. Then the transpose of $\Phi$ is a morphism $\Phi^t:\cM^\vee\to\cN^\vee$ of matroid bundles.
\end{prop}

\begin{proof}
Let $\cL := \cL_{\cM^\vee} \otimes \cL_{\cN^\vee} = \cL_\cN \otimes \cL_\cM$. Let $\textbf{y} \in T^{r^*+1}$ and $\textbf{x} \in S^{w^*-1}$. If $|\textbf{y}| < r^*+1$ or $|\textbf{x}| < w^*-1$, then 
\[
\underset{k = 1}{\overset{w+1}{\sum}} (-1)^k ~ \mu^\vee({\bf y}_{\widehat{k}}) ~ \otimes {\Phi^t_{y_k}} ~ \widetilde{\nu^\vee} \big(\underline{\Phi^t}(y_k),{\bf x}\big) \; \in \; N_{\Gamma \cL}.
\]

Suppose that $|\textbf{y}| = r^*+1$ and $|\textbf{x}| = w^*-1$, and fix $\textbf{z} \in T^{r-1}$ and $\textbf{u} \in S^{w+1}$ such that $|\textbf{y}| \sqcup |\textbf{z}| = T$ and $|\textbf{x}| \sqcup |\textbf{u}| = S$. Let $U := |\textbf{u}| \cap \underline{\Phi^t}(|{\bf y}|) \subseteq S$ and define
\[
  I \ := \ \big\{i \in \{1,\dotsc, w+1\}\, \big| \, u_i\in U\big\} \quad \text{and} \quad J \ := \ \big\{j \in \{1,\dotsc, n-r+1\} \,\big|\, \underline{\Phi}(y_j)\in U\big\}.  
\]
There exists a bijection $\mathfrak{b}: J\rightarrow I$ such that $\underline{\Phi^t}(y_j)=u_{\mathfrak{b}(j)}$, for each $j\in J$. Note that
\[
\begin{aligned}
    & \underset{k = 1}{\overset{w+1}{\sum}} (-1)^k ~ \mu^\vee({\bf y}_{\widehat{k}}) ~ \otimes {\Phi^t_{y_k}} ~ \widetilde{\nu^\vee} \big(\underline{\Phi^t}(y_k),{\bf x}\big)
\\
    = & \underset{j \in J}{\sum} (-1)^j ~ \mu^\vee({\bf y}_{\widehat{j}}) ~ \otimes {\Phi^t_{y_j}} ~ {\nu^\vee} \big(\underline{\Phi^t}(y_j),{\bf x}\big)
\\
    = & \; \underset{j \in J}{\sum} (-1)^j ~ \sign({\textbf{y}}_{\widehat{j}}, y_j, \textbf{z}) ~ \mu(y_j, \textbf{z}) ~ \otimes {\Phi_{u_{\mathfrak{b}(j)}}} ~ \sign\big(u_{\mathfrak{b}(j)}, {\textbf{x}}, \textbf{u}_{\widehat{\mathfrak{b}(j)}}\big) ~ {\nu}(\textbf{u}_{\widehat{\mathfrak{b}(j)}})
\\
    = & \; (-1)^{r^*+1} \sign({\textbf{y}}, \textbf{z}) \cdot \underset{i \in I}{\sum}  ~ \mu(y_{\mathfrak{b}^{-1}(i)}, \textbf{z}) ~ \otimes {\Phi_{u_i}} ~ \sign\big(u_i, {\textbf{x}}, \textbf{u}_{\widehat{i}}\big) ~ {\nu}(\textbf{u}_{\widehat{i}})
\\
    = & \; (-1)^{r^*+1} \sign({\textbf{y}}, \textbf{z}) \cdot (-1)^{w^*-1} \sign({\textbf{x}}, \textbf{u}) \cdot \underset{i \in I}{\sum}  ~ (-1)^i ~ {\nu}(\textbf{u}_{\widehat{i}}) ~ \otimes {\Phi_{u_i}} ~ \mu(\underline{\Phi}(u_i), \textbf{z}),
\end{aligned}
\]
which is in $N_{\Gamma \cL}$, because, as $\Phi: \cN \rightarrow \cM$ is a morphism, one has
\[
\begin{aligned}
    & \; \underset{i \in I}{\sum}  ~ (-1)^i ~ {\nu}(\textbf{u}_{\widehat{i}}) ~ \otimes {\Phi_{u_i}} ~ \mu(\underline{\Phi}(u_i), \textbf{z})
\\
    = & \; \underset{k=1}{\overset{w+1}{\sum}}  ~ (-1)^k ~ {\nu}(\textbf{u}_{\widehat{k}}) ~ \otimes {\Phi_{u_k}} ~ \widetilde{\mu}(\underline{\Phi}(u_k), \textbf{z}) \; \in \; N_{\Gamma \cL},
\end{aligned}
\]
which concludes the proof.
\end{proof}


\section{The moduli space of quiver matroids}
\label{section: The moduli space of quiver matroids}

In \cite[Section 3]{Jarra-Lorscheid24}, the first and second authors construct the moduli space for flag matroids with coefficients. Here we generalize this construction to the case of \emph{quiver matroids}. In this section we assume that every quiver is finite, i.e. it has a finite set of vertices and arrows.

\subsection{Classical quiver Grassmannians}
\label{subsection: classical quiver Grassmannians}
Quiver Grassmannians appeared for the first time in a paper of Schofield (\cite{Schofield1992}), but they did not gain much attention until Caldero and Chapoton exhibit in \cite{Caldero-Chapoton06} an explicit formula for cluster variables in terms of the Euler characteristics of quiver Grassmannians, which turned quiver Grassmannians and their Euler characteristics into objects of central interest for representation theory. 

Let $k$ be a field, $Q$ be a quiver and $\Lambda$ a finite dimensional representation of $Q$ over $k$. Let $\textbf{e}=(e_v)_{v\in Q_0} \in \Z^{Q_0}$ be a dimension vector for $Q$. As a set, the \emph{quiver Grassmannian} $\Gr_{\textbf{e}}(\Lambda)(k)$ is defined as the collection of all subrepresentations $M$ of $\Lambda$ with dimension vector $\textbf{e}$. It gains its structure as a projective variety from the embedding
\[
    \begin{tikzcd}[row sep = -5pt, /tikz/column 1/.append style={anchor=base east}, /tikz/column 2/.append style={anchor=base west}]
        \Gr_{\textbf{e}}(\Lambda)(k) \arrow[r, hook] & \displaystyle\prod_{v\in Q_0}\Gr(e_v, \Lambda_v)(k) \\
        (M_v)_{v\in Q_0} \arrow[r, mapsto] & (M_v)_{v \in Q_0}
    \end{tikzcd}
\]
as a closed subset of the product of classical Grassmann varieties $\Gr(e_v, \Lambda_v)(k)$ over $k$.

The quiver Grassmannian gains the structure of a $k$-scheme in terms of the following set of defining equations (see \cite{Lorscheid-Weist2019}). For each vertex $v$, let $\cB_v$ be an ordered basis of $\Lambda_v$, which yields homogeneous coordinates $[\dotsc:\Delta_I:\dotsc]$ for the Grassmannian $\Gr(e_v, \Lambda_v)$, where $I$ is an $e_v$-subset of $\cB_v$. For each arrow $\alpha$, let $[a_{i,j}]$ be the matrix which represents $\Lambda_\alpha$ with respect to the bases $\cB_{s(\alpha)}$ and $\cB_{t(\alpha)}$. For $I \subseteq \cB_v$ and $i \in \cB_v$, define 
\[
 \epsilon(i,I) \ = \ \#\, \{i'\in I \mid i'\le i\}. 
\]
The \emph{quiver Pl\"ucker relations} are the equations
\[
    E(\alpha,I,J): \quad \sum_{i\in\cB_v-I,\, j\in J}(-1)^{\epsilon(i,I)+\epsilon(j,J)} ~ a_{i,j} ~ \Delta_{I\cup\{i\}} ~ \Delta_{J-\{j\}} \; = \;0,
\]
where $\alpha$ varies through $Q_1$, $I$ varies through all $(e_{s(\alpha)}-1)$-subsets of $\cB_{s(\alpha)}$ and $J$ varies through all $(e_{t(\alpha)}+1)$-subsets of $\cB_{t(\alpha)}$. These equations define $\Gr_{\textbf{e}}(\Lambda)$ as a closed subscheme of the product of Grassmannians $\prod \Gr(e_v, \Lambda_v)$.

\subsection{The moduli functor of quiver matroids} 
In the following, we introduce quiver matroid bundles and show that the quiver Pl\"ucker relations define a band model of the quiver Grassmannian, which turns out to be the moduli space of quiver matroid bundles.

For the remainder of this section, we fix a finite quiver $Q$ and a representation $\Lambda$ of $Q$ over $\cat{Vect}_{\F_1}$. We assume that $E_v:=\Lambda_v-\{0\}$ is finite for all $v\in Q_0$.

Let $B$ be a band. We denote by $\Lambda_{\alpha,B}$ the submonomial $B$-matrix with coefficients
\[
 (\Lambda_{\alpha,B})_{i,j} \ = \ \begin{cases}
                                1 & \text{if }\Lambda(j)=i\neq0, \\
                                0 & \text{if there is no $i\in E_{t(\alpha)}$ such that $\Lambda(j)=i$},
                                \end{cases}
\]
where $(i,j)$ ranges through $E_{t(\alpha)}\times E_{s(\alpha)}$; cf.\ \autoref{subsection: morphisms of matroid bundles}.

\begin{df} 
 Let $X$ be a band scheme. A \emph{$\Lambda$-matroid bundle over $X$}, or $(\Lambda,X)$-matroid bundle for short, is a tuple $\cM=(\cM_v)_{v \in Q_0}$ of matroid bundles $\cM_v$ over $X$ with ground set $E_v$ such that the submonomial $\Gamma X$-matrix $\Lambda_{\alpha,\Gamma X}: \cM_{s(\alpha)} \to \cM_{t(\alpha)}$ is a morphism for all $\alpha \in Q_1$. The \emph{rank vector} of $\cM$ is the tuple $\big(\rk(\cM_v)\big)_{v \in Q_0}$.
\end{df}

Let $\textbf{r} = (r_v)_{v \in Q_0} \in \Z^{Q_0}$ be a dimension vector for $Q$. We define the functor $\mathpzc{Gr}_{\textbf{r}}(\Lambda)(-): \cat{BSch}^{\textup{op}} \rightarrow \cat{Sets}$ as follows: for a band scheme $X$, let
\[
\mathpzc{Gr}_{\textbf{r}}(\Lambda)(X)  \; := \; \{(\Lambda, X)\text{-matroid bundles of rank }\textbf{r}\},
\]
and for a morphism of band schemes $\varphi: X \rightarrow Y$, define
\[
\begin{array}{cccc}
\mathpzc{Gr}_{\textbf{r}}(\Lambda)(\varphi): & \mathpzc{Gr}_{\textbf{r}}(\Lambda)(Y) & \longrightarrow & \mathpzc{Gr}_{\textbf{r}}(\Lambda)(X)\\
& \cM = (\cM_v)_{v\in Q_0} & \longmapsto & \varphi^*(\cM) := (\varphi^*\cM_v)_{v\in Q_0}.
\end{array}
\]
In \autoref{thm: moduli space of quiver matroid bundles}, we prove that this functor is representable.

\subsection{Quiver Grassmannian of matroids}
\label{subsection: quiver Grassmannian}

Let $Q$ be a quiver, $\Lambda$ an $\Fun$-representation of $Q$ and $\textbf{r}=(r_v)_{v\in Q_0}$ a rank vector for $Q$. We define the multi-homogeneous band
\[
 B \ = \ \bigotimes_{v \in Q_0} \ \bandquot{\F_1^\pm[T_\mathbf{x} \mid \mathbf{x} \in E_v^{r_v}]}{\qr_{\Lambda,\textbf{r}}}
\]
with grading $\deg T_\bx=(\delta_{v,w})_{w\in Q_0}$ for $\bx\in E_v^{r_v}$, where $\qr_{\Lambda,\textbf{r}}$ is generated by the following elements:
\begin{align}
 \tag{qr1}\label{qr1} 
 & T_{{\bf x}} 
 && \text{for all $\mathbf{x} \in E_v^{r_v}$ with $\#|\mathbf{x}| < r_v$;}
 \\
 \tag{qr2}\label{qr2} 
 & T_{{\bf x}^\sigma} \ - \ \sign(\sigma)\cdot T_{{\bf x}} 
 && \text{for all $v \in Q_0$, $\mathbf{x} \in E_v^{r_v}$ and permutation $\sigma$;} 
 \\
 \tag{qr3}\label{qr3} 
 & \sum_{k=1}^{r_{v}+1} \ (-1)^k \cdot T_{{\bf y}_{\widehat{k}}} \otimes T_{(y_k,{ \bf x })} 
 && \text{for all $v \in Q_0$, $\mathbf{y} \in E_v^{r_v+1}$ and $\mathbf{x} \in E_v^{r_v-1}$;}
 \\
 \tag{qr4}\label{qr4} 
 & \sum_{k=1}^{r_{s(\alpha)}+1} \ (-1)^k \cdot T_{{\bf y}_{\widehat{k}}} \otimes T_{(\Lambda_{\alpha}(y_k),{ \bf x })}
 && \text{for all $\alpha \in Q_1$, $\mathbf{y} \in E_{s(\alpha)}^{r_{s(\alpha)}+1}$ and $\mathbf{x} \in E_{t(\alpha)}^{r_{t(\alpha)}-1}$,}
\end{align}
where $T_{(0, \textbf{x})}:=0$ for $\textbf{x} \in E^{r_v-1}_v$ in \eqref{qr4}.

\begin{rem}
 Since $(\Lambda_{\alpha,\Gamma X})_{\Lambda_\alpha(y), y}$ is equal to $1$ for $T_{(\Lambda_{\alpha}(y),{ \bf x })}\neq 0$, we can omit the factor $(\Lambda_{\alpha,\Gamma X})_{\Lambda_\alpha(y), y}$ from the quiver Pl\"ucker relation \eqref{qr4}, in contrast to the more general case considered in \autoref{subsection: morphisms of matroid bundles}. 
\end{rem}

\begin{df}
\label{def: quiver Grassmannian}
The \emph{quiver Grassmannian for $\Lambda$ with rank vector $\mathbf{r}$} is the closed subscheme
\[
 \Gr_{\textbf{r}}(\Lambda) \ = \ \Proj \Big( \bigotimes_{v \in Q_0} \ \bandquot{\F_1^\pm[T_\mathbf{x} \mid \mathbf{x} \in E_v^{r_v}]}{\qr_{\Lambda,{\textbf{r}}}} \Big)
\]
of $\mathbb{P}(\Lambda)=\prod_{v\in Q_0}\P^{\#E_v^{r_v}-1}$. The closed immersion
\[
\pl: \Gr_{{\textbf{r}}}(\Lambda) \ \longrightarrow \ \mathbb{P}(\Lambda)
\]
is called the \emph{Pl\"ucker embedding}.
\end{df}

\begin{rem}
 By \eqref{qr1}--\eqref{qr3}, the Pl\"ucker embedding above factors through 
 \[
  \Gr_{\textbf{r}}(\Lambda) \ \longrightarrow \ \prod_{v \in Q_0} \Gr(r_v, E_v) \ \longrightarrow \ \mathbb{P}(\Lambda).
 \]
\end{rem}

\begin{ex}[usual Grassmannians and flag varieties]\label{ex: Quiver Grass.}
    Quiver Grassmannians are a vast generalization of Grassmannians. Namely, a Grassmannian $\Gr(r,E)$, as in \autoref{ex: Grassmannian}, is a quiver Grassmannian $\Gr_{\br}(\Lambda)$, where $\Lambda$ is the $\F_1$-representation of the quiver with a single vertex $v$ and no arrows that satisfies $\Lambda_v = \uE$, and $\br = (r)$.
    
    Quiver Grassmannians also generalize flag varieties. To illustrate, consider the quiver
    \begin{equation*}
        \begin{tikzcd}[column sep = 25pt] Q: ~ v_1 \arrow["\alpha"]{r} & v_2 \end{tikzcd}
    \end{equation*}
    which is a Dynkin quiver of type $A_2$. Let $\Lambda$ be the $\Fun$-representation of $Q$ given by $\Lambda_{v_1}=\{0,1,2,3\}$, $\Lambda_{v_2}=\{0,4,5,6\}$ and $\Lambda_\alpha:\Lambda_{v_1}\to\Lambda_{v_2}$ with
    \[
     \Lambda_\alpha(1) \ = \ 4, \qquad \Lambda_\alpha(2) \ = \ 5, \qquad \Lambda_\alpha(3) \ = \ 6.
    \]
 The $\Fun$-representation $\Lambda$ is also determined by its coefficient quiver
 (cf. \autoref{subsection: the Euler characteristic and the Tits space}), as illustrated in \autoref{fig: coefficient quiver A_2-1}.
\begin{figure}[ht]
  \begin{tikzpicture}[scale=1, x=2cm, y=0.5cm, font=\normalsize, vertices/.style={draw, fill=black, circle, inner sep=0pt},decoration={markings,mark=at position 0.5 with {\arrow{>}}}]
   \draw (1,6) node (1) {$1$};
   \draw (1,5) node (2) {$2$};
   \draw (1,4) node (3) {$3$};
   \draw (2.5,6) node (4) {$4$};
   \draw (2.5,5) node (5) {$5$};
   \draw (2.5,4) node (6) {$6$};
   \draw[->] (1) -- (4);
   \draw[->] (2) -- (5);
   \draw[->] (3) -- (6);
   \draw (1,2)     node (v1) {$v_1$};
   \draw (2.5,2)   node (v2) {$v_2$};
   \draw[->] (v1) -- node[above] {\footnotesize $\alpha$} (v2);
   \draw (-.4,5) node (G) {$\Gamma$};
   \draw (-.4,2) node (Q) {$Q$};
   \draw[->] (G) -- (Q);
   \end{tikzpicture}
    \caption{Coefficient quiver $\Gamma$ of $\Lambda$, and the projection to $Q$}
    \label{fig: coefficient quiver A_2-1}
\end{figure}
  
Consider the rank vector $\br = (r_1, r_2)=(1,2)$. Then
$$
    \Gr_{\br}(\Lambda) \ = \ \Proj\big(\bandquot{\Funpm[T_1,T_2,T_3,T_{45},T_{46},T_{56}]}{\langle T_1T_{56}-T_2T_{46}+T_3T_{45}\rangle}\big),
$$
where $T_1,T_2$ and $T_3$ have multidegree $(1,0)$ and $T_{45},T_{46}$ and $T_{56}$ have multidegree $(0,1)$. This recovers the flag variety $\Fl(\textbf{r},E)$ from \autoref{ex: flag varieties} as the quiver Grassmannian $\Gr_{\br}(\Lambda)$.
\end{ex}

\begin{ex}[Degenerate flag variety]
 Degenerate flag varieties (as in \cite{CerulliIrelli-Feigin-Reineke12}) also have $\Funpm$-models as quiver Grassmannians. The simplest interesting example is the following. 

 Let $Q$ be the same quiver as in \autoref{ex: Quiver Grass.}, which has two vertices $v_1$ and $v_2$ and a unique arrow $\alpha:v_1\to v_2$. Let $\Lambda$ be an $\Fun$-representation of $Q$ with $\Lambda_{v_1}=\{0,1,2\}$, $\Lambda_{v_2}=\{0,3,4\}$ and $\Lambda_\alpha:\Lambda_{v_1}\to\Lambda_{v_1}$ given by
 \[
  \Lambda_\alpha(1) \ = \ 3, \qquad \Lambda_\alpha(2) \ = 0,
 \]
 whose coefficient quiver (cf.\ \autoref{subsection: the Euler characteristic and the Tits space}) is illustrated in \autoref{fig: coefficient quiver A_2-2}.
 
\begin{figure}[ht]
  \begin{tikzpicture}[scale=1, x=2cm, y=0.5cm, font=\normalsize, vertices/.style={draw, fill=black, circle, inner sep=0pt},decoration={markings,mark=at position 0.5 with {\arrow{>}}}]
   \draw (1,5) node (1) {$1$};
   \draw (1,4) node (2) {$2$};
   \draw (2.5,5) node (3) {$3$};
   \draw (2.5,4) node (4) {$4$};
   \draw[->] (1) -- (3);
   \draw (1,2)     node (v1) {$v_1$};
   \draw (2.5,2)   node (v2) {$v_2$};
   \draw[->] (v1) -- node[above] {\footnotesize $\alpha$} (v2);
   \draw (-.4,4.5) node (G) {$\Gamma$};
   \draw (-.4,2) node (Q) {$Q$};
   \draw[->] (G) -- (Q);
   \end{tikzpicture}
    \caption{Coefficient quiver $\Gamma$ of $\Lambda$, and the projection to $Q$}
    \label{fig: coefficient quiver A_2-2}
\end{figure}

Consider the rank vector $\br = (r_1,r_2)=(1,1)$. The Pl\"ucker embedding
$$
    \begin{tikzcd}
        \Gr_{\br}(\Lambda)=\Proj\big(\bandquot{\Funpm[T_1,T_2,T_3,T_4]}{\langle T_1T_4\rangle}\big) \arrow[hookrightarrow]{r} & \P^1 \times \P^1.
    \end{tikzcd}
$$
identifies the quiver Grassmannian $\Gr_{\br}(\Lambda)$ with the closed subscheme of $\P^1\times\P^1$ that is the union of two projective lines (given by $T_1=0$ and $T_4=0$, respectively) that meet in one point (where $T_1=T_4=0$). In other words,
\[
 \Gr_\br(\Lambda,F) \ = \ \big\{ \, [a_1:a_2|a_3:a_4] \in\P^1(F)\times\P^1(F) \ \big| \ a_1\cdot a_4=0 \, \big\}
\]
for every idyll $F$.
\end{ex}

\begin{rem}[Affine covering of the Quiver Grassmannian of matroids]
\label{subsubsection - Affine covering of the Quiver Grassmannian of matroids}

For each tuple $\textbf{J} = (J_v)_{v \in Q_0}$, where $J_v \in E_v^{r_v}$ for all $v$, define the affine band scheme
\[
\mathbb{A}_{\textbf{J}} := \Spec \bigg( \underset{v \in Q_0}{\bigotimes} \F_1^\pm[T_\textbf{x} / T_{J_v} \mid \textbf{x} \in E_v^{r_v}] \bigg)
\]
and its closed subscheme $U_{\textbf{J}} := \mathbb{A}_{\textbf{J}} \times_{\mathbb{P}(\Lambda)} \Gr_{\textbf{r}}(\Lambda)$. As $\{\mathbb{A}_{\textbf{J}}\}$ is an affine open covering of $\mathbb{P}(\Lambda)$, the collection $\{U_{\textbf{J}}\}$ is an affine covering of $\Gr_{\textbf{r}}(\Lambda)$. Note that 

\[
U_{\textbf{J}} = \Spec \bigg(\Big(\underset{v \in Q_0}{\bigotimes} \F_1^\pm[T_\textbf{x} / T_{J_v} \mid \textbf{x} \in E_v^{r_v}]\Big) \big/\hspace{-4pt}\big/ \mathpzc{qr}_{\textbf{J}, \Lambda} \bigg),
\]
where $\mathpzc{qr}_{\textbf{J}, \Lambda}$ is generated by the elements:

\begin{enumerate}
    \item\label{QR1} $\dfrac{T_{{\bf x}^\sigma}}{T_{J_v}} - \sign(\sigma)\dfrac{T_{{\bf x}}}{T_{J_v}}$ for all $v \in Q_0$, $\mathbf{x} \in E_v^{r_v}$ and permutation $\sigma$;
    \\[5pt]
    \item\label{QR2}  $\dfrac{T_{{\bf x}}}{T_{J_v}}$ for all $\mathbf{x} \in E_v^{r_v}$ with $\#|\mathbf{x}| < r_v$;
    \\[5pt]
    \item\label{QR3} $\displaystyle \underset{k=1}{\overset{r_{v}+1}{\sum}} (-1)^k ~ \dfrac{T_{{\bf y}_{\widehat{k}}}}{T_{J_{v}}} \otimes \dfrac{T_{(y_k,{ \bf x })}}{T_{J_{v}}}$ for all for all $v \in Q_0$, $\mathbf{y} \in E_v^{r_v+1}$ and $\mathbf{x} \in E_v^{r_v-1}$;
    \\[5pt]
    \item\label{QR4} $\displaystyle \underset{k=1}{\overset{r_{s(\alpha)}+1}{\sum}} (-1)^k ~ \dfrac{T_{{\bf y}_{\widehat{k}}}}{T_{J_{s(\alpha)}}} \otimes \dfrac{T_{(\Lambda_{\alpha}(y_k),{ \bf x })}}{T_{J_{t(\alpha)}}}$ for all $\alpha \in Q_1$, $\mathbf{y} \in E_{s(\alpha)}^{r_{s(\alpha)}+1}$ and $\mathbf{x} \in E_{t(\alpha)}^{r_{t(\alpha)}-1}$,
\end{enumerate}
where $T_{(0, \textbf{x})}:=0$ for $\textbf{x} \in E^{r_v-1}_v$ in \eqref{QR4}.

Note that if $\#|J_v| < r_v$ for some $v\in Q_0$, then $U_{\mathbf{J}}$ is empty. 
\end{rem}

Let $\pi_v: \P(\Lambda) \rightarrow \mathbb{P}^{(\#E_v^{r_v} - 1)}$ be the canonical projection for $v \in Q_0$ and $\pl_v=\pi_v\circ\pl$. Let $\cL^\univ_v:=\textup{pl}_v^*\big(\cO(1)\big)$, which is a line bundle on $\Gr_{\textbf{r}}(\Lambda)$.
 
\begin{df}
 The \emph{universal line bundle on $\Gr_{\mathbf{r}}(\Lambda)$} is $\cL^\univ_v:=\textup{pl}_v^*\big(\cO(1)\big)$. The \emph{universal} $\Lambda$\emph{-matroid bundle} of rank ${\textbf{r}}$ is the $\Lambda$-matroid bundle
\[
{\cM}^{\textnormal{univ}} \; := \; \big([\mu_v^{\textnormal{univ}}: E_v^{r_v} \rightarrow \Gamma(\Gr_{\textbf{r}}(\Lambda), \cL^\univ_v)]\big)_{v \in Q_0}
\]
over $\Gr_{\textbf{r}}(\Lambda)$, where $\mu_v^{\textnormal{univ}}({\bf x}) := \Gamma\textup{pl}_v(T_{\bf x})$.
\end{df}

\begin{thm}\label{thm: moduli space of quiver matroid bundles}
The quiver Grassmannian $\Gr_{\mathbf{r}}(\Lambda)$ is the fine moduli space of $\Lambda$-matroid bundles, i.e. the maps
\[
\begin{array}{cccc}
\Phi_X: & \textup{Hom}\big(X, \Gr_{\mathbf{r}}(\Lambda)\big) & \longrightarrow & \mathpzc{Gr}_{\mathbf{r}}(\Lambda)(X)\\
 & \beta & \longmapsto & \beta^*({\cM}^{\textnormal{univ}}).
\end{array}
\]
(where $X$ is a band scheme) define an isomorphism $\Hom(-,\Gr_{\mathbf{r}}(\Lambda))\to\mathpzc{Gr}_{\mathbf{r}}(\Lambda)$ of functors.
\end{thm}

\begin{proof}
Let $X$ be a band scheme. As a first step, we construct the inverse $\Psi_X: \mathpzc{Gr}_{\textbf{r}}(\Lambda)(X) \rightarrow  \textup{Hom}\big(X, \Gr_{\textbf{r}}(\Lambda)\big)$ to $\Phi_X$. Consider a $(\Lambda, X)$-matroid bundle 
\[
\cM = \big([\mu_v: E_v^{r_v} \rightarrow \Gamma(X,\cL_v)]\big)_{v\in Q_0}
\]
with rank vector $\textbf{r}$. For $v \in Q_0$, the set of global sections $\{\mu_v(\mathbf{x}) \mid \mathbf{x} \in E_v^{r_v}\}$ generates $\mathcal{L}_v$. Thus by \autoref{prop: characterization of morphisms to projective space by global sections of line bundles}, there exists a unique morphism of band schemes $\psi_v: X \rightarrow \mathbb{P}^{(\#E_v^{r_v} - 1)}$ and an isomorphism $\iota_v : \psi_v^*\big(\cL^\univ\big)\rightarrow\cL_v$ satisfying $\iota_v \big(\Gamma \psi_v(T_{\mathbf{x}})\big) = \mu_v(\mathbf{x})$ for all $\mathbf{x} \in E_v^{r_v}$. Consider the product of the $\psi_v$, to obtain a morphism 
\[
\psi:  X \rightarrow \P(\Lambda).
\]
As $\cM$ is a $\Lambda$-matroid bundle, $\psi|_{\psi^{-1}(\A_{\mathbf{J}})}: \psi^{-1}(\A_{\mathbf{J}}) \rightarrow \A_{\mathbf{J}}$ has image in $U_\mathbf{J}$, for all $\mathbf{J} = (J_v)_{v \in Q_0}$. Thus $\psi$ factorizes through $\pl: \Gr_{\textbf{r}}(\Lambda) \hookrightarrow \mathbb{P}(\Lambda)$, i.e. there exists a unique morphism $\gamma: X \rightarrow \Gr_{\textbf{r}}(\Lambda)$ such that the diagram
\[
\begin{tikzcd}
X \arrow[r, "\psi"] \arrow[d, dashrightarrow, swap, "\gamma"] & \P(\Lambda)
\\
\Gr_{\textbf{r}}(\Lambda) \arrow[ru, hook, "\pl", swap]
\end{tikzcd}
\]
commutes. Define $\Psi_X(\cM) := \gamma$. Next we show that $\Psi_X = \Phi_X^{-1}$.

For a morphism $\beta: X \rightarrow \Gr_{\textbf{r}}(\Lambda)$, note that $\pl_v \circ \beta: X \rightarrow \mathbb{P}^{(\#E_v^{r_v} - 1)}$ is a morphism of band schemes that satisfies 
$$
(\pl_v \circ \beta)^*\big(\cO(1)\big) = \beta^*(\cL^\univ_v) \quad \text{ and } \quad \Gamma(\pl_v \circ \beta)(T_{\mathbf{x}}) = \Gamma\beta \big(\Gamma\pl_v(T_{\mathbf{x}})\big)
$$
for all $\mathbf{x} \in E_v^{r_v}$ and all $v \in Q_0$. Therefore
\[
\beta = \Psi_X\big(\beta^*(\cM^{\textnormal{univ}})\big) = \Psi_X \circ \Phi_X(\beta).
\]

For a $(\Lambda, X)$-matroid bundle $\cM = \big([\mu_v: E_v^{r_v} \rightarrow \Gamma(X,\cL_v)]\big)_{v\in Q_0}$, consider the morphisms $\psi: X \rightarrow \P(\Lambda)$ and $\psi_v = \pi_v \circ \psi: X \rightarrow \mathbb{P}^{(\#E_v^{r_v} - 1)}$ of band schemes, and the morphism $\iota_v : \psi_v^*\big(\cO(1)\big) \rightarrow \cL_v$ of lines bundles over $X$, as above. Note that 
\[
\cM = \big(\big[\Gamma\psi_v(T_{-}): E_v^{r_v} \rightarrow \Gamma\big(X,\psi_v^*(\cO(1))\big)\big]\big)_{v\in Q_0},
\]
and let $\gamma := \Psi_X(\cM)$. As $(\pl_v \circ \gamma)^*\big(\mathcal{O}(1)\big) = \psi_v^*\big(\cO(1)\big)$ and $\Gamma(\pl_v \circ \gamma)(T_{\mathbf{x}}) = \Gamma\psi_v(T_{\mathbf{x}})$ for all $\mathbf{x} \in E_v^{r_v}$ and $v \in Q_0$, it follows that 
\[
\cM = \Phi_X(\gamma) = \Phi_X \circ \Psi_X(\cM).
\]

For a morphism $\eta: Y \rightarrow X$ of band schemes and $\beta \in \textup{Hom}\big(X, \Gr_{\textbf{r}}(\Lambda)\big)$, 
\[
\Phi_Y(\beta \circ \eta) = (\beta \circ \eta)^*({\cM}^{\textnormal{univ}}) = \eta^*\big(\beta^*({\cM}^{\textnormal{univ}})\big) = \eta^*\big(\Phi_X(\beta)\big).
\]
Therefore $\Phi: \textup{Hom}\big(-, \Gr_{\textbf{r}}(\Lambda)\big) \Rightarrow \mathpzc{Gr}_{\textbf{r}}(\Lambda)(-)$ is a natural isomorphism.
\end{proof}

\begin{cor}
\label{cor: F-rational points of the quiver Grassmannian}
If $F$ is a perfect idyll, then the $F$-rational points of $\Gr_{\mathbf{r}}(\Lambda)$ are in bijection with the set of $(Q, F)$-matroids $M$ with rank vector $\textbf{r}$, underlying $\Fun$-representation $\cat{U}_F(M) = \Lambda$ and such that $M_\alpha = \Lambda_{\alpha, F}$ for all $\alpha$.
\end{cor}

\begin{proof}
 Let $X := \Spec F$. By \autoref{thm: Grassmannian as moduli space of matroid bundles}, the $F$-rational points of $\Gr_{\mathbf{r}}(\Lambda)$ are in bijection with the set of $(\Lambda, \Spec F)$-matroid bundles of rank $\mathbf{r}$. By \autoref{ex: matroid bundles on points}, tuples $(\mathcal{M}_v)_{v \in Q_0}$ of matroid bundles on $X$, with $\rk(\mathcal{M}_v) = r_v$ for all $v \in Q_0$, are in bijection with tuples $(M_v)_{v \in Q_0}$ of $F$-matroids, with $\rk(M_v) = r_v$ for all $v \in Q_0$. Thus the result follows from the application of \autoref{ex: morphisms over points} to all arrows $\alpha \in Q_1$.
\end{proof}

\begin{df}
Given a $(\Lambda, X)$-matroid bundle $\cM$ of rank $\textbf{r}$, its \emph{characteristic morphism} $\chi_{\cM}$ is the (unique) morphism $X \to \Gr_{\textbf{r}}(\Lambda)$ that corresponds to $\cM$, i.e. $\chi_{\cM} := \Psi_X (\cM)$ (see \autoref{thm: moduli space of quiver matroid bundles}).
\end{df}

\subsection{Duality}
Let $Q$, $\Lambda$ and $\textbf{r}$ as above. Assume that $E_v=\{1, \dotsc, n_v\}$ for all $v \in Q_0$. Define $\textbf{r}^* := (r_v^*)_{v \in Q_0} \in \Z^{Q_0}$ where  $r^\ast_v= n_v - r_v$.

The isomorphisms $\delta_{r_v, E_v}: \Gr(r_v, E_v) \rightarrow \Gr(r_v^*, E_v)$ (see \autoref{subsection: the dual of a matroid bundle}) induce an isomorphism
\[
\widehat{\delta}_{\textbf{r}, \Lambda}: \underset{v \in Q_0}{\prod} \Gr(r_v, E_v) \longrightarrow \underset{v \in Q_0}{\prod} \Gr(r_v^*, E_v).
\]

\begin{df}
The \emph{dual} of $\Lambda$ is $\Lambda^* := \big((\Lambda_v), (\Lambda_\alpha^t)\big)$, which is a representation of $Q^*$ over $\cat{Vect}_{\F_1}$.
\end{df}

\begin{prop}
The morphism $\widehat{\delta}_{\mathbf{r}, \Lambda}$ restricts to an isomorphism
\[
\delta_{\mathbf{r}, \Lambda}: \Gr_{\mathbf{r}}(\Lambda) \longrightarrow \Gr_{\mathbf{r}^*}(\Lambda^*).
\]
\end{prop}
\begin{proof}
Note that $(\textbf{r}^*)^* = \textbf{r}$, $(\Lambda^*)^* = \Lambda$ and $(\widehat{\delta}_{\textbf{r}, \Lambda})^{-1} = \widehat{\delta}_{\textbf{r}^*, \Lambda^*}$. Thus it is enough to show that $\widehat{\delta}_{\textbf{r}, \Lambda}\big(\Gr_{\textbf{r}}(\Lambda)\big) \subseteq \Gr_{\textbf{r}^*}(\Lambda^*)$.

Let $\textbf{I} = (I_v)_{v \in Q_0}$ be a tuple such that $I_v \in E_v^{r_v}$ and $\#|I_v| = r_v$ for all $v$. For each $v$, fix $I_v' \subseteq E_v^{r_v^*}$ such that $|I_v| \sqcup |I_v'| = E_v$ and define $\textbf{I}' := (I_v')_{v \in Q_0}$.

Note that $\widehat{\delta}_{\textbf{r}, \Lambda} (U_\textbf{I}) \subseteq \A_{\textbf{I}'} \subseteq \mathbb{P}(\Lambda^*)$, and the restriction $\widehat{\delta}_{\textbf{r}, \Lambda}: U_\textbf{I} \rightarrow \A_{\textbf{I}'}$ is induced by the morphism of $\F_1^\pm$-algebras
\[
\begin{array}{cccc}
\Theta_\mathbf{I}: & \underset{v \in Q_0}{\bigotimes} \F_1^\pm[T_\textbf{y} / T_{I_v'} \mid \textbf{y} \in E_v^{r_v^*}] & \longrightarrow& \Big(\underset{v \in Q_0}{\bigotimes} \F_1^\pm[T_\textbf{x} / T_{I_v} \mid \textbf{x} \in E_v^{r_v}]\Big) \big/\hspace{-4pt}\big/ \mathpzc{qr}_{\textbf{I}, \Lambda}\\
&&&\\
 & \underset{v \in Q_0}{\otimes} T_{\textbf{y}_v} / T_{I_v'} &\longmapsto & \begin{cases}
                \;\;\; 0  \quad \text{ if $\#|\textbf{y}_v| < r_v^*$ for some $v$}\\
                \\
                \underset{v \in Q_0}{\otimes} \bigg( \dfrac{\sign({\bf y}_v,{\bf y}_v')\cdot T_{{\bf y}'_v}}{\sign({I_v'},I_v)\cdot T_{I_v}}\bigg) \;\;\text{otherwise},  
                \end{cases}
\end{array}
\]
where, for $\textbf{y}_v \in E_v^{r_v^*}$ with $\#|\textbf{y}_v| = r_v^*$, the tuple ${\bf y}_v' \in E_v^{r_v}$ satisfies $E_v = |{\bf y}_v|\sqcup |{\bf y}_v'|$. Next we show that $\Theta_\mathbf{I}$ factorizes through 
\[
    \begin{tikzcd}[column sep = 20pt]
        \underset{v \in Q_0}{\bigotimes} \F_1^\pm[T_\textbf{y} / T_{I_v'} \mid \textbf{y} \in E_v^{r_v^*}]\big/\hspace{-4pt}\big/ \mathpzc{qr}_{\textbf{I}', \Lambda} \arrow[two heads]{r} & \Big( \underset{v \in Q_0}{\bigotimes} \F_1^\pm[T_\textbf{x} / T_{I_v} \mid \textbf{y} \in E_v^{r_v}]\Big)\big/\hspace{-4pt}\big/ \mathpzc{qr}_{\textbf{I}, \Lambda}.
    \end{tikzcd}
\]

Let $\textbf{x} \in E_v^{r_v^* + 1}$ and $\mathbf{z} \in E_w^{r_w^* - 1}$. If $\#|\mathbf{x}| < r_v^* + 1$ or $\#|\mathbf{z}| < r_w^* - 1$, one has
\[
\begin{aligned}
    &\Theta_\mathbf{I}\bigg(\underset{k=1}{\overset{r_v^* + 1}{\sum}} (-1)^k ~ \dfrac{T_{{\textbf{x}}_{\widehat{k}}}}{T_{I'_{v}}} \otimes \dfrac{T_{(\Lambda_\alpha^t(x_k),{ \textbf{z} })}}{T_{I'_{w}}}\bigg) \; \in \; \mathpzc{qr}_{\mathbf{I}, \Lambda}.
\end{aligned}
\]

Assume $\#|\mathbf{x}| = r_v^* + 1$ and $\#|\mathbf{z}| = r_w^* - 1$. Let $\mathbf{x}' \in E_v^{r_v - 1}$ such that $|\text{$\mathbf{x}$}| \sqcup |\mathbf{x}'| = E_v$, and $\mathbf{z}' \in E_w^{r_w + 1}$ such that $|\mathbf{z}| \sqcup |\mathbf{z}'| = E_w$.

For the following computation, we define $\mathbf{z'}_{\widehat{c}} := \mathbf{z'}_{\widehat{\ell}}$ where $\ell$ is the index such that $z'_\ell=c$. Note that $\bx_{\widehat{k}}$ has the usual meaning.

Note that
\[
\begin{aligned}
    & \Theta_\mathbf{I}\bigg(\underset{k=1}{\overset{r_v^* + 1}{\sum}} (-1)^k ~ \dfrac{T_{\mathbf{x}_{\widehat{k}}}}{T_{I'_{v}}} \otimes \dfrac{T_{(\Lambda^t_{\alpha}(x_k),{ \textbf{z} })}}{T_{I'_{w}}}\bigg)
\\
    =\; & \underset{k=1}{\overset{r_v^* + 1}{\sum}} (-1)^k ~ \bigg(\dfrac{\sign(\text{$\mathbf{x}$}_{\widehat{k}},x_k, \text{$\mathbf{x}$}')\cdot T_{(x_k, \text{$\mathbf{x}$}')}}{\sign({I_v'},I_v)\cdot T_{I_v}}\bigg) \otimes \left(\dfrac{\sign\big(\Lambda^t_{\alpha}(x_k),{ \textbf{z} },{\textbf{z}}'_{\widehat{\Lambda^t_{\alpha}(x_k)}}\big)\cdot T_{{\textbf{z}}'_{\widehat{\Lambda^t_{\alpha}(x_k)}}}}{\sign({I_w'},I_w)\cdot T_{I_w}}\right)
\end{aligned}
\]
As $(-1)^k \sign(\mathbf{x}_{\widehat{k}},x_k, \mathbf{x}')$ is the same for any $k$, up to a global sign, the expression above is equal to
\[
\begin{aligned}
    &\underset{k=1}{\overset{r_v^* + 1}{\sum}} \bigg( \dfrac{T_{(x_k, \text{$\mathbf{x}$}')}}{T_{I_v}}\bigg) \otimes \left( \dfrac{\sign\big({ \textbf{z} }, \Lambda^t_{\alpha}(x_k),{\textbf{z}}'_{\widehat{\Lambda^t_{\alpha}(x_k)}}\big)\cdot T_{{\textbf{z}}'_{\widehat{\Lambda^t_{\alpha}(x_k)}}}}{T_{I_w}}\right)
\\
    =\; & \underset{b \notin |\textbf{x}'|}{\sum} \bigg( \dfrac{T_{(b, \text{$\mathbf{x}$}')}}{T_{I_v}}\bigg) \otimes \left( \dfrac{\sign\big({ \textbf{z} }, \Lambda^t_{\alpha}(b),{\textbf{z}}'_{\widehat{\Lambda^t_{\alpha}(b)}}\big)\cdot T_{{\textbf{z}}'_{\widehat{\Lambda^t_{\alpha}(b)}}}}{T_{I_w}}\right)
\\
    =\; & \underset{a \notin |\mathbf{z}'|}{\sum} \bigg( \dfrac{\sign\big({ \textbf{z} }, a,{\textbf{z}}'_{\widehat{a}}\big)\cdot T_{{\textbf{z}}'_{\widehat{a}}}}{T_{I_w}}\bigg) \otimes \left( \dfrac{T_{(\Lambda_{\alpha}(a), \text{$\mathbf{x}$}')}}{T_{I_v}}\right)
\\
    =\; & \sign(\mathbf{z}, \mathbf{z}') ~ \underset{k=1}{\overset{r_{w}+1}{\sum}} (-1)^k ~ \dfrac{T_{{\textbf{z}}'_{\widehat{k}}}}{T_{I_w}} \otimes \dfrac{T_{(\Lambda_\alpha(z_k'),{\bf x})}}{T_{I_{v}}}, 
\end{aligned}
\]
which is in $\mathpzc{qr}_{\mathbf{I}, \Lambda}$, because $\underset{k=1}{\overset{r_{w}+1}{\sum}} (-1)^k ~ \dfrac{T_{{\textbf{z}}'_{\widehat{k}}}}{T_{I_w}} \otimes  \dfrac{T_{(\Lambda_\alpha(z_k'),{\bf x})}}{T_{I_{v}}} \in \mathpzc{qr}_{\mathbf{I}, \Lambda}$. Thus $\widehat{\delta}_{\textbf{r}, \Lambda} (U_\mathbf{I}) \subseteq (U_{\mathbf{I}'})$. Therefore $\widehat{\delta}_{\textbf{r}, \Lambda}\big(\Gr_{\textbf{r}}(\Lambda)\big) \subseteq \Gr_{\textbf{r}^*}(\Lambda^*)$.
\end{proof}

\begin{thm}
Let $\cM = (\cM_v)_{v \in Q_0}$ be a $(\Lambda, X)$-matroid bundle with rank vector $\mathbf{r}$. Then the characteristic morphism of $\cM^*$ is equal to the composition
\[
\begin{tikzcd}
\chi_{\cM^*}: X \arrow[r, "\chi_{\cM}"] & \Gr_{\mathbf{r}}(\Lambda) \arrow[r, "\delta_{\mathbf{r}, \Lambda}"] & \Gr_{\mathbf{r}^*}(\Lambda^*).
\end{tikzcd}
\]
\end{thm}

\begin{proof}
As the diagram
\[
\begin{tikzcd}
\Gr_{r}(\Lambda) \arrow[r, hook] \arrow[d, swap, "\delta_{\textbf{r}, \Lambda}"] & \underset{v \in Q_0}{\prod} \Gr(r_v, E_v) \arrow[d, "\widehat{\delta}_{\textbf{r}, \Lambda}"] \arrow[r, "\pi_w"] & \Gr(r_w, E_w) \arrow[d, "\delta_{r_w,  E_w}"]
\\
\Gr_{\textbf{r}^*}(\Lambda^*) \arrow[r, hook] &\underset{v \in Q_0}{\prod} \Gr(r_v^*, E_v) \arrow[r, "\pi_w"] & \Gr(r_w^*, E_w)
\end{tikzcd}
\]
commutes and $\chi_{\cM_w^*} = \delta_{r_w,  E_w} \circ \chi_{\cM_w}$ for all $w \in Q_0$, the result follows.
\end{proof}


\section{Euler characteristics of quiver Grassmannians}
\label{section: Euler characteristics}

In this section, we extend previous comparison results between the complex Euler characteristic and the number of $\Fun$-rational points from flag varieties to a larger class of quiver Grassmannians. Throughout this section we assume that every quiver is finite, i.e. it has a finite set of vertices and arrows.

More precisely, let $X$ be a band scheme and $X(\C)=X^+_\C(\C)$  be the complex analytic space of $\C$-rational points. The Tits space of $X$ is a certain subspace of $X(\K)$. In general the Euler characteristic of $X(\C)$ differs from the cardinality of $X^\Tits$, but in nice situations they agree, e.g.\ for flag varieties over $\Funpm$; cf.\ \cite{Lorscheid-Thas23} and \cite[Section 3.4]{Baker-Jin-Lorscheid24}. 

In this section, we extend this result to a larger class of quiver Grassmannians $\Gr_{\textbf{r}}(\Lambda)$. Our method of proof uses a torus action on $\Gr_{\textbf{r}}(\Lambda)$ whose fixed points in $\Gr_{\textbf{r}}(\Lambda)(\C)$ are coordinate vectors. By a result of Bia\l ynicki-Birula (\cite{Bialynicki-Birula73}), the Euler characteristic of $\Gr_{\textbf{r}}(\Lambda)(\C)$ agrees with the number of these coordinate vectors. An analysis of the combinatorics of $\Lambda$ (in terms of its coefficient quiver) eventually leads to an identification of the coordinate vectors with the points of $X^\Tits$. This analysis relies on some basic matroid theoretic constructions.

\subsection{The fine topology}
\label{subsection: fine topology}

Let $F$ be an idyll with a topology for which the multiplication of $F$ and the inversion of $F^\times$ are continuous and such that $0$ is a closed point. The examples of main interest are the complex numbers $\C$ with the usual topology and the Krasner hyperfield $\K$ with the topology for which $0$ is closed, but not $1$.

Then we can topologize the sets of $F$-rational points $X(F)$ of a band scheme $X$ in the following way (cf.\ \cite[Thm.\ 3.5]{Baker-Jin-Lorscheid24}).

\begin{thm}\label{thm: fine topology}
 There is a unique enrichment of $h_F=\Hom(\Spec F,-)$ to a functor $h_F:\BSch\to\Top$ such that the following properties hold (where we write $X(F)$ for $h_F(X)$):
 \begin{enumerate}
  \item The canonical bijection $\A^n(F)\to F^n$ is a homeomorphism.
  \item If $\iota:Y\to X$ is an open (closed) immersion, then $\iota_F:Y(F)\to X(F)$ is an open (closed) topological embedding.
  \item If $X=\bigcup U_i$ is an open covering, then $X(F)=\bigcup U_i(F)$ is an open topological covering.
  \item If $X=\Spec B$, then the canonical bijection $X(F)\to\Hom(B,F)$ is a homeomorphism where $\Hom(B,F)$ is endowed with the compact-open topology with respect to the discrete topology for $B$.
 \end{enumerate}
 This topology for $X(F)$ is called the \emph{fine topology}.
\end{thm}

In particular, the topology of $X(\C)=X^+_\C(\C)$ coincides with the usual complex topology for every band scheme $X$. The space of $\K$-rational points 
\[
 \P^n(\K) \ = \ \big\{ [a_0:\dotsc:a_n] \, \big| \, a_0,\dotsc,a_n\in \K \big\}
\]
is finite and its topology is determined by the rule that $[a_0:\dotsc:a_n]$ is in the closure of $[b_0:\dotsc:b_n]$ if and only if $b_i=0$ implies $a_i=0$ for all $i$. According to \autoref{thm: fine topology}, the product $(\prod\P^{n_i})(\K)=\prod\big(\P^{n_i}(\K)\big)$ of projective spaces carries the product topology and the $\K$-points $X(\K)$ of a closed subscheme $X$ carry the subspace topology. This applies to all quiver Grassmannians.

\subsection{The Tits space}
\label{subsection: the Tits space}

The \emph{Tits space $X^\Tits$} of a band scheme $X$ is defined as the subset of closed points of $X(\K)$.

For example, the Tits space of $\P^n$ is the subspace
\[
 (\P^n)^\Tits \ = \ \big\{ [1:0:\dotsc:0], \dotsc, [0:\dotsc:0:1] \big\}
\]
of $\P^n(\K)=\{ [a_0:\dotsc:a_n] \mid a_0,\dotsc,a_n\in \K \}$. More generally,
\[
 \Gr(r,E)^\Tits \ = \ \big\{ [\mu_{\textbf{x}}]_{{\textbf{x}}\in E^r} \, \big| \, \text{there is an $r$-subset $I\subseteq E$ such that }\mu_{\textbf{x}}=0\text{ if }|{\textbf{x}}|\neq I \big\}
\]
as a subspace of $\P^{n^r-1}(\K)$. Since for given $I\subseteq E$, there is a unique point $[\mu_{\textbf{x}}]\in\Gr(r,E)(\K)$ with $\mu_{\textbf{x}}=0$ for $|{\textbf{x}}|\neq I$, the Tits space of $\Gr(r,E)$ has $\binom nr$ points, which is equal to the Euler characteristic of the complex Grassmann variety $\Gr(r,E)(\C)$.

\subsection{The initial matroid}
\label{subsection: the limit matroid}

The initial matroid of a valuated matroid was introduced in Dress-Wenzel's seminal paper on valuated matroids (see \cite[Prop.\ 2.9]{Dress-Wenzel92b}). The term ``initial matroid" does not appear in this paper, however, but makes its first appearance in \cite[Def. 4.2.7]{MR3287221}, to the best of our knowledge; also see \cite{Brandt-Speyer22}. 

In this section, we consider only particular types of initial matroids, cf.\ \autoref{rem: valuated matroids from gradings} for details on the relation to valuated matroids.

\begin{prop}
\label{grading on matroids}
 Let $M$ be a $\K$-matroid on $E$ and $\partial: E \to \Z$ a map. For $A \subseteq E$, define $\partial(A) := \sum_{a \in A} \partial(a)$ and $m= \min\{\partial(B) \mid B \in \mathcal{B}(M)\}$. Then 
 \[
  \mathcal{B} \ = \ \{B \in \mathcal{B}(M) \mid \partial(B) = m\}
 \]
 is the set of bases of a $\K$-matroid.
\end{prop}

\begin{proof}
Let $B_1, B_2 \in \mathcal{B}$ and $x \in B_1 - B_2$. As $M$ is a $\K$-matroid, there exists $y \in B_2 - B_1$ such that $A_1 := B_1 \cup \{y\} - \{x\}$ and $A_2 := B_2 \cup \{x\} - \{y\}$ are bases of $M$. Note that $\partial(A_1) + \partial(A_2) = \partial(B_1) + \partial(B_2) = 2m$ and $m \leq \partial(A_1), \partial(A_2)$, thus $\partial(A_1) = m = \partial(A_2)$, i.e. $A_1, A_2 \in \mathcal{B}$.
\end{proof}

\begin{df}
\label{def - initial matroid}
Let $M$ be a $\K$-matroid $M$ on $E$ and $\partial: E \rightarrow \Z$ a map. The \emph{initial matroid with respect to $\partial$} is the $\K$-matroid $\min_\partial(M)$ whose set of basis is 
\[
 \mathcal{B}\big(\min\nolimits_{\partial}(M)\big) \ = \ \big\{B \in \mathcal{B}(M) \, \big| \, \partial(B) \leq \partial(X) \textup{ for all } X \in \mathcal{B}(M) \big\}.
\]
Let $(\partial_0, \dotsc, \partial_n)$ be a finite sequence of maps $\partial_i:E\to\Z$. The \emph{initial matroid with respect to $(\partial_0, \dotsc, \partial_n)$} is the $\K$-matroid 
\[\textstyle
 \min_{\partial_n, \dotsc, \partial_0}(M)=\min_{\partial_n} \circ \dots \circ \min_{\partial_0}(M).
\]
\end{df}

\begin{rem}\label{rem: valuated matroids from gradings}
 The relation to initial matroids in the sense of \cite{Dress-Wenzel92b} and \cite{MR3287221} is as follows. Let $M$ be a matroid with Grassmann-Pl\"ucker function $\Delta:E^r\to\K$ and $v=\iota\circ\Delta:R^r\to\T$ the trivial valuation where $\iota:\K\to\T$ is the natural inclusion. The choice of a function $\partial:E\to\Z$ and a $q>1$ defines a rescaling $v_{q,\partial}:E^r\to\T$ of $v$, defined by $v_{q,\partial}(x_1,\dotsc,x_r)=q^{\partial(x_1)+\dotsb+\partial(x_r)}\cdot v(x_1,\dotsc,x_r)$. The initial matroid $\min_\partial(M)$ (in our sense) is the same as the initial matroid of the valuated matroid $[v_{q,\partial}]$ in the sense of \cite{Dress-Wenzel92b} and \cite{MR3287221} (which does not depend on the choice of $q>1$).
\end{rem}

 Since $\mathcal{B}\big(\min_\partial(M)\big) \subseteq \mathcal{B}(M)$, we have $\rk(M) = \rk\big(\min_\partial(M)\big)$. The identity map on $\uE$ is a {weak map} (in the sense of \cite{Kung-Nguyen86}) between the underlying matroids of $M$ and $\min_\partial(M)$.

 If we interpret $M$ and $\min_\partial(M)$ as points of $\Gr(r,E)(\K)$ via \autoref{thm: Grassmannian as moduli space of matroid bundles}, then this means that $\min_\partial(M)$ is contained in the closure of $M$ with respect to the natural topology from \autoref{subsection: fine topology}.

\begin{lemma}
\label{lemma: distinguishes elements}
Let $M$ be a $\K$-matroid on $E$ and $\partial_0, \dotsc, \partial_n: E \rightarrow \Z$ maps such that for any $a,b\in E$ there is an $i$ with $\partial_i(a) \neq \partial_i(b)$. Then $\min_{\partial_n, \dotsc, \partial_0}(M) = M$ if, and only if, $M$ has only one basis.
\end{lemma}

\begin{proof}
Assume $\min_{\partial_n, \dotsc, \partial_0}(M) = M$ and let $B_1, B_2 \in \mathcal{B}(M)$. Note that $\partial_i$ is constant in $\mathcal{B}(M)$ for every $i$. If $B_1 \neq B_2$, fix $x \in B_1 - B_2$. Then there exists $y \in B_2 - B_1$ such that $B_1 \cup\{y\} - \{x\}$ is a basis of $M$.

Let $\ell$ such that $\partial_\ell(x) \neq \partial_\ell(y)$. Then 
\[
\partial_\ell(B_1) \;=\; \partial_\ell(B_1 - \{x\}) + \partial_\ell(x) \;\neq\; \partial_\ell(B_1 - \{x\}) + \partial_\ell(y) \;=\; \partial_\ell(B_1 \cup\{y\} - \{x\}).
\]
Since $\partial_\ell$ is constant on $\mathcal{B}(M)$, one has a contradiction.
\end{proof}

\begin{lemma}
\label{lemma: grading on morphisms of matroids}
Let $N$ and $M$ be two $\K$-matroids with respective ground sets $S$ and $T$. Let $\Phi: N \rightarrow M$ be a morphism and $\partial: S \sqcup T \rightarrow \Z$ a map such that 
\[
\partial(s_1) - \partial\big(\underline{\Phi}(s_1)\big) \ = \ \partial(s_2) - \partial\big(\underline{\Phi}(s_2)\big)
\]
for all $s_1, s_2 \in \underline{\Phi}^{-1}(T)$. Then $\Phi: \min_{\partial|_S}(N) \rightarrow \min_{\partial|_T}(M)$ is a morphism.
\end{lemma}

\begin{proof}
 Let $\nu$ and $\mu$ be Grassmann-Pl\"ucker functions for $N$ and $M$, respectively. Write $\min_{\partial|_S}(N) = [\nu']$, $\min_{\partial|_T}(M) = [\mu']$, $w := \rk(N)$ and $r:=\rk(M)$. Let $\mathbf{y} \in S^{w+1}$ and $\mathbf{x} \in T^{r-1}$. Suppose that there exists $\ell \in \{1, \dotsc, w+1\}$ such that 
\[
\nu'(\mathbf{y}_{\widehat{\ell}}) \ \Phi_{y_\ell} \ \mu'\big(\underline{\Phi}(y_\ell), \mathbf{x}\big) = 1.
\]

In particular, $|\mathbf{y}_{\widehat{\ell}}|$ is a basis of $N$, $\{\underline{\Phi}(y_\ell)\} \cup |\mathbf{x}|$ is a basis of $M$ and $\Phi_{y_\ell} = 1$, i.e.
\[
\nu(\mathbf{y}_{\widehat{\ell}}) \ \Phi_{y_\ell} \ \mu\big(\underline{\Phi}(y_\ell), \mathbf{x}\big) = 1.
\]
As $\Phi: N \rightarrow M$ is a morphism and $N_\K = \N - \{1\}$, there exists $m \neq \ell$ such that 
\[
\nu(\mathbf{y}_{\widehat{m}}) \ \Phi_{y_m} \ \mu\big(\underline{\Phi}(y_m), \mathbf{x}\big) = 1.
\]

Next, we prove that $|\mathbf{y}_{\widehat{m}}|$ is basis of $\min_{\partial|_S}(N)$ and $\{\underline{\Phi}(y_\ell)\} \cup |\mathbf{x}|$ is a basis of $\min_{\partial|_T}(M)$. Let $m_N:= \min\{\partial(B) \mid B \in \mathcal{B}(N)\}$ and $m_M:= \min\{\partial(B) \mid B \in \mathcal{B}(M)\}$. Note that 
\[
m_N + m_M \;= \;\partial(|\mathbf{y}_{\widehat{\ell}}|) + \partial(\{\underline{\Phi}(y_\ell)\} \cup |\mathbf{x}|) \;= \;\partial(|\mathbf{y}| \cup |\mathbf{x}|) \;=\; \partial(|\mathbf{y}_{\widehat{m}}|) + \partial(\{\underline{\Phi}(y_m)\} \cup |\mathbf{x}|).
\]
As $m_N \leq \partial(|\mathbf{y}_{\widehat{m}}|)$ and $m_M \leq \partial(\{\underline{\Phi}(y_m)\} \cup |\mathbf{x}|)$, these must be equalities, thus $|\mathbf{y}_{\widehat{m}}|$ is basis of $\min_{\partial|_S}(N)$ and $\{\underline{\Phi}(y_\ell)\} \cup |\mathbf{x}|$ is a basis of $\min_{\partial|_T}(M)$, i.e.
\[
\nu'(\mathbf{y}_{\widehat{m}}) \ \Phi_{y_m} \ \mu'\big(\underline{\Phi}(y_m), \mathbf{x}\big) = 1.
\]
By \autoref{thm: cryptomorphisms}, $\Phi$ is a morphism from $\min_{\partial|_S}(N)$ to $\min_{\partial|_T}(M)$.
\end{proof}

\subsection{Nice gradings}
\label{subsection: nice gradings}

Nice gradings of quiver representations are a tool introduced by Cerulli Irelli (\cite{CerulliIrelli11}) to compute the Euler characteristic of quiver Grassmannians in terms of subquivers of the coefficient quivers. They were further refined by Haupt (\cite{Haupt12}) and by Jun-Sistko (\cite{Jun-SistkoEuler}). The latter paper links coefficient quivers that are windings (as they appear in \cite{Haupt12}) to $\Fun$-representations; cf.\ \autoref{rem: windings}.

For the remainder of this section, we fix a finite quiver $Q$ and an $\Fun$-representation $\Lambda$ of $Q$. We write $E_v$ for the (unpointed) set $\Lambda_v-\{0\}$ and $n_v$ for its cardinality. For a dimension vector ${\textbf{r}}\in\Z^{Q_0}$ of $Q$, we define ${\textbf{r}}^*=(r_v^*)_{v \in Q_0}$ with $r^\ast=n_v-r_v$.

\begin{df}[{\cite[Def. 4.1]{Jun-SistkoEuler}}]\label{def: gradings}
 A \emph{grading of $\Lambda$} is a map $\partial: {\bigsqcup}_{v \in Q_0} E_v \rightarrow \Z$. A grading $\partial$ of $\Lambda$ is \emph{nice} if 
 \[
  \partial(a) - \partial\big(\Lambda_\alpha(a)\big) = \partial(b) - \partial\big(\Lambda_\alpha(b)\big)
 \]
 for all $\alpha\in Q_1$ and $a, b \in \Lambda_{s(\alpha)}$ with $\Lambda_\alpha(a)\neq0$ and $\Lambda_\alpha(b) \neq 0$.
\end{df}

A nice grading $\partial$ of $\Lambda$ extends to $Q_1$ via the rule
\[
 \partial(\alpha) := \begin{cases}
                      \partial(a) - \partial\big(\Lambda_\alpha(a)\big) &\text{for } a \in \Lambda_{s(\alpha)} \text{ with } \Lambda_\alpha(a) \neq 0,\\
                      {} 0 &\text{if there is no }a \in \Lambda_{s(\alpha)} \text{ with } \Lambda_\alpha(a) \neq 0.
                     \end{cases}
\]

Note that the value $\partial(\alpha)$ does not depend on the choice of $a\in\Lambda_{s(\alpha)}$ since $\partial$ is nice.

\begin{df}
 Let $M$ be a $(\Lambda,\K)$-matroid and $\partial: {\bigsqcup}_{v \in Q_0} E_v \rightarrow \Z$ a grading. The \emph{initial matroid of $M$ with respect to $\partial$} is the tuple $\min_\partial(M) := \big(\min_{\partial|_{E_v}} (M_v)\big)_{v \in Q_0}$ of $\K$-matroids.
\end{df}

\begin{prop}
\label{prop - del(M) is a quiver matroid}
 Let $M$ be a $(\Lambda,\K)$-matroid. If $\partial: {\bigsqcup}_{v \in Q_0} E_v \rightarrow \Z$ a nice grading, then $\min_\partial(M)$ is a $(\Lambda,\K)$-matroid. 
\end{prop}

\begin{proof}
This follows at once from \autoref{lemma: grading on morphisms of matroids}.
\end{proof}

\begin{df}
 A \emph{coordinate matroid} is a $\K$-matroid $M$ with a unique basis. For $A \subseteq E$, we denote by $\textup{Co}(A,E)$ the coordinate matroid on $E$ that has $A$ as its unique basis.
\end{df}

\begin{lemma}
\label{coordinate morphisms}
Let $N = \textup{Co}(A,S)$ and $M = \textup{Co}(B,T)$ and $f: \underline{S} \rightarrow \underline{T}$ an $\F_1$-linear map. Then $f$ is a strong map from $N$ to $M$ if, and only if, $f(\underline{A^c}) \subseteq \underline{B^c}$.
\end{lemma}

\begin{proof}
 Since the circuits of $N$ are the elements of $S - A$ and the cocircuits of $M$ are the elements of $B$, the result follows from \autoref{duality - classical strong maps}.
\end{proof}

\begin{df}
 A \emph{subrepresentation of $\Lambda$} is a tuple $\Omega=(\Omega_v)_{v\in Q_0}$ of pointed subsets $\Omega_v\subseteq\Lambda_v$ such that $\Lambda_\alpha({\Omega_{s(\alpha)}}) \subseteq \Omega_{t(\alpha)}$ for all $\alpha\in Q_1$ (cf. \autoref{subsection: quiver representations}). The \emph{dimension vector of $(\Omega_v)$} is the tuple 
 \[
  \mathbf{\dim}(\Omega) \ = \ (\# \Omega_v-1)_{v \in Q_0} \quad \in \quad \Z^{Q_0}.
 \]
\end{df}

\begin{prop}
\label{prop: coordinate matroids and F1-subrepresentations}
Let $\mathbf{r}\in\Z^{Q_0}$ be a dimension vector for $Q$ and $\mathbf{r}^* := (\#E_v - r_v)_{v \in Q_0}$. Then the map
\[
 \begin{array}{ccc}
  \bigg\{ \begin{array}{c}\text{subrepresentations of }\Lambda\\ \text{of dimension } \mathbf{r}^* \end{array}\bigg\} & \longrightarrow & \bigg\{ \begin{array}{c}\text{coordinate $(\Lambda,\K)$-matroids}\\ \text{of rank } \mathbf{r} \end{array}\bigg\} \\[10pt]
  ({\Omega_v})_{v \in Q_0} & \longmapsto & \big(\textup{Co}(\Omega_v^{c},E_v)\big)_{v \in Q_0}
 \end{array}
\]
is a bijection, where $\Omega_v^c:=\Lambda_v-\Omega_v$.
\end{prop}

\begin{proof}
 This follows from \autoref{thm: K-matroids and F1-linear spaces} and \autoref{coordinate morphisms}.
\end{proof}

\begin{rem}\label{rem: coordinate matroids are closed points of the quiver Grassmannian}
 By \autoref{thm: moduli space of quiver matroid bundles}, a $(\Lambda,\K)$-matroid $M$ of rank $\textbf{r}$ corresponds to a $\K$-rational point of the quiver Grassmannian $\Gr_{\textbf{r}}(\Lambda)$. If $M$ is a coordinate $(\Lambda,\K)$-matroid, then each matroid $M_v$ has a unique basis and is thus a closed point of $\Gr(r_v,E_v)(\K)$. This defines an injection
 \[
   \bigg\{ \begin{array}{c}\text{coordinate $(\Lambda,\K)$-matroids}\\ \text{of rank } {\textbf{r}} \end{array}\bigg\} \ \longrightarrow \ \Gr_{\textbf{r}}(\Lambda)^\Tits
 \]
 In the remainder of this section, we exhibit conditions that guarantee that this map is a bijection.
\end{rem}

\begin{df}[{\cite[Def.\ 4.1 and 4.6]{Jun-SistkoEuler}}]
\label{def: nice seqence of grandings}
 Let $\partial_1, \dotsc, \partial_n$ be gradings of $\Lambda$. A \emph{nice $(\partial_1, \dotsc, \partial_n)$-grading} of $\Lambda$ is a grading $\partial$ such that for all $\alpha \in Q_1$ and $a, b \in E_{s(\alpha)}$ with $\Lambda_\alpha(a), \Lambda_\alpha(b) \neq 0$ and with
\[
\partial_i(a) = \partial_i(b) \quad \text{and} \quad \partial_i(\Lambda_\alpha(a)\big) = \partial_i(\Lambda_\alpha(b)\big) \quad \text{for all} \quad i = 1, \dotsc, n,
\]
it satisfies
\[
\partial\big(\Lambda_\alpha(a)\big) - \partial\big(\Lambda_\alpha(b)\big) \ = \ \partial(a) - \partial(b).
\]
A \emph{nice sequence} for $\Lambda$ is a sequence $\underline{\partial} = \{\partial_i\}_{i = 0}^{\infty}$ of gradings of $\Lambda$ such that $\partial_0$ is nice and $\partial_i$ is a nice $(\partial_0, \dotsc, \partial_{i-1})$-grading for all $i > 0$. We say that $\underline{\partial}$ \emph{distinguishes elements} if for any $a, b \in {\bigsqcup}_{v \in Q_0} E_v$ with $a \neq b$, we have $\partial_i(a) \neq \partial_i(b)$ for some $i$.
\end{df}

\begin{thm}\label{thm: rational points and Euler characteristics}
 If there exists a nice sequence of gradings for $\Lambda$ that distinguishes elements, then 
 \[
  \chi\big(\Gr_{\mathbf{r}^*}(\Lambda_\C)\big) \ = \ \#\, \Gr_{\mathbf{r}}(\Lambda)^\Tits
 \]
 for every ${\mathbf{r}} \in \Z^{Q_0}$.
\end{thm}

\begin{proof}
 Since nice sequences of gradings of $\Lambda$ in the sense of \autoref{def: nice seqence of grandings} are nice sequences of gradings for the winding associated to the coefficient quiver of $\Lambda$ in the sense of \cite[Def. 4.1]{Jun-SistkoEuler}, $\chi\big(\Gr_{\textbf{r}^*}(\Lambda_\C)\big)$ is equal to the number of subrepresentations of $\Lambda$ of dimension ${\textbf{r}^*}$ by \cite[Prop. 5.2]{Jun-SistkoEuler}, which is the same as the number of coordinate $(\Lambda,\K)$-matroids of rank ${\textbf{r}}$ by \autoref{prop: coordinate matroids and F1-subrepresentations}.

Let $\underline{\partial} = \{\partial_i\}_{i = 0}^{\infty}$ be a nice sequence of gradings for $\Lambda$ that distinguishes elements. Since $Q_0$ is finite, there exists an $n>0$ such that for all distinct $a, b \in  {\bigsqcup} \; E_v$, we have $\partial_i(a) \neq \partial_i(b)$ for some $0 \leq i \leq n$. Consider a $(\Lambda,\K)$-matroid $M \in \Gr_{\textbf{r}}(\Lambda)(\K)$. By \autoref{lemma: distinguishes elements} and \autoref{prop - del(M) is a quiver matroid}, the initial $(\Lambda,\K)$-matroid $\min_{\partial_n,\dotsc,\partial_0}(M)$ is a coordinate matroid and thus lies in the closure of $M$. This shows that the injection
\[
  \bigg\{ \begin{array}{c}\text{coordinate $(\Lambda,\K)$-matroids}\\ \text{of rank } {\textbf{r}}\end{array}\bigg\} \ \longrightarrow \ \Gr_{\textbf{r}}(\Lambda)^\Tits
\]
from \autoref{rem: coordinate matroids are closed points of the quiver Grassmannian} is a bijection under the assumption of the theorem. 

Combining these identifications shows that the Euler characteristic of $\Gr_{\textbf{r}^*}(\Lambda_\C)$ is equal to the cardinality of $\Gr_{\textbf{r}}(\Lambda)^\Tits$.
\end{proof}

\begin{rem}\label{rem: windings}
 In \cite{CerulliIrelli11} and \cite{Haupt12}, one finds the condition that the coefficient quiver of $\Lambda_\C$ is a \emph{winding}, which means that the matrices $\Lambda_{\C,\alpha}$ are submonomial. Since the matrices stemming from $\Fun$-representations are by definition submonomial, the condition of a winding does not appear in our context.
\end{rem}

\subsection{A formula for the Euler characteristic}
\label{subsection: the Euler characteristic and the Tits space}

Let $Q$ be a finite quiver and $\Lambda$ an $\Fun$-representation of $Q$ with $E_v=\Lambda_v-\{0\}$ finite for all $v\in Q_0$. For a dimension vector ${\textbf{r}}$, we denote by ${\textbf{r}}^\ast=(\#\, E_v-r_v)_{v\in Q_0}$ its codimension vector.

The \emph{coefficient quiver of $\Lambda$} is the quiver $\Gamma$ with vertex set 
\[
 \Gamma_0 \ = \ \bigsqcup_{v\in Q_0} \ E_v
\]
and arrow set
\[
 \Gamma_1 \ = \ \big\{ (\alpha,i,j) \, \big| \, \alpha\in Q_1,\, j\in\Lambda_{s(\alpha)},\, i=\Lambda_\alpha(j)\neq0 \big\}.
\]
We say that $\Gamma$ is a \emph{tree} if its underlying graph is a tree and that $\Gamma$ is a \emph{primitive cycle} if its underlying graph is a single cycle and if $\#\,\Lambda_v\leq2$ for all $v\in Q_0$.

\begin{thm}\label{thm: Euler characteristics for tree modules}
 If the coefficient quiver of $\Lambda$ is a tree or a primitive cycle, then 
 \[
  \chi\big(\Gr_{\mathbf{r}}(\Lambda)(\C)\big) \ = \ \#\, \Gr_{\mathbf{r}}(\Lambda)^\Tits
 \]
 for any dimension vector ${\mathbf{r}}$.
\end{thm}

\begin{proof}
 This follows at once from \autoref{thm: rational points and Euler characteristics} and \cite[Cor.\ 5.18, Thm.\ 5.20]{Jun-SistkoEuler}.
\end{proof}

Recall that $\Lambda_\C$ is the complex representation of $Q$ with $\Lambda_{\C,v}=\C^{E_v}$ and whose matrices $\Lambda_{\C,\alpha}$ have coefficients
\[
 (\Lambda_{\C,\alpha})_{i,j} \ = \ \begin{cases}
                                    1 & \text{if } i=\Lambda_\alpha(j) \\
                                    0 & \text{otherwise}
                                   \end{cases}
\]
for $(i,j)\in E_{t(\alpha)}\times E_{s(\alpha)}$.

\begin{cor}\label{cor: Euler characteristic for quivers with at most one cycle}
 If $Q$ is a simply laced (extended) Dynkin quiver and $\Lambda_\C$ is irreducible, then 
 \[
  \chi\big(\Gr_{\mathbf{r}}(\Lambda)(\C)\big) \ = \ \#\, \Gr_{\mathbf{r}}(\Lambda)^\Tits
 \]
 for any dimension vector ${\mathbf{r}}$.
\end{cor}

\begin{proof}
 Let $\Gamma$ be the coefficient quiver of $\Lambda$, and therefore of $\Lambda_\C$ with respect to the canonical basis of $\Lambda_\C$. Since $\Lambda_\C$ is irreducible, $\Gamma$ is connected.
 
 If $Q$ is a tree, then also $\Gamma$ is a tree, thus the result follows from \autoref{thm: Euler characteristics for tree modules} in this case. If $Q$ is not a tree, then $Q$ must be of extended Dynkin type $\widetilde{A}_n$, i.e.\ the underlying graph of $Q$ is a single cycle. Then every irreducible and finite dimensional representation of $Q$ is either a string (which is a tree) or a primitive cycle. Thus also in this case, the result follows from \autoref{thm: Euler characteristics for tree modules}.
\end{proof}

\begin{ex}
\label{ex - last example}
Let $Q$ be a Dynkin quiver of type $D_4$ and $\Lambda$ the $\Fun$-representation of $Q$ with dimension vector $(3,2,2,2)$ and coefficient quiver $\Gamma$ as in \autoref{fig: coefficient quiver Gamma of Lambda}.

\begin{figure}[ht] 
  \begin{tikzpicture}[scale=1, x=2cm, y=0.5cm, font=\normalsize, vertices/.style={draw, fill=black, circle, inner sep=0pt},decoration={markings,mark=at position 0.5 with {\arrow{>}}}]
   \draw (.4,6.5) node (2) {$2$};
   \draw (.4,5.5) node (5) {$5$};
   \draw (1.2,4.5) node (7) {$7$};
   \draw (1.2,3.5) node (9) {$9$};
   \draw (2,6.0) node (1) {$1$};
   \draw (2,5.0) node (4) {$4$};
   \draw (2,4.0) node (6) {$6$};
   \draw (3.6,5.5) node (3) {$3$};
   \draw (3.6,4.5) node (8) {$8$};
   \draw[->] (1) -- (2);
   \draw[->] (1) -- (3);
   \draw[->] (4) -- (5);
   \draw[->] (4) -- (7);
   \draw[->] (6) -- (8);
   \draw[->] (6) -- (9);
   \draw (2,1.0)     node (v0) {$v_0$};
   \draw (.4,1.5)    node (v1) {$v_1$};
   \draw (3.6,1.0)   node (v2) {$v_2$};
   \draw (1.2,0.5)   node (v3) {$v_3$};
   \draw[->] (v0) -- node[above] {\footnotesize $\alpha_1$} (v1);
   \draw[->] (v0) -- node[above] {\footnotesize $\alpha_2$} (v2);
   \draw[->] (v0) -- node[pos=0.4,below] {\footnotesize $\alpha_3$} (v3);
   \draw (-.4,5.5) node (G) {$\Gamma$};
   \draw (-.4,1) node (Q) {$Q$};
   \draw[->] (G) -- (Q);
   \end{tikzpicture}
  \caption{Coefficient quiver $\Gamma$ of $\Lambda$, and the projection to $Q$} 
  \label{fig: coefficient quiver Gamma of Lambda}    
\end{figure}
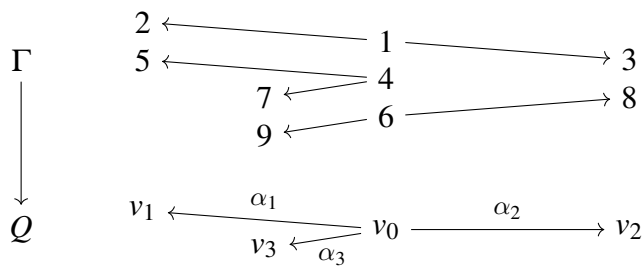 

Consider the rank vector $\br=(r_0,r_1,r_2,r_3)=(2,1,1,1)$. The Pl\"ucker embedding $\Gr_{\textbf{r}}(\Lambda) \to \P(\Lambda)$ is given by
$$
    \begin{tikzcd}[column sep = 23pt]
        \Gr_{\textbf{r}}(\Lambda) = \Proj\big(\Funpm[T_{14},T_{16},T_{46},T_2,T_5,T_3,T_8,T_7,T_9]\big)/\hspace{-4pt}/\mathpzc{qr}_{\Lambda,\br} \arrow[hook]{r} 
        & \P^2\times \P^1 \times \P^1 \times \P^1  
    \end{tikzcd}
$$
with multidegrees of coordinates given by 
\[
\begin{array}{ccccccc}
    (1,0,0,0) & = & \deg(T_{14}) & = & \deg(T_{16}) & = & \deg(T_{46}),\\
    (0,1,0,0) & = & \deg(T_2) & = & \deg(T_5), &&\\
    (0,0,1,0) & = & \deg(T_3) & = & \deg(T_8), &&\\
    (0,0,0,1) & = & \deg(T_7) & = & \deg(T_9) &&
\end{array}
\]
and $\mathpzc{qr}_{\Lambda,\br}$ generated by
\begin{eqnarray*} 
    T_2 T_{46} - T_5 T_{16} \quad \quad (\Lambda_{\alpha_1}),  \\
    T_3 T_{46} - T_6 T_{14} \quad  \quad (\Lambda_{\alpha_2}), \\
    T_7 T_{16} - T_9 T_{14} \quad \quad (\Lambda_{\alpha_3}).
\end{eqnarray*}

The quiver Grassmannian $\Gr_{\mathbf{r}^*}(\Lambda_\C) \simeq \Gr_{\mathbf{r}}((\Lambda^*)_\C)$ is a del Pezzo surface of degree $6$, see \cite[Example 5.1]{Lorscheid-Weist2019}. A del Pezzo surface is the blow-up of $\P^2$ in $3$ points, see \cite[Prop. 1.1 (d)]{Patrick_Corn}. Since the Euler characteristic of $\P^2$ is $3$, and it increases by $1$ for each exceptional curve of the blow-up, we conclude that $\chi\big(\Gr_{\textbf{r}^*}(\Lambda_\C)\big) = 6$.

On the other hand, consider the grading $\partial: \Gamma_0 \to \Z$ defined by $i \mapsto i$. If we extend this grading to $\partial: Q_1 \to \Z$ by $\alpha_i \mapsto i$, then 
$$
    \partial\big(\Lambda_{\alpha_i}(j)\big) - \partial(j) \ = \ \partial(\alpha_i)
$$ 
for each $\alpha_i \in Q_1$ and $j \in \Lambda_{s(\alpha_i)}$ with $\Lambda_{\alpha_i}(j) \neq 0$. This shows that $\partial$ is a nice grading of $\Lambda$. It follows from the proof of \autoref{thm: rational points and Euler characteristics} that
$$
    \Gr_{\textbf{r}}(\Lambda)^{\Tits} \cong  
    \begin{Bmatrix}
        \text{coordinate } (\Lambda, \K)\text{-matroids} 
        \\
        \text{of rank } \br
    \end{Bmatrix}.
$$

There are $6$ coordinate $(\Lambda,\K)$-matroids of rank $\br$, which are represented by the connected full subquivers of $Q''$ with exactly one vertex per colored area in \autoref{fig: coordinate A-Lambda-matroids}. Note that the quiver $Q''$ is obtained from the quiver $Q'$ from \autoref{ex: quiver matroids of type D4} by deleting the matroids with more than one basis, i.e. by deleting the matroids which are not coordinate matroids. The corresponding $\Fun$-subrepresentations are illustrated in \autoref{fig: Subrepresentations of Lambda}. 

This shows that 
$$
    \#\Gr_{\mathbf{r}}(\Lambda)^{\Tits} \ = \ 6 \ = \ \chi\big(\Gr_{\mathbf{r}^*}(\Lambda_{\C})\big),
$$
which re-establishes \autoref{thm: rational points and Euler characteristics} in this example.

\begin{figure}[H] 
  \[
   \begin{tikzpicture}[scale=1, font=\footnotesize, vertices/.style={draw, fill=black, circle, inner sep=0pt},decoration={markings,mark=at position 0.33 with {\arrow{>}}}]
    \draw[fill=green!20!white,draw=green!80!black,rounded corners=5pt] (0:2.8cm) -- (60:2.8cm) -- (120:2.8cm) -- (180:2.8cm) -- (240:2.8cm) -- (300:2.8cm) -- cycle;
    \node[color=green!40!black] at (0,0) {\normalsize $v_0$};
    \draw[fill=blue!10!white,draw=blue!40!white,rounded corners=7pt] (300:2.9cm) -- (0:2.9cm) -- (60:2.9cm) -- (34.5:4.4cm) -- (0:4.5cm) -- (325.5:4.4cm) -- cycle;
    \node[color=blue!80!white] at (0:3.6) {\normalsize $v_2$};
    \draw[fill=blue!10!white,draw=blue!40!white,rounded corners=7pt] (60:2.9cm) -- (120:2.9cm) -- (180:2.9cm) -- (154.5:4.4cm) -- (120:4.5cm) -- (85.5:4.4cm) -- cycle;
    \node[color=blue!80!white] at (120:3.6) {\normalsize $v_1$};
    \draw[fill=blue!10!white,draw=blue!40!white,rounded corners=7pt] (180:2.9cm) -- (240:2.9cm) -- (300:2.9cm) -- (274.5:4.4cm) -- (240:4.5cm) -- (205.5:4.4cm) -- cycle;
    \node[color=blue!80!white] at (240:3.6) {\normalsize $v_3$};
    \foreach \a in {1,...,5}{\draw (0+\a*120: 1.8cm) node [draw,circle,inner sep=1.5pt,fill=yellow] (a\a) {};
                             \draw (30+\a*120: 4cm) node [draw,circle,inner sep=1.5pt,fill=black] (d\a) {};
                             \draw (-30+\a*120: 4cm) node [draw,circle,inner sep=1.5pt,fill=black] (e\a) {};
                            }
    \foreach \a in {1,...,3} {\setcounter{tikz-counter}{\a};
                              \draw (a\a) edge [-,postaction={decorate}] (d\arabic{tikz-counter});
                              \draw (a\a) edge [-,postaction={decorate}] (e\arabic{tikz-counter});
                              \addtocounter{tikz-counter}{1};
                              \draw (a\a) edge [-,postaction={decorate}] (e\arabic{tikz-counter});
                              \addtocounter{tikz-counter}{1};
                              \draw (a\a) edge [-,postaction={decorate}] (d\arabic{tikz-counter});
                              }
    \draw (0:1.45cm) node {\tiny $\smalltrivector010$};
    \draw (120:1.4cm)  node {\tiny $\smalltrivector001$};
    \draw (240:1.4cm) node {\tiny $\smalltrivector100$};
    \draw (336:3.9cm)  node {$\smallvector10$};
    \draw ( 24:3.9cm)  node {$\smallvector01$};
    \draw ( 96:3.9cm)  node {$\smallvector01$};
    \draw (144:3.9cm)  node {$\smallvector10$};
    \draw (216:3.9cm)  node {$\smallvector01$};
    \draw (264:3.9cm)  node {$\smallvector10$};
   \end{tikzpicture}
  \] 
  \caption{The quiver $Q''$, which illustrates all coordinate $(\Lambda,\K)$-matroids of rank $\br$} 
 \label{fig: coordinate A-Lambda-matroids}     
\end{figure}
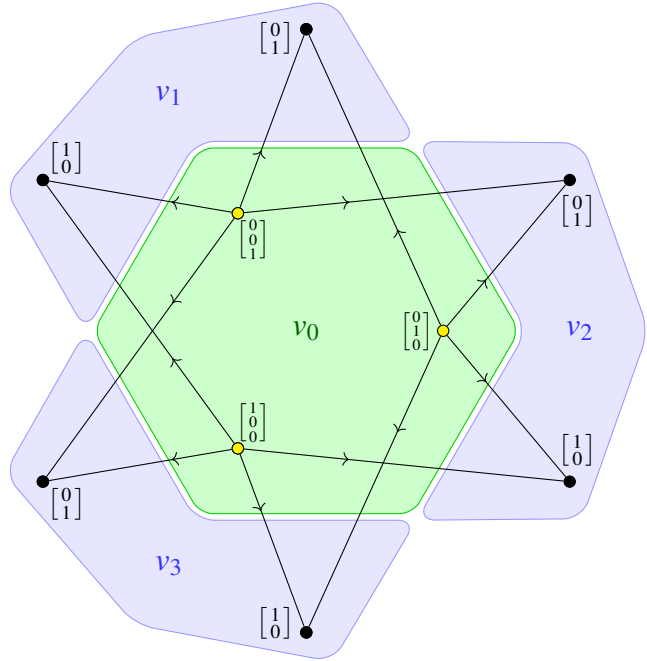

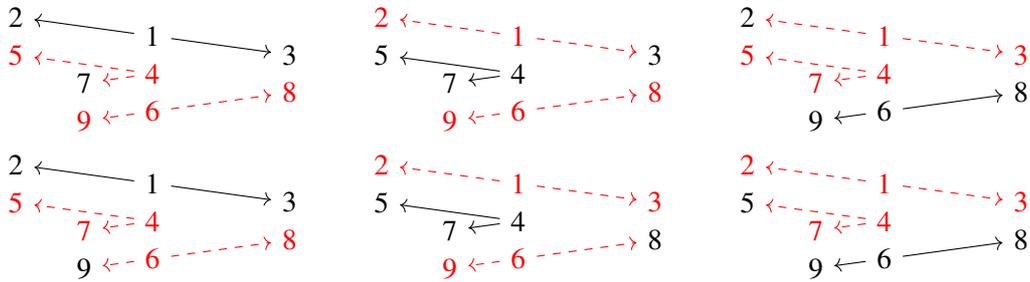
\begin{figure}[H]
    \centering
    \begin{tikzpicture}[scale=1, x=1.8cm, y=0.5cm, font=\small, vertices/.style={draw, fill=black, circle, inner sep=0pt},decoration={markings,mark=at position 0.5 with {\arrow{>}}}]
    \draw      (0,6.5)  node (2) {$2$};
    \draw[red] (0,5.5)  node (5) {$5$};
    \draw      (.5,4.75) node (7) {$7$};
    \draw[red] (.5,3.75) node (9) {$9$};
    \draw      (1,6.0)  node (1) {$1$};
    \draw[red] (1,5.0)  node (4) {$4$};
    \draw[red] (1,4.0)  node (6) {$6$};
    \draw      (2,5.5)  node (3) {$3$};
    \draw[red] (2,4.5)  node (8) {$8$};
    \draw[->] (1) -- (2);
    \draw[->] (1) -- (3);
    \draw[red, dashed, ->] (4) -- (5);
    \draw[red, dashed, ->] (4) -- (7);
    \draw[red, dashed, ->] (6) -- (8);
    \draw[red, dashed, ->] (6) -- (9);
    \end{tikzpicture}
    \hspace{15pt}
    \begin{tikzpicture}[scale=1, x=1.8cm, y=0.5cm, font=\small, vertices/.style={draw, fill=black, circle, inner sep=0pt},decoration={markings,mark=at position 0.5 with {\arrow{>}}}]
    \draw[red] (0,6.5) node (2) {$2$};
    \draw      (0,5.5) node (5) {$5$};
    \draw      (.5,4.75) node (7) {$7$};
    \draw[red] (.5,3.75) node (9) {$9$};
    \draw[red] (1,6.0) node (1) {$1$};
    \draw      (1,5.0) node (4) {$4$};
    \draw[red] (1,4.0) node (6) {$6$};
    \draw      (2,5.5) node (3) {$3$};
    \draw[red] (2,4.5) node (8) {$8$};
    \draw[red, dashed, ->] (1) -- (2);
    \draw[red, dashed, ->] (1) -- (3);
    \draw[->]              (4) -- (5);
    \draw[->]              (4) -- (7);
    \draw[red, dashed, ->] (6) -- (8);
    \draw[red, dashed, ->] (6) -- (9);
    \end{tikzpicture}
    \hspace{15pt}
    \begin{tikzpicture}[scale=1, x=1.8cm, y=0.5cm, font=\small, vertices/.style={draw, fill=black, circle, inner sep=0pt},decoration={markings,mark=at position 0.5 with {\arrow{>}}}]
    \draw      (0,6.5)  node (2) {$2$};
    \draw[red] (0,5.5)  node (5) {$5$};
    \draw[red] (.5,4.75) node (7) {$7$};
    \draw      (.5,3.75) node (9) {$9$};
    \draw[red] (1,6.0)  node (1) {$1$};
    \draw[red] (1,5.0)  node (4) {$4$};
    \draw      (1,4.0)  node (6) {$6$};
    \draw[red] (2,5.5) node (3) {$3$};
    \draw      (2,4.5) node (8) {$8$};
    \draw[red, dashed, ->] (1) -- (2);
    \draw[red, dashed, ->] (1) -- (3);
    \draw[red, dashed, ->] (4) -- (5);
    \draw[red, dashed, ->] (4) -- (7);
    \draw[->]              (6) -- (8);
    \draw[->]              (6) -- (9);
    \end{tikzpicture}
    \begin{tikzpicture}[scale=1, x=1.8cm, y=0.5cm, font=\small, vertices/.style={draw, fill=black, circle, inner sep=0pt},decoration={markings,mark=at position 0.5 with {\arrow{>}}}]
    \draw (0,6.5) node (2) {$2$};
    \draw[red] (0,5.5)  node (5) {$5$};
    \draw[red] (.5,4.75) node (7) {$7$};
    \draw      (.5,3.75) node (9) {$9$};
    \draw      (1,6.0)  node (1) {$1$};
    \draw[red] (1,5.0)  node (4) {$4$};
    \draw[red] (1,4.0)  node (6) {$6$};
    \draw      (2,5.5)  node (3) {$3$};
    \draw[red] (2,4.5)  node (8) {$8$};
    \draw[->] (1) -- (2);
    \draw[->] (1) -- (3);
    \draw[red, dashed, ->] (4) -- (5);
    \draw[red, dashed, ->] (4) -- (7);
    \draw[red, dashed, ->] (6) -- (8);
    \draw[red, dashed, ->] (6) -- (9);
    \end{tikzpicture}
    \hspace{15pt}
    \begin{tikzpicture}[scale=1, x=1.8cm, y=0.5cm, font=\small, vertices/.style={draw, fill=black, circle, inner sep=0pt},decoration={markings,mark=at position 0.5 with {\arrow{>}}}]
    \draw[red] (0,6.5)  node (2) {$2$};
    \draw      (0,5.5)  node (5) {$5$};
    \draw      (.5,4.75) node (7) {$7$};
    \draw[red] (.5,3.75) node (9) {$9$};
    \draw[red] (1,6.0)  node (1) {$1$};
    \draw      (1,5.0)  node (4) {$4$};
    \draw[red] (1,4.0)  node (6) {$6$};
    \draw[red] (2,5.5)  node (3) {$3$};
    \draw (2,4.5) node (8) {$8$};
    \draw[red, dashed, ->] (1) -- (2);
    \draw[red, dashed, ->] (1) -- (3);
    \draw[->]              (4) -- (5);
    \draw[->]              (4) -- (7);
    \draw[red, dashed, ->] (6) -- (8);
    \draw[red, dashed, ->] (6) -- (9);
    \end{tikzpicture}
    \hspace{15pt}
    \begin{tikzpicture}[scale=1, x=1.8cm, y=0.5cm, font=\small, vertices/.style={draw, fill=black, circle, inner sep=0pt},decoration={markings,mark=at position 0.5 with {\arrow{>}}}]
    \draw[red] (0,6.5)  node (2) {$2$};
    \draw      (0,5.5)  node (5) {$5$};
    \draw[red] (.5,4.75) node (7) {$7$};
    \draw      (.5,3.75) node (9) {$9$};
    \draw[red] (1,6.0)  node (1) {$1$};
    \draw[red] (1,5.0)  node (4) {$4$};
    \draw      (1,4.0)  node (6) {$6$};
    \draw[red] (2,5.5)  node (3) {$3$};
    \draw      (2,4.5)  node (8) {$8$};
    \draw[red, dashed, ->] (1) -- (2);
    \draw[red, dashed, ->] (1) -- (3);
    \draw[red, dashed, ->] (4) -- (5);
    \draw[red, dashed, ->] (4) -- (7);
    \draw[->] (6) -- (8);
    \draw[->] (6) -- (9);
    \end{tikzpicture}
    \caption{The $6$ subrepresentations of $\Lambda$ with dimension vector $\br^*=(1,1,1,1)$}
    \label{fig: Subrepresentations of Lambda}
\end{figure}
\end{ex}

\begin{small}
 \bibliographystyle{alpha}
 \bibliography{matroids2}
\end{small}

\end{document}